\documentclass[11pt]{article}
\usepackage[utf8]{inputenc}
\usepackage[T1]{fontenc}
\usepackage{macros}
\usepackage[english]{babel}
\usepackage{amsmath}
\usepackage{tikz-cd}
\usepackage{tikz}
\usepackage{amsthm}
\usepackage{amssymb}
\usepackage{stmaryrd}
\usepackage{amsfonts}
\usepackage{graphicx}
\usepackage{caption}
\usepackage{subcaption}

\usepackage{xspace}
\usepackage[all]{xy}
\usepackage{txfonts}
\usepackage{csquotes}
\usepackage{mathtools}
\usepackage{thmtools}
\usepackage{tdsfrmath}
\usepackage{xfrac}
\usepackage[shortlabels]{enumitem}
\usepackage{geometry}
\usepackage{arcs}
\usepackage{txfonts}

\newtheorem{thm}{Theorem}[section]
\newtheorem{lem}[thm]{Lemma}

\theoremstyle{remark}
\newtheorem{rem}{Remark} [section]

\theoremstyle{definition}
\newtheorem{defi}[thm]{Definition}

\newtheorem*{thm*}{Theorem}
\newtheorem*{cor*}{Corollary}
\newtheorem*{conj*}{Conjecture}
\newtheorem{proper}[thm]{Properties}
\newcommand{\Span}[1]{\operatorname{span}\set{#1}}

\DeclareFontFamily{OMX}{yhex}{}
\DeclareFontShape{OMX}{yhex}{m}{n}{<->yhcmex10}{}
\DeclareSymbolFont{yhlargesymbols}{OMX}{yhex}{m}{n}
\DeclareMathAccent{\wip}{\mathord}{yhlargesymbols}{"F3}

\title{Strip deformations of decorated hyperbolic polygons}
\author{Pallavi Panda}
\date{}

\begin{document}
	
	\maketitle
	\paragraph{Abstract.}In this paper we study the hyperbolic and parabolic strip deformations of ideal (possibly once-punctured) hyperbolic polygons whose vertices are decorated with horoballs. We prove that the interiors of their arc complexes parametrise the open convex set of all uniformly lengthening infinitesimal deformations of the decorated hyperbolic metrics on these surfaces, motivated by the work of Danciger-Guéritaud-Kassel. 
	\tableofcontents

	\section{Introduction}
A crowned hyperbolic surface is a complete non-compact hyperbolic surface with polygonal boundary where the vertices (called spikes) are projections of points on the boundary of the hyperbolic plane $\HP$. Topologically, this is an orientable surface with finitely many points removed from its boundary. The smallest crowned hyperbolic surfaces are an ideal polygon $\ip n$ ($n\geq 3$) and an ideal once-punctured polygon $\punc n$ ($n\geq 1$). 
	 
	 The arc complex $\ac{S_{g,n}}$ of such a surface, first defined by Harer \cite{harer}, is a pure flag simplicial complex generated by isotopy classes of embedded arcs with endpoints on the spikes. The arc complexes of most of the surfaces are locally non-compact. Harer showed that a specific open dense subset of the arc complex, referred to as the \emph{pruned arc complex} in this paper, is an open ball of dimension one less than that of the deformation space of the surface. Penner \cite{penner} proved that the arc complexes of $\ip n$ and $\punc n$ are $PL$-homeomorphic to spheres of dimension $n-4$ and $n-2$, respectively. In \cite{sphereconj}, he gave a complete list of surfaces for which the quotient of the arc complex by the pure mapping class group is a sphere or a $PL$-manifold. 
	 The arc complex of a crowned surface is simplicially isomorphic to that of a surface with marked points on its boundary, where the arcs end on the marked points. 
	  In \cite{cluster}, Fomin-Shapiro-Thurston proved that the latter is a subcomplex of the cluster complex associated to the cluster algebra of a surface.  The relationship to cluster algebras was heavily motivated by Penner's Decorated Teichmuller Theory \cite{Pennerpunc} \cite{Pennerbordered}. In these papers, he gave \emph{lambda length} coordinates (see Section \ref{lambda}) to the decorated Teichmuller space of a crowned surface whose spikes are decorated with horoballs. Furthermore, using the arc complex, he gave a cell-decomposition of the decorated Teichmuller space of the surface. 
	  
	   An \emph{admissible deformation} of such a surface is an infinitesimal deformation that uniformly lengthens all closed geodesics.  Goldman-Labourie-Margulis, in \cite{glm}, proved that the subspace of admissible deformations forms an open convex cone, called the \emph{admissible cone}. On the other hand, the pruned arc complex of these surfaces once again form an open ball of dimension one less than that of the Teichmuller space, which is obtained by reinterpreting Harer's result. Danciger-Guéritaud-Kassel \cite{dgk} showed that the pruned arc complex parametrises the positively projectivised admissible cone. The authors uniquely realised (Theorem \ref{cpt} of \cite{dgk}) an admissible deformation of the surface by performing hyperbolic strip deformations along positively weighted embedded arcs with endpoints on the boundary, corresponding to a point in the pruned arc complex. Hyperbolic strip deformations were first introduced by Thurston in \cite{thurston}. A hyperbolic strip is defined to be the region in $\HP$ bounded by two geodesics whose closures are disjoint.  A hyperbolic strip deformation is the process of cutting the surface along an embedded arc and gluing in a strip, without shearing. 
	  
	 Motivated by the above works, in this paper we study the arc complexes of a decorated ideal polygon $\dep n$ ($n\geq 3$) and a decorated once-punctured ideal polygon $\depu n$ ($n\geq 1$). These are generated by finite arcs whose endpoints lie on the boundary, and infinite arcs with one end at a decorated vertex and the other endpoint on the boundary. In both of these cases the arc complexes are finite. The horoball connections are simply the edges and diagonals. The admissible cone the set of all infinitesimal deformations that lengthens all edges and diagonals of the polygon. The main results of this paper are to show that the pruned arc complexes parametrise the admissible cone of decorated (possibly once-punctured) polygons via the projectivised strip map.
	 
	 \begin{thm*}
	 	Let $\Pi$ be a decorated $n$-gon $\dep n$ ($n\geq 3$)   (resp. a decorated once-punctured $n$-gon $\depu n$ ($n\geq 1$)) with a decorated metric $m$    in the deformation space $\tei {\Pi}$. Fix a choice of strip template (see Subsection \ref{template}). Then the projectivised infinitesimal strip map $\mathbb{P}f: \ac \Pi\longrightarrow \ptan \Pi$, when restricted to the pruned arc complex $\sac {\Pi}$, is a homeomorphism onto its image $\mathbb{P}^+(\adm m)$ $\simeq \R^{2n-4}\, (\text{resp. } \R^{2n-2}$), where $\adm m$ is the admissible cone.
	 \end{thm*}
	 
	 This is obtained by following the proof structure in \cite{dgk}, using both hyperbolic strip deformations along finite arcs and parabolic strip deformations along infinite arcs. 
	 
	 Finally, we also give a version of the above theorem for the undecorated ideal polygons and once-punctured ideal polygons.  In these cases the arcs are finite with endpoints on non-consecutive edges of the polygon. We show (Theorems \ref{Mainideal} and \ref{Mainpunc}) that the arc complex parametrises the entire positively projectivised deformation space in these cases.
	 \begin{thm*}
	 	Let $\Pi$ be an ideal polygon $\ip n$ ($n\geq4$) (resp. a once punctured polygon $\punc n$ ($n\geq2$)) with a metric $m\in \tei \Pi$. Fix a choice of strip template. Then, the projectivised infinitesimal strip map $\mathbb{P}f: \ac \Pi\longrightarrow \ptan \Pi$ $\simeq \s{n-4}\,( \text{resp. } \s{n-2})$ is a homeomorphism.
	 \end{thm*}		
	 In a future paper we give the version of the above results for "bigger" decorated crowned hyperbolic surfaces. 
	 
	 The paper is structured into sections in the following way:
	 Section  \ref{prelim} recapitulates the necessary vocabulary from hyperbolic, Lorentzian and projective geometry, and introduces every type of surface mentioned above along with their deformation spaces and admissible cones. In Section \ref{arc}, we discuss the arcs and the arc complexes of the different types of surfaces and study their topology. Section \ref{Sd} gives the definitions of the various strip deformations along different types arcs and some estimations that will be required in the proofs. We also give a recap of the main steps of the proof of their main result in \cite{dgk}. Finally, Section \ref{proof} contains the proofs of our main theorems.

\paragraph{Acknowledgements.}  
This work was done during my PhD at Université de Lille from 2017-2020 funded by the AMX scholarship of Ecole Polytechnique. I would like to thank my thesis advisor François Guéritaud for his valuable guidance and extraordinary patience. I am grateful to my thesis referees Hugo Parlier and Virginie Charette for their helpful remarks. I am also grateful to Université du Luxembourg for funding my postdoctoral research (Luxembourg National Research Fund OPEN grant O19/13865598). Finally, I would like to thank Katie Vokes, Viola Giovannini and Thibaut Benjamin for their constant support and faith in my work.

	\section{Preliminaries}\label{prelim}
	In this section we recall the necessary vocabulary and notions and also prove some results in hyperbolic geometry that will be used in the rest of the paper.
	\subsection{Minkowski space $\Min$}
	\begin{defi}
		The \emph{Minkowski space} $\R^{2,1}$ is the affine space $\R^3$ endowed with the quadratic form $\norm \cdot$ of signature $(2,1)$: \[\text{for } v=(x_1,x_2,x_3)\in \R^3,\quad \norm v ^2=x_1^2+x_2^2-x_3^2.\] 
		
	\end{defi}
	There is the following classification of points in the Minkowski space: a non-zero vector $\mathbf{v}\in \Min$ is said to be
	\begin{itemize}
		\item \emph{space-like} if and only if $\norm {\mathbf v}^2>0$,
		\item \emph{light-like} if and only if $\norm {\mathbf v}^2=0$,
		\item \emph{time-like} if and only if $\norm {\mathbf v}^2<0$.
	\end{itemize}
	A vector $\mathbf v$ is said to be \emph{causal} if it is time-like or light-like. A causal vector $\mathbf v=(x,y,z)$ is called \emph{positive} (resp.\ \emph{negative}) if $z>0$ (resp.\ $z<0$). Note that by definition of the norm, every causal vector is either positive or negative. 
	The set of all light-like points forms the \emph{light-cone}, denoted by $$L:=\{\mathbf v=(x,y,z)\in \Min \mid x^2+y^2-z^2=0\}.$$
	The \emph{positive} (resp.\ \emph{negative}) cone is defined as the set of all positive (resp.\ \emph{negative}) light-like vectors.
	\paragraph{Subspaces.} A vector subspace $W$ of $\Min$ is said to be 
	\begin{itemize}
		\item \emph{space-like} if $W\cap C=\{(0,0,0)\}$,
		\item \emph{light-like} if $W\cap C=\Span {\mathbf v}$ where $\mathbf v$ is light-like,
		\item \emph{time-like} if $W$ contains at least one time-like vector.
	\end{itemize}
	A subspace of dimension one is going to be called a line and a subspace of dimension two a plane. The adjective "affine" will be added before the words "line" and "plane" when we are referring to some affine subspace of the corresponding dimension.
	\paragraph{Duals.} Given a vector $\mathbf {v}\in\Min$, its dual with respect to the bilinear form of $\Min$ is denoted $\bf v^{\perp}$. For a light-like vector $\bf v$, the dual is given by the light-like hyperplane tangent to $C$ along $\Span{\bf v}$. For a space-like vector $\bf v$, the dual is given by the time-like plane that intersects $C$ along two light-like lines, respectively generated by two light-like vectors $\bf v_1$ and $\bf v_2$ such that $\Span {\bf v}=\bf v_1^{\perp}\cap \bf v_2^{\perp}$. Finally, the dual of a time-like vector $\bf v$ is given by a space-like plane. One way to construct it is to take two time-like planes $W_1,W_2$ passing through $\bf v$. Then the space $\bf v^\perp$ is the vectorial plane containing the space-like lines $W_1^\perp$ and $W_2^\perp$.
	
	\subsection{The different models of the hyperbolic 2-space}
	In this section we recall some vocabulary and notions related to the different models for the hyperbolic plane, that will be used in the calculations and proofs later.
	\paragraph{Hyperboloid model.} The classical hyperbolic space of dimension two $\HP$ can be identified with the upper sheet of the two-sheeted hyperboloid  $\{\mathbf v=(x,y,z)\in\Min \mid \norm{\mathbf v}^2=-1\},$ along with the restriction of the bilinear form. It is the unique (up to isometry) complete simply-connected Riemannian 2-manifold of constant curvature equal to -1. Its isometry group is isomorphic to $\so$ and the identity component $\sop$ of this group forms the group of its orientation-preserving isometries; they preserve each of the two sheets of the hyperboloid individually. If the hyperbolic distance between two points $\mathbf u,\mathbf v\in \HP$ is denoted by $d_{\HP}(\mathbf u,\mathbf v)$, then $\cosh d_{\HP}(\mathbf u,\mathbf v)= -\bil {\mathbf u }{\mathbf v}$. The geodesics of this model are given by the intersections of time-like hyperplanes with $\HP$.
	\paragraph{Klein's disk model.} This model is the projectivisation of the hyperboloid model.
	
	Let $\mathbb{P}:\Min\smallsetminus \{\mathbf 0\} \longrightarrow \pp$ be the projectivisation of the Minkowski space. The projective plane $\pp$ can be considered as the set $A \cup \mathbb{RP}^1$, where $A:=\set{(x,y,1)\,|\,x,y\in \R}$ is an affine chart and the one-dimensional projective space represents the line at infinity, denoted by $\pli$. The $\p$-image of a point $\mathbf v\in \Min$ is denoted by $[\mathbf v]$. A line in $A$, denoted by $\pl$, is defined as $A\cap V$ where $V$ is a two-dimensional vector subspace of $\Min$, not parallel to $A$. 
	
	In the affine chart $A$, the light cone is mapped to the unit circle and the hyperboloid is embedded onto its interior.  This is the Klein model of the hyperbolic plane; its boundary a circle. This model is non-conformal. The geodesics are given by open finite Euclidean straight line segments, denoted by $l$, lying inside $\HP$, such that the endpoints of the closed segment $\pls$ lie on $\HPb$. 
	The distance metric is given by the Hilbert metric 
	$d_{\HP}(w_1,w_2)=\frac{1}{2}\log [p,w_1;w_2,q]$, where $p$ and $q$ are the endpoints of $\pls$, $l$ being the unique hyperbolic geodesic passing through $w_1,w_2\in \HP$, and the cross-ratio $[a,b;c,d]$ is defined as $\frac{(c-a)(d-b)}{(b-a)(d-c)}$. The group of orientation-preserving isometries is identified with $\mathrm{PSU}(1,1)$. A point $p$ is called \emph{real} (\emph{ideal}, \emph{hyperideal}) if $p\in \HP$ (resp. $p\in \HPb$, $p\in \pl\cup A\backslash \cHP$).
	
	The dual of $\pli$ is the point $(0,0,1)$ in $A$.
	The dual of any other projective line $\pl=A\cap V$ is given by the point $A\cap V^{\perp}$. The dual $p^{\perp}$ of a point $p\in \pp$ is the projective line $A\cap {\Span p}^{\perp}$. If $l$ is a hyperbolic geodesic, then $l^{\perp}$ is defined to be $\pl^{\perp}$; it is given by the intersection point in $\pp$ of the two tangents to $\HPb$ at the endpoints of $\pls$.
	
	\emph{Notation:} We shall use the symbol $\cdot^{\perp}$ for referring to the duals of both linear subspaces as well as their projectivisations.

	\paragraph{Upper Half-plane Model.} The subset $\{z=x+iy\in \C \,| \,y>0\}$ of the complex plane is the upper half-space model of the hyperbolic space of dimension 2.
	The geodesics are given by semi-circles whose centres lie on $\R$ or straight lines that are perpendicular to $\R$. We shall call the former as \emph{horizontal} and the latter as $vertical$ geodesics. The boundary at infinity $\HPb$ is given by $\R \cup \{\infty\}$. The orientation-preserving isometry group is given by $\psl$ that acts by Möbius transformations on $\HP$. 
	
	\emph{Notation:} We shall denote by $G$ the isomorphic groups $\mathrm{Isom}(\HP), \so,\pgl$ and by $\lalg$ the Lie algebra of $G$.

	\subsection{Horoballs and decorated geodesics}\label{lambda}
	
	\begin{figure}
		\centering
		\frame{\includegraphics[height=4cm]{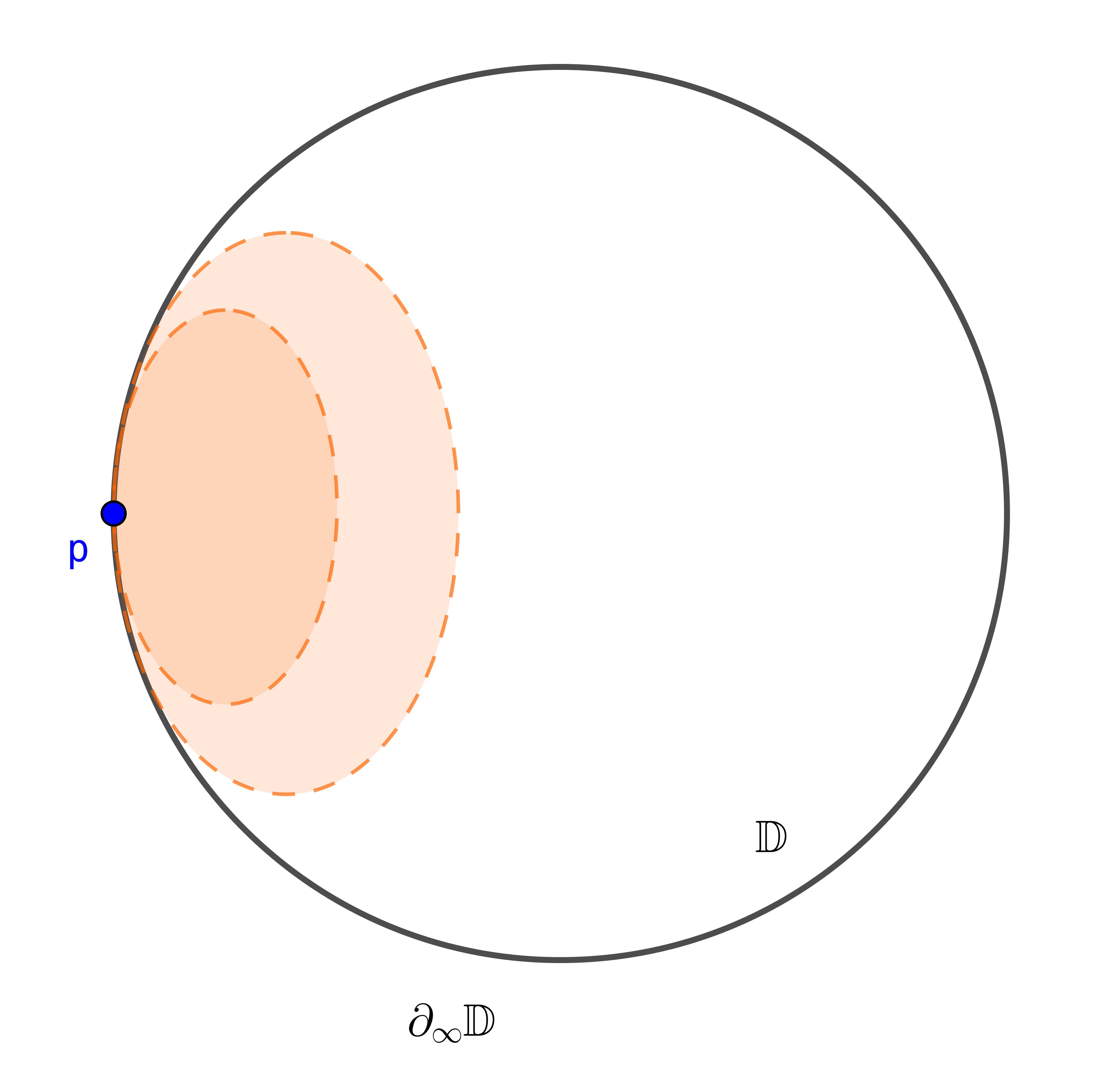}
			\includegraphics[height=4cm]{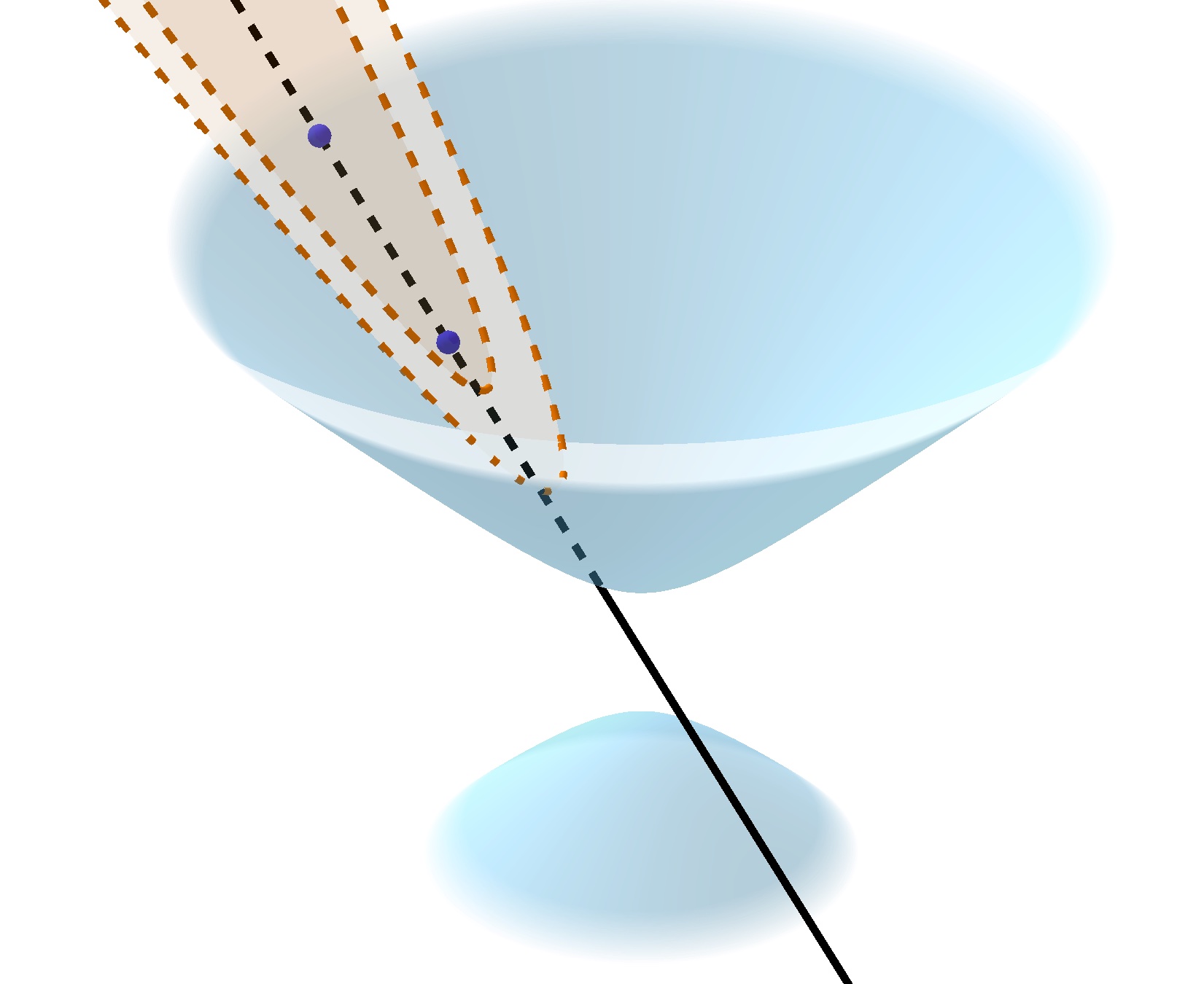}}
		\caption{Concentric horoballs}
		\label{horoball}
	\end{figure}
	An \emph{open horoball} $h(p)\subset \HP\subset\pp$ based at $p\in\HPb$ is the projective image of $$H (\mathbf v)=\{\mathbf w\in \HP\mid \bil{\mathbf w}{\mathbf{v}}>-1 \}$$x where $\mathbf v$ is a future-pointing light-like point in $\pinv p$. 
	If $k\geq k'>0$, then $H(k\mathbf v_0)\subset H(k'\mathbf v_0)$. See Fig. \ref{horoball}.
	
	The boundary of $h(p)$ is called a \emph{horocycle}. It is the projective image of the set $$h(\mathbf v):=\{\mathbf w\in \HP \mid \ang{\mathbf w,\mathbf v}=-1 \}.$$
	In the projective disk model, it is a Euclidean ellipse inside $\HP$, tangent to $\HPb$ at $p$. In the upper half-plane model, horocycles are either Euclidean circles tangent to a point on the real line or horizontal lines which are horocycles based at $\infty$. In the Poincaré disk model, a horocycle is an Euclidean circle tangent to $\HPb$ at $[p]$. A geodesic, one of whose endpoints is the centre of a horocycle, intersects the horocycles perpendicularly. Note that any horoball is completely determined by a future-pointing light-like vector in $\Min$ and vice-versa. From now onwards, we shall use either of the notations introduced above to denote a horoball. Finally, the set of all horoballs of $\HP$ forms an open cone (the positive light cone). 
	
	Given an ideal point $p\in \HPb$, a \emph{decoration} of $p$ is the specification of an open horoball centred at $p$. A geodesic, whose endpoints are decorated, is called a \emph{horoball connection}.
	The following definition is due to Penner \cite{penner}.
	
	\begin{figure}
		\centering
		\frame{\includegraphics[width=8cm]{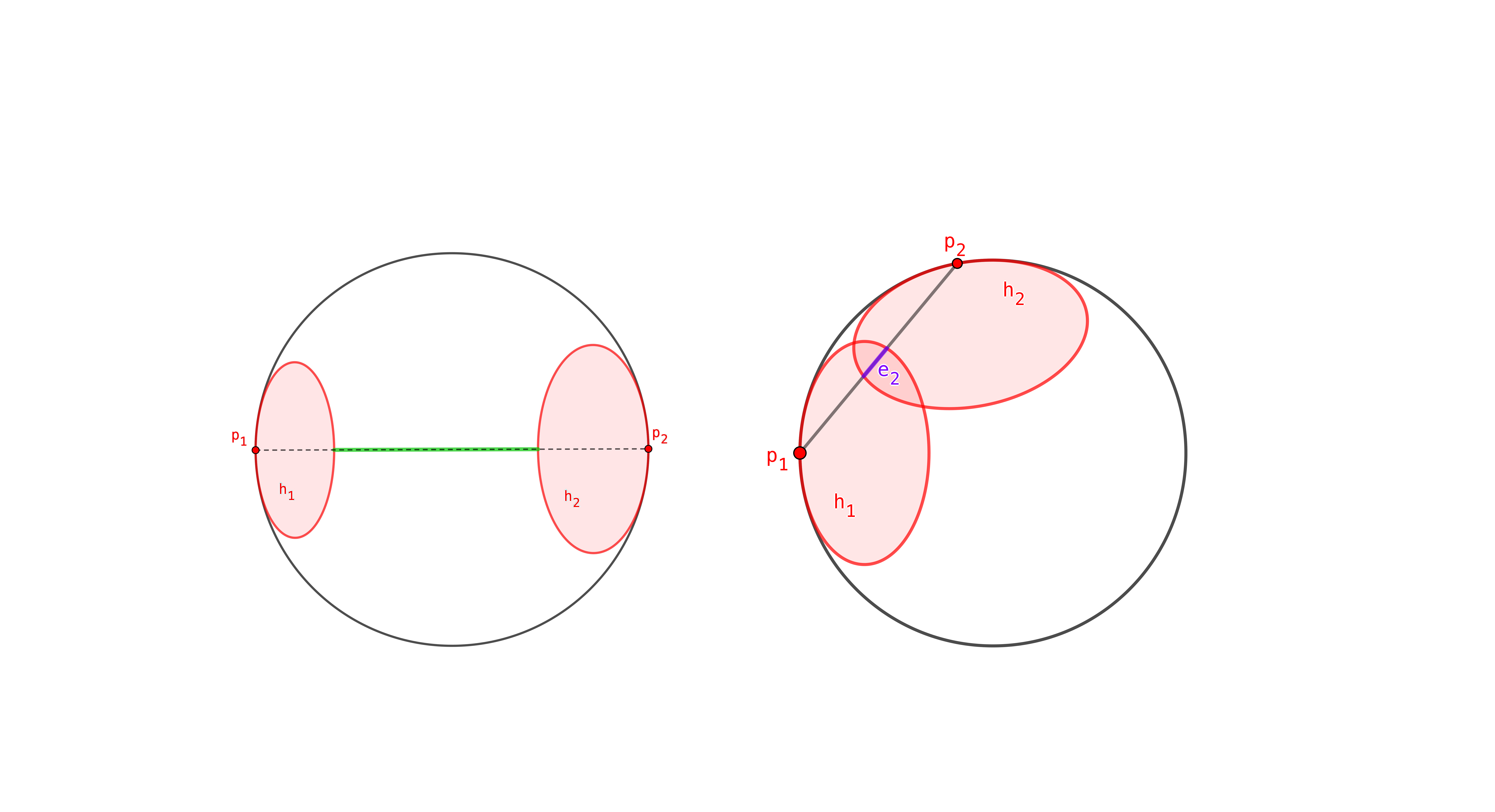}}
		\caption{Length of horoball connections}
	\end{figure}
	The \emph{length} of a horoball connection joining two horoballs based at $\mathbf v_1,\mathbf v_2$ is given by
	\begin{equation*}
		l:= \ln(-\frac{  \ang{\mathbf v_1,\mathbf v_2}}{2}).
	\end{equation*}
	It is the signed length of the geodesic segment intercepted by the two horoballs at the endpoints. In particular, if the horoballs are not disjoint, then the length of the horoball connection is negative. 
	
	Penner defined the \emph{lambda length} of two horocycles $h(\mathbf v_1),h(\mathbf v_2)$ to be \[ \lambda (h(\mathbf v_1),h(\mathbf v_2)):=\sqrt{-\ang{\mathbf v_1,\mathbf v_2}}=\sqrt{2 \e ^l}.  \]
	\subsection{Killing Vector Fields of $\HP$}
	The Minkowski space $\Min$ is isomorphic to $(\lalg, \kappa)$ where $\lalg$ is the Lie algebra of $G:=\pgl$ and $\kappa$ is its Killing form, via the following map:
	\[\mathbf v=(x,y,z)\mapsto V=\begin{pmatrix}
		y & x+z\\
		x-z & -y
	\end{pmatrix} .\]
	
	The Lie algebra $\lalg$ is also isomorphic to the set $\mathscr{X}$ of all Killing vector fields  of $\HP$:
	\[V\mapsto \bra{ 
		\begin{array}{ccc}
			X_v:&\HP\longrightarrow &\mathrm T\HP\\
			& p\mapsto &\frac{d}{dt} (e^{tV}\cdot p)|_{t=0}
	\end{array}}\]
	Next, one can identify $\Min$ with $\mathscr{X}$ via the map:
	\[ \bf v\mapsto \bra{
		\begin{array}{ccc}
			X_v:&\HP\longrightarrow &\mathrm T\HP\\
			&p\mapsto &\bf {v} \mcp p
	\end{array}}
	\] where $\mcp$ is the Minkowski cross product:
	\[(x_1,y_1,z_1)\mcp(x_2,y_2,z_2):=(-y_1z_2+z_1y_2,-z_1x_2+x_1z_2, x_1y_2-y_1x_2).\]
	
	Finally, in the upper half space model, one can identify $\mathscr{X}$ with the real vector space $\R_2[z]$ of polynomials of degree at most 2:
	\[
	P(\cdot) \mapsto \bra{z\mapsto P(z)\frac{\partial}{\partial z}}
	\]
	The discriminant of a polynomial in $\R_2[z]$ corresponds to the quadratic form $\norm \cdot$ in $\Min$. So the nature of the roots of a polynomial determines the type of the Killing vector field. In particular, when
	\begin{itemize}
		\item $P(z)=1$, the corresponding Killing vector field is parabolic, fixing $\infty$;
		\item $P(z)=z$, the corresponding Killing vector field is hyperbolic, fixing $0,\infty$;
		\item $P(z)=z^2$, the corresponding Killing vector field is parabolic, fixing $0$.
	\end{itemize}
	
	\begin{proper}
		
		Using these isomorphisms, we have that 
		\begin{itemize}
			\item A spacelike vector $\mathbf v$ corresponds, in $\mathscr{X}$, to an infinitesimal hyperbolic translation whose axis is given by $\mathbf v^{\perp}\cap \HP$. If $\mathbf v^+$ and $\mathbf v^-$ are respectively its attracting and repelling fixed points in $C^+$, then $(\bf v^-, v,v^+)$ are positively oriented in $\Min$.
			\item A lightlike vector $\mathbf v$ corresponds, in $\mathscr{X}$, to an infinitesimal parabolic element that fixes the light-like line $\Span {\mathbf v}$.
			\item A timelike vector $\mathbf v$ corresponds, in $\mathscr{X}$, to an infinitesimal rotation of $\HP$ that fixes the point $\frac{\mathbf v}{\sqrt {-\norm{\mathbf v}}}$ in $\HP$.
		\end{itemize} 
	\end{proper}
	\vspace{0.4cm}
	\begin{proper}\label{tang}~
		\begin{enumerate}
			\item Given a light-like vector $\mathbf v\in\Min$, the set of all Killing vector fields that fix $\Span {\mathbf v}$ is given by its dual $\mathbf v^{\perp}$. In $\pp$, the set of projectivised Killing vector fields that fix $[\mathbf v] \in \HPb$ is given by the tangent line at $[\mathbf v]$.
			\item The set of all Killing vector fields that fix a given ideal point $p\in\HPb$ and a horocycle in $\HP$ with centre at $p$ is given by $\Span {\mathbf v}$, where $\mathbf v\in \mathbb{P}^{-1}(p)$ in $\Min$.
			\item The set of all Killing vector fields that fix a given hyperbolic geodesic $l$ in $\HP$ is given by $\mathbb{P}^{-1}(l^{\perp})$.
		\end{enumerate}
		
	\end{proper}
	\subsection{Calculations in different models of hyperbolic geometry}
	\begin{defi}
		Let $\ol{l}$ be a projective line segment contained in $\cHP$ with endpoints, denoted by $A,B$. Then the two projective triangles formed by $\du{A},\du{B}$ and $\oarr{l}$, with their disjoint interiors intersecting $\HP$, are said to be \emph{based} at $\ol{l}$.
	\end{defi}

\begin{figure}[h!]
	\frame{	\includegraphics[width=15cm]{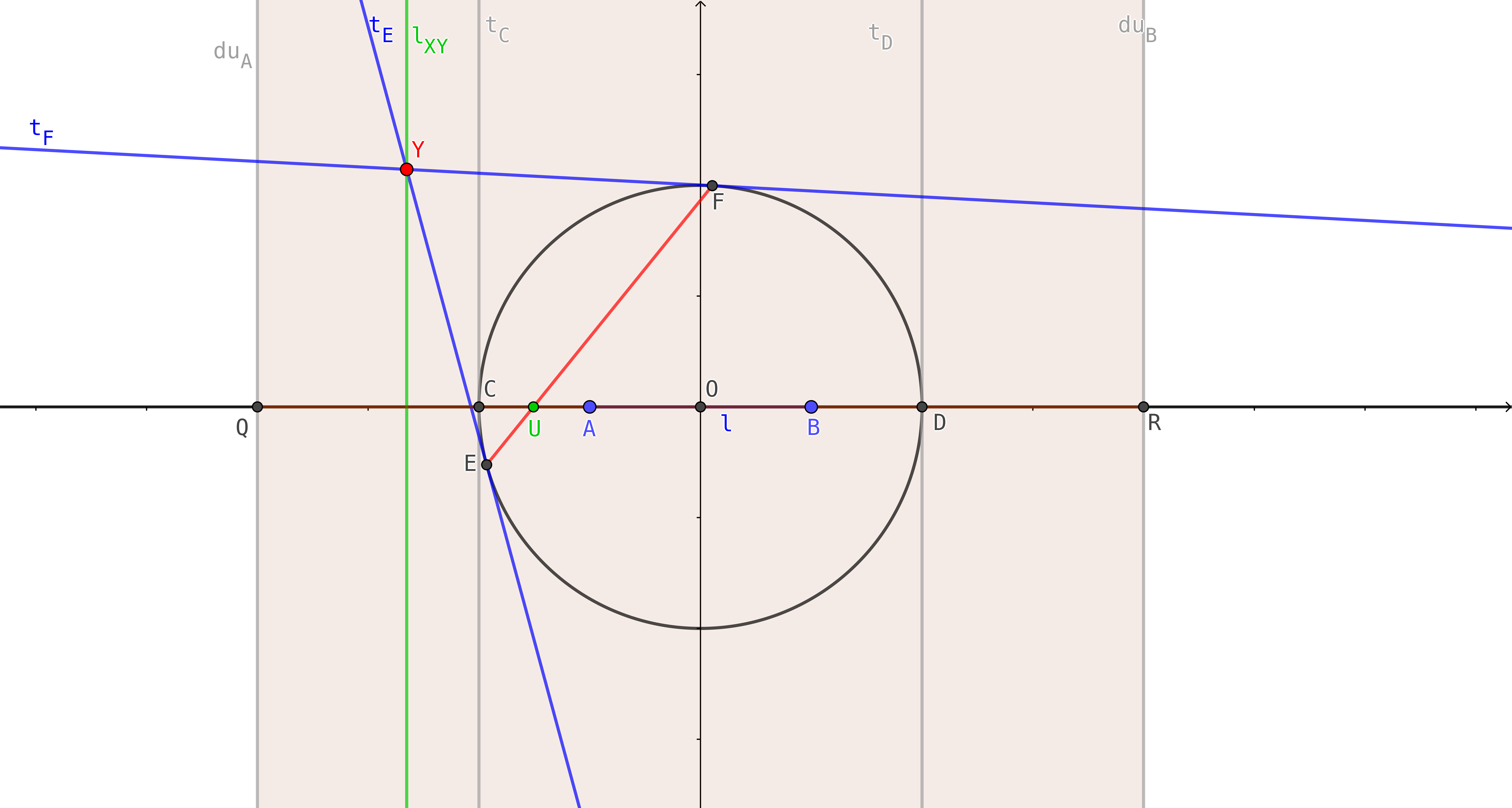}}
	\centering
	\caption{Properties \ref{line}}
	\label{assochord}
\end{figure}
\begin{proper}\label{line}
	Let $\ol{l}$ be a projective line segment contained in $\cHP$. Then, any projective line $\ol{l'}$ that intersects $\HP$, is disjoint from $\ol{l}$ if and only if its dual $\du{l'}$ is a space-like point contained in the interior of the bigon equal to the union of the two triangles based at $\ol{l}$.
\end{proper}

\begin{proof}
	Let the endpoints of $\ol l$ be denoted by $A,B$.
	There are three possibilities for $l$ — either a geodesic segment (both $A,B \in \HP$) or $l$ is a geodesic ($A,B\in \HPb$ ), or a geodesic ray ($A$ or $B$ on $\HPb$, the other inside $\HP$). It is enough prove the lemma for first case, the two others being limit cases of the first.
	
	Let $\oarr {l'}$ be another projective line that intersects $\HP$. Let $X,Y$ be the respective dual points $\du {\ol l}, \du {\ol {l'}}$. Since both the line segments intersect $\HP$, neither $X$ nor $Y$ can line inside $\cHP$. Then, $\ol l$ and $\ol {l'}$ intersect each other at a point $U\in \HP$ if and only if $U=\du {\ol {XY}}$. 
	
	Using a hyperbolic isometry, we can assume that both the points $A,B$ lie on the horizontal axis, on either side of the origin. Then the line segment $\ol l$ is given by the closed interval $[a,b]\times {0}$, where $A=(a,0), B=(b,0)$, with $a<0<b$. Owing to this choice of $A,B$, the duals $\du A, \du B$ are vertical lines passing through $(\frac{1}{a},0),(\frac{1}{b},0)$, respectively, with their point of intersection $X$ lying on the line at infinity $\pli$. The union $\Delta$ of the two projective triangles based at $\ol l$ is given by the open vertical strip bounded by these two verticals, that contains $\HP$. Now the line segment $\ol {XY}$ is a vertical line passing through $(y,0)$, where $y$ is the horizontal coordinate of $Y$. The coordinates of the dual point $\du{\ol{XY}}$ is given by $(\frac{1}{y},0)$. Then $\ol l$ and $\ol {l'}$ intersect each other if and only if $$a\leq \frac{1}{y} \leq b\Leftrightarrow \frac{1}{a}\leq y \text{ or } y\geq \frac{1}{b}.$$ In other words, the line $\oarr{l'}$ is disjoint from the segment $\ol l$ if and only if $Y$ is a space-like point inside $\Delta$.
\end{proof}
\begin{lem}\label{centre}
	Let $\ga_1=(a_1,b_1)$ and $\ga_2=(a_2,b_2)$ be two geodesics in the upper half-plane model of $\HP$ where $a_1,b_1,a_2,b_2$ are real numbers satisfying \[a_1<b_1<a_2<b_2.\] Let $\ga$ be the unique common perpendicular to $\ga_1$ and $\ga_2$. If $x$ denotes the centre of the semi-circle containing $\ga$, then \[x=\frac{a_2b_2-a_1b_1}{a_2+b_2-a_1-b_1}.\]
\end{lem}
\begin{proof}
	Let $g=\begin{pmatrix}
		p&q\\
		r&s
	\end{pmatrix} \in \pgl$ be the inversion with respect to the semi-circle $\ga$. Then, by definition of inversion, we have 
	\begin{align}
		x\mapsto \infty &\Rightarrow rx+s=0, \label{xinfi}\\
		a\mapsto b& \Rightarrow pa_1+q=a_1b_1r+b_1s,\label{a2b}\\
		b\mapsto a & \Rightarrow pb_1+q=a_1b_1r+a_1s,\label{b2a}\\
		c\mapsto d & \Rightarrow pa_2+q=a_2b_2r+b_2s.\label{c2d}
	\end{align}, where "$\mapsto$" refers to the action of $g$.
	From eq.\eqref{a2b} and \eqref{b2a}, we get that $p=-s$ and from the eqs.\eqref{xinfi}, \eqref{a2b} and \eqref{c2d}, we get that \[x=\frac{-s}{r}=\frac{a_2b_2-a_1b_1}{a_2+b_2-a_1-b_1}.\]
\end{proof}
\begin{lem}
	Let $\ga_1,\ga_2,\ga_3$ be three pairwise disjoint semi-circular, possibly asymptotic geodesics in $\HP$ such that none of them separates the remaining two geodesics from each other. For $i\in \Z_3$, let $\be_i$ be the common perpendicular to $\ga_{i-1}$ and $\ga_{i+1}$, whenever possible. Let $x_i$ be the centre of $\ga_i$ for $i=1,2,3$. Let $y_i$ be the centre of $\be_i$ or the common endpoint of $\ga_{i-1},\ga_{i+1}$ for $i=1,2,3$. Then the following equation holds:
	\begin{align}
		\frac{x_1-x_2}{x_2-x_3}&=\frac{y_1-y_2}{y_2-y_3}\label{ratio}\\
		\ie \bra{\infty,x_1,x_2,x_3}&=\bra{\infty,y_1,y_2,y_3}.
	\end{align}
\end{lem}
\begin{figure}
	\frame{\includegraphics[width=\linewidth]{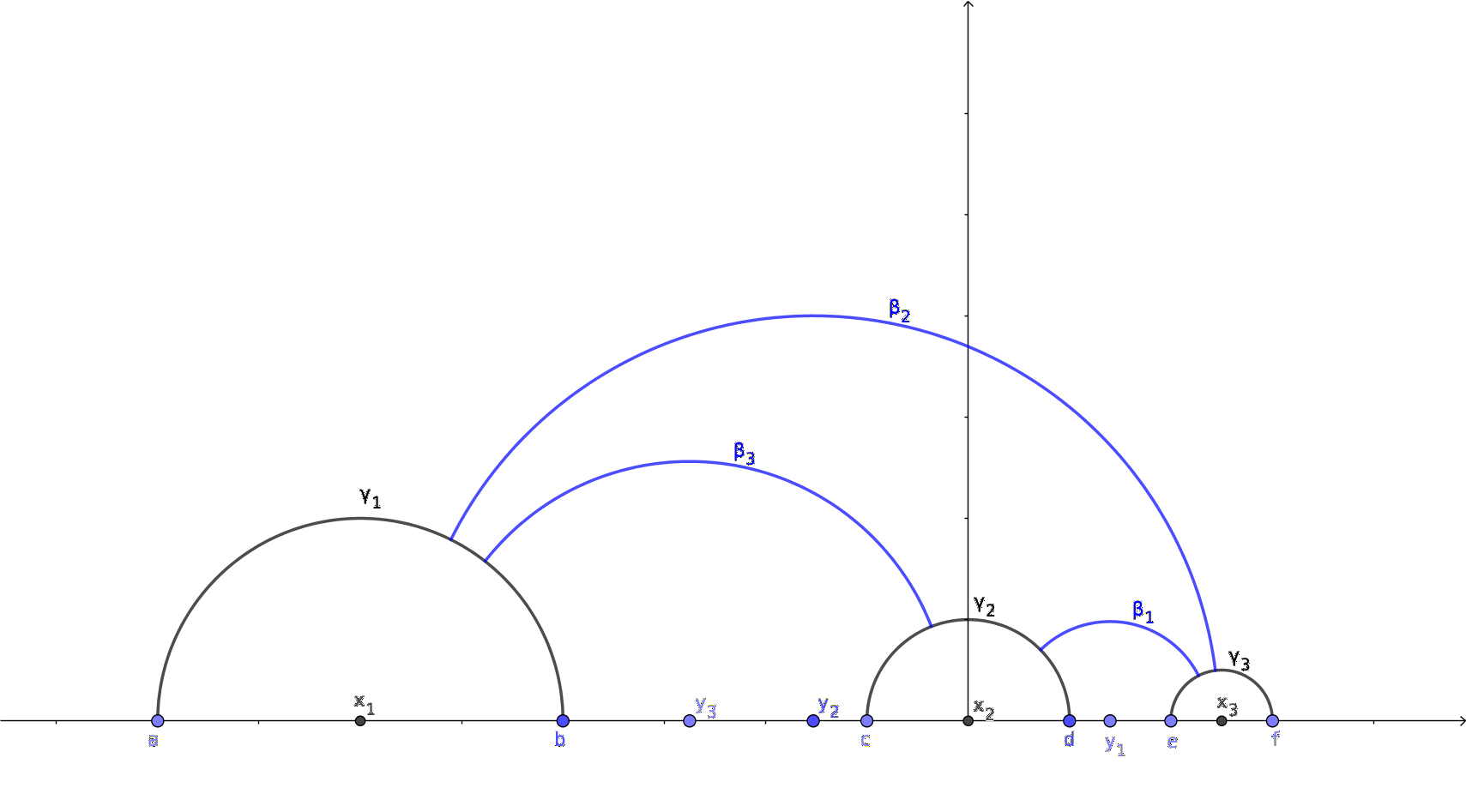}}
	\caption{Common perpendiculars}
	\label{comperp}
\end{figure}

\begin{proof}
	Label the endpoints of $\ga_1,\ga_2,\ga_3$ by $\{a,b\}, \{c,d\},\{e,f\}$ such that \[ a<b\leq c<d\leq e<f. \] Then from Lemma \eqref{centre}, we get that 
	\begin{align}
		x_1&=\frac{a+b}{2}, &y_1&=\frac{ef-cd}{e+f-c-d},\\
		x_2&=\frac{c+d}{2}, & y_2&=\frac{ef-ab}{e+f-a-b},\\
		x_3&=\frac{e+f}{2},& y_3&=\frac{cd-ab}{c+d-a-b},
	\end{align} 
	Using these coordinates, we calculate the right hand side of \eqref{ratio}:
	\begin{align*}
		y_1-y_2&=\frac{ef-cd}{e+f-c-d}-\frac{ef-ab}{e+f-a-b}\\
		&=\frac{ ab(e+f-c-d)+cd(a+b-e-f)+ef(c+d-a-b)}{(e+f-c-d)(e+f-a-b)},\\
		y_2-y_3&=\frac{ef-ab}{e+f-a-b}-\frac{cd-ab}{c+d-a-b}\\
		&=\frac{ ab(e+f-c-d)+cd(a+b-e-f)+ef(c+d-a-b)}{(e+f-a-b)(c+d-a-b)}.
	\end{align*}
	Hence,
	\begin{align}
		\frac{y_1-y_2}{y_2-y_3}&=\frac{(c+d-a-b)}{(e+f-a-b)}
		=\frac{x_1-x_2}{x_2-x_3}.
	\end{align}
\end{proof}
\begin{lem}\label{ineqcentres}
	Let $y_1,y_2,y_3$ be as in the hypothesis of the previous lemma. Then we have $y_3<y_2<y_1$.
\end{lem}
In order to prove this, we need the following lemma:
\begin{lem}\label{ineqasym}
	Let $\ga_1:=(a,b)$ and $\ga_2:=(b,c)$ be two asymptotic geodesics in $\HP$. Let $\ga_3:=(e,f)$ be another geodesic ultraparallel to $\ga_1,\ga_2$ such that \begin{equation}\label{asymineq}
		a<b<c<e<f.
	\end{equation} Let $\be_1, \be_2$ be the common perpendiculars to the pairs $\ga_2,\ga_3$ and $\ga_1,\ga_3$. Let $y_i$ be the centre of the semi-circle $\be_i$, for $i=1,2$. Then we have $y_1>y_2$.
\end{lem}
\begin{proof}
	From Lemma \eqref{centre}, we know that,
	\begin{align*}
		y_1&=\frac{ef-bc}{e+f-b-c},& y_2&=\frac{ef-ab}{e+f-a-b}.
	\end{align*}
	Calculating their difference, we get that,
	\begin{align*}
		y_1-y_2&=\frac{ef-bc}{e+f-b-c}-\frac{ef-ab}{e+f-a-b}\\
		&=\frac{ef(c-a)+bc(a+b-e-f)+ab(e+f-b-c)}{(e+f-b-c)(e+f-a-b)}.
	\end{align*}
	Using the hypothesis \eqref{asymineq}, we know that the denominator of $y_1-y_2$ is positive. So it suffices to check the sign of the numerator.
	\begin{align*}
		ef(c-a)+bc(a+b-e-f)+ab(e+f-b-c)&=ef(c-a)+b\{c(a+b-e-f)+a(e+f-b-c)\}\\
		&=ef(c-a)+b\{c(b-e-f)-a(b-e-f)\}\\
		&=(c-a)(ef+b(b-e-f))\\
		&=(c-a)(e-b)(f-b).
	\end{align*}
	By eq\eqref{asymineq}, we have that the numerator is positive. Hence, $y_1>y_2$.
\end{proof}

\begin{proof}[Proof of Lemma \ref{ineqcentres}]
	Firstly, $x_1<x_2<x_3$. Then from \eqref{ratio} we get that, $y_1-y_2$ and $y_2-y_3$ have the same sign. So we shall calculate the sign of only one of them. 
	Let $\be$ be the common perpendicular to $\ga_3$ and $\ga:=(b,c)$. Let $y'$ be the centre of the semi-circle $\be$. Then using Lemma \eqref{ineqasym} for the geodesics $\ga,\ga_2,\ga_3$, we get that $y'<y_1$. Again, by using the same Lemma for the geodesics $\ga_1,\ga,\ga_3$, we get that $y_2<y'$. Hence, $y_1>y_2$.
\end{proof}
\subsection{The four types of polygons}
In this section, we will introduce the four different types of polygons which are the main objects of study in this paper.

\paragraph{Ideal Polygons.} An ideal $n$-gon, denoted by $\ip n$, is defined as the convex hull in $\HP$ of $n(\geq3)$ distinct points on $\HPb$. The points on the boundary are called \emph{vertices} and they are marked as $x_1,\ldots, x_n$. The \emph{edges} are infinite geodesics of $\HP$ joining two consecutive vertices. The restriction of the hyperbolic metric to an ideal polygon gives it a geodesically complete finite-area (equal to $\pi(n-2)$) hyperbolic metric with geodesic boundary. The top-left panel of Fig.\ref{4typespoly} illustrates an ideal quadrilateral.

\paragraph{Ideal once-punctured polygons.} For $n\geq 2$, an \emph{ideal once-punctured $n$-gon}, denoted by $\punc n$, is another non-compact complete hyperbolic surface with geodesic boundary, obtained from an ideal $(n+2)$-gon, by identifying two consecutive edges using a parabolic element $T \in\psl$ that fixes the common vertex. The resulting surface has a missing point which we shall call a \emph{puncture}. The fundamental group $\fg {\punc n}$ of the surface is generated by the homotopy class of a simple closed loop that bounds a disk containing this puncture inside the surface. If $\rho:\fg {\punc n} \longrightarrow\psl$ is the holonomy representation, then $\rho(\fg {\punc n})\simeq \Z$, with $\rho(\ga)$ a parabolic element of $G$. The edges of the polygon are the connected components of the boundary. The vertices are the ideal points. In the bottom-left panel of Fig.\ref{4typespoly}, we have a once-punctured bigon.
\begin{figure}[h!]
	\centering
	\includegraphics[width=10cm]{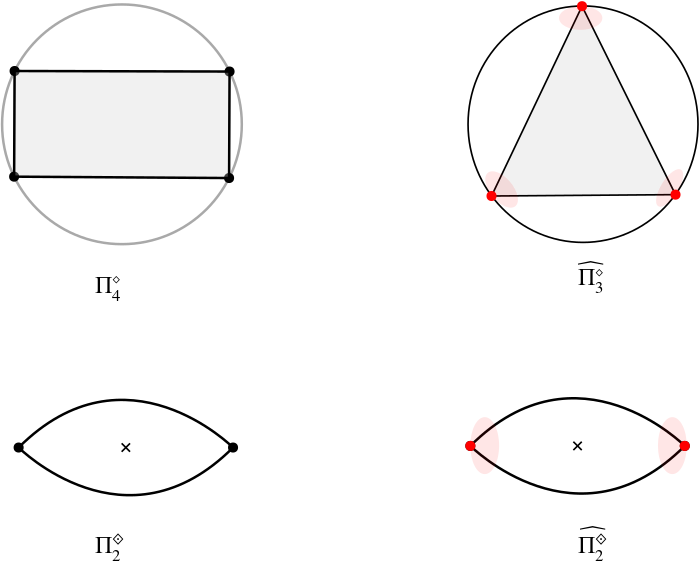}
	\caption{The four types of polygons}
	\label{4typespoly}
\end{figure}

\paragraph{Decorated Polygons.} An ideal vertex $v$ is said to be \emph{decorated} if a horoball, based at $v$ is added. For $n\geq 3$, a \emph{decorated ideal $n$-gon}, denoted by $\dep n$, is an ideal $n$-gon, all of whose vertices are  decorated with pairwise disjoint horoballs. Similarly, for $n\geq 2$, a \emph{decorated ideal once-punctured $n$-gon}, denoted by $\depu n$, is a once-punctured ideal $n$-gon, all of whose vertices are decorated with pairwise disjoint horoballs. See right panels of Fig.\ref{4typespoly}.

For an ideal or punctured polygon, its \emph{deformation space} is defined to be the set of all complete finite-area hyperbolic metrics with geodesic boundary, up to isometries that preserve the markings of the vertices.
\begin{thm}\label{tei}~
	\begin{enumerate}
\item The deformation space $\tei {\ip n}$ of an ideal polygon $\ip n$, $n\geq 3$, is homeomorphic to an open ball $\ball{n-3}$.
\item The deformation space $\tei {\punc n}$ of a punctured polygon $\punc n$, $n\geq 1$, is homeomorphic to an open ball $\ball{n-1}$.
\item 	The deformation space $\tei {\dep n}$ of a decorated polygonal surface $\dep n$ ($n\geq 3$)  is homeomorphic to an open ball of dimension $2n-3$. 
\item The deformation space $\tei {\dep n}$ of a decorated once-punctured polygonal surface $\dep n$ ($n\geq 3$) is homeomorphic to an open ball of dimension $2n-1$. 
	\end{enumerate}
\end{thm}
\begin{proof}
	Let $x_1,\ldots,x_n \in \R\cup\{\infty\}$ denote the cyclically ordered vertices of an ideal polygon. Since the isometry group $G$ of $\HP$, acts triply transitively on $\HPb$, there exists a unique $g\in G$ that maps $(x_1,x_2,x_3)$ to $(\infty, 0, 1)$. Therefore, a metric on an ideal $n$-gon is determined by the real numbers $x_4,\ldots,x_n$. Hence, the deformation space $\tei {\ip n}=\left\{(x_4,\ldots,x_n)\in \R^{n-3}: 1<x_4<\ldots<x_n \right\}$ is homeomorphic to $\ball{n-3}$. \\
	Since a once-punctured $n$-gon  is constructed by identifying two consecutive edges of an ideal $(n+2)$-gon, we have that $\tei{\punc n}\simeq \ball{n-1}$. There is one real-parameter family of horoballs based on an ideal point. So the deformations spaces of $\tei {\dep n}$ and $\tei {\depu n}$ are homeomorphic to $\ball{2n-3}$ and $\ball{2n-1}$, respectively.
\end{proof}
	Given a polygonal surface $\Pi$, a vector in the tangent space $\tang{\Pi}$ is called an \emph{infinitesimal deformation} of $\Pi$.

\begin{defi}
The \emph{admissible cone} of a decorated polygonal surface $\dep n$ is defined to be the set of all infinitesimal deformations of a metric $m\in\tei {\dep n}$, such that all the decorated vertices are moved away from each other. It is denoted by $\adm{m}$.
\end{defi}

\begin{lem}
The admissible cone of a decorated (possibly punctured) polygon $\Pi$, endowed with a metric $m$, is an open convex subset of $\tang \Pi$.
\end{lem}
\begin{proof}
Two decorated vertices are moved away from each other if and only if the length of the horoball connection joining them increases. Let $l_1,\ldots,l_N$ be the set of all edges and diagonals of the polygon. Then we can define the following smooth positive function for every $i=1,\ldots,N$:
\[
\begin{array}{ccl}
	l_i:&\tei \Pi\longrightarrow& \R_{>0}\\
	&m\mapsto&\text{length of $l_i$ w.r.t $m$.}
\end{array}
\] 
An infinitesimal deformation $v$ increases the length of $l_i$ if and only if $dl_i(v)>0$. So the admissible cone can be written as
\[\adm m=\bigcap\limits_{i=1}^N \{dl_i>0\}, \] which is open and convex in $\tang {\dep n}$.
\end{proof}

\section{Arcs and arc complexes} \label{arc}
\subsection{The different kinds of arcs}
An \emph{arc} on a polygon $\Pi$, is an embedding $\al$ of a closed interval $I\subset \R$ into $\Pi$. There are two possibilities depending on the nature of the interval:
\begin{enumerate}
	\item $I=[a,b]$: In this case, the arc $\al$ is finite. We consider those finite arcs that verifiy:  $\al(a),\al(b) \in \partial \Pi$ and $\al(I)\cap S=\set{\al(a),\al(b)}$. 
	\item $I=[a,\infty)$: These are embeddings of hyperbolic geodesic rays in the interior of the polgyon such that  $\al(a)\in \partial \Pi$.  The infinite end converges to an ideal point, \ie $\al(t)\overset{t\to \infty}{\longrightarrow} x$, where $x\in \HP$.
\end{enumerate}  
An arc $\al$ of a polygon $\Pi$ with non-empty boundary is called \emph{non-trivial} if each connected component of $\Pi\smallsetminus \{\al \}$ has at least one spike or decorated vertex. 
	
Let $\mathscr A$ be the set of all non-trivial arcs of the two types above. Two arcs $\al,\al':I\longrightarrow \Pi$ in $\mathscr A$ are said to be \emph{isotopic} if there exists a homeomorphism $f:\Pi\longrightarrow \Pi$ that preserves the boundary and fixes all decorated vertices or (possibly decorated) spikes and a continuous function $H:\Pi \times [0,1]\longrightarrow \Pi\}$ such that 
\begin{enumerate}
	\item $H(\cdot,0)=\mathrm{Id}$ and $H(\cdot,1)=f$,
	\item for every $t\in [0,1]$, the map $H(\cdot,t): S\longrightarrow \Pi$ is a homeomorphism,
	\item for every $t\in I$, $f(\al(t))=\al'(t)$.
\end{enumerate}

\begin{defi}\label{ac}
	The \emph{arc complex} of a surface $\Pi$, generated by a subset $\mathcal{K}\subset \mathscr A$, is a simplicial complex $\ac \Pi$ whose base set $\ac{\Pi}^{(0)}$ consists of the isotopy classes of arcs in $\mathcal K$, and there is an $k$-simplex for every $(k+1)$-tuple of pairwise disjoint and distinct isotopy classes. 
\end{defi}
The elements of $\mathcal{K}$ are called \emph{permitted} arcs and the elements of $\mathscr A\smallsetminus\mathcal{K} $ are called \emph{rejected} arcs. 

Next we specify the elements of $\mathcal{K}$ for the different types of surfaces:
\begin{itemize}
	\item In the case of an undecorated ideal or punctured polygon, the set $\mathcal K$ of permitted arcs comprises of non-trivial finite arcs that separate at least two spikes from the surface.
	\item In the case of decorated polygons, an arc is permitted if either both of its endpoints lie on two distinct edges of $\dep n$ (\emph{edge-to-edge} arc) or exactly one endpoint lies on a decorated vertex (\emph{edge-to-vertex} arc).
\end{itemize}
\begin{rem}
	\begin{itemize}
		\item Two isotopy classes of arcs of $\Pi$ are said to be disjoint if it is possible to find a representative arc from each of the classes such that they are disjoint in $\Pi$. Such a configuration can be realised by geodesic segments in the context of polygons. 
		Since the surface is endowed with a metric of constant negative curvature, such a configuration can be realised by arcs that are geodesics segments with respect to such a metric. 
		In our discussion, we shall always choose such arcs as representatives of the isotopy classes. 
		\item In the cases of ideal and punctured polygons, we shall choose those geodesic arcs whose lifts are supported on projective lines that intersect outside $\pp\smallsetminus\cHP$.
	\end{itemize}
\end{rem}
\paragraph{Vocabulary.}	The 0-skeleton $\sigma^{(0)}$ of a top-dimensional simplex $\sigma$ of the arc complex is called a \emph{triangulation} of the polygon. A finite arc of a one-holed ideal polygon or a once-punctured ideal polygon is called \emph{maximal} if both its endpoints lie on the same edge. A finite arc of an non-decorated polygon is called \emph{minimal} if it separates a quadrilateral with two ideal vertices from the surface.

\begin{defi}\label{big} We define a \emph{filling} simplex of the arc complex of the different types of surfaces:
	\begin{itemize}
		\item For an undecorated ideal or a punctured polygon, a simplex $\sigma$ is said to be filling if the arcs corresponding to $\sigma^{(0)}$  decompose the surface into topological disks with at most two vertices. 
		\item For a decorated polygon, a simplex $\sigma$ is said to be filling if the arcs corresponding to $\sigma^{(0)}$  decompose the surface into topological disks with at most one vertex and a punctured disk with no vertex. 
	\end{itemize}
\end{defi}
From the definition it follows that any simplex containing a filling simplex is also filling. 
\begin{defi}\label{pac}
	The \emph{pruned arc complex} of a polygon $\Pi$, denoted by $\sac \Pi$ is the union of the interiors of the filling simplices of the arc complex $\ac \Pi$.
\end{defi}
Every point $x\in \sac {\Pi}$ is contained in the interior of a unique simplex, denoted by $\sigma_x$, \ie there is a unique family of arcs $\{\al_1,\ldots,\al_p\}$, namely the 0-skeleton of $\sigma_x$, such that \[ x=\sum_{i=1}^p t_{i} \al_i, \, \sum_{i=1}^p t_i =1,\,\text{ and } \forall i, \, t_i>0 .\]
Define the \emph{support} of a point $x\in \sac \Pi$ as $\supp x:= \sigma_x^{(0)}$.

\subsection{Ideal and Punctured Polygons}
To every ideal polygon $\ip n$, one can associate a Euclidean regular polygon with $n$ vertices, denoted by $\poly n$, in the following way:
\begin{itemize}
	\item The vertices of $\poly n$ correspond to the infinite geodesics of the boundary of $\ip n$,
	\item Two vertices in $\poly n$ are consecutive if and only if the corresponding infinite geodesics have a common ideal endpoint.
\end{itemize}
See Fig.\ref{arc2diag} for a transformation between an ideal quadrilateral and a Euclidean square.
\begin{figure}
	\centering
	\includegraphics[width=12cm]{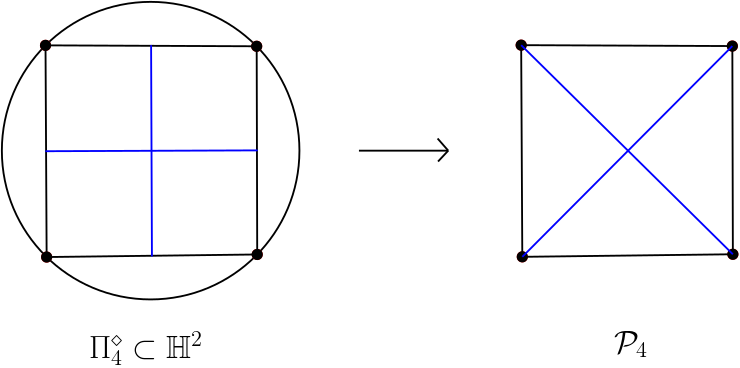}
	\caption{From hyperbolic to Euclidean polygon}
	\label{arc2diag}
\end{figure}
Then we have the following bijection:
\[\cur{\text{Isotopy classes of permitted arcs of } \ip n} \leftrightarrow \cur{\text{Diagonals of }\poly n}\]
Two distinct isotopy classes are pairwise disjoint if and only if the corresponding diagonals in $\poly n$ don't intersect inside $\poly n$. However, the diagonals are allowed to intersect at vertices – this takes place whenever the arcs have exactly one endpoint on a common edge of the ideal polygon. One can construct the arc complex of $\poly n$ in the same way as before and one has $\ac {\poly n}=\ac {\ip n}$. 
\begin{figure}
\begin{subfigure}{\linewidth}
			\centering
	\includegraphics[width=12cm]{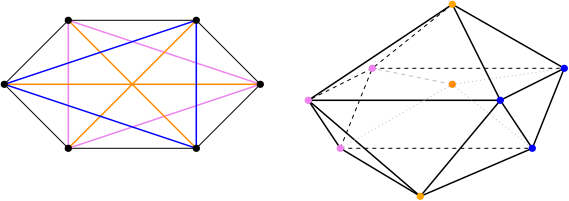}
\subcaption{The arcs and the arc complex of a hexagon $\poly6$}
\label{achexa}
	\end{subfigure}
\vspace{0.5cm}
\begin{subfigure}{\linewidth}
	\centering
\includegraphics[width=12cm]{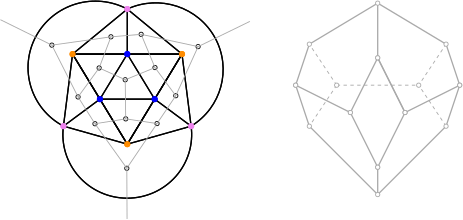}
\subcaption{The flattened perspective of $\ac{\poly 6} $ and its dual associahedron of dimension 2}
\label{asso}
\end{subfigure}
\caption{}
\end{figure}
The following theorem is a classical result from combinatorics about the topology of the arc complex of a polygon. See, for instance, \cite{penner} for a proof by Penner.
\begin{thm}\label{acideal}
	The arc complex $\ac {\poly n}$ ($n\ge 4$) is a sphere of dimension $n-4$. 
\end{thm}
 Fig.\ref{achexa} shows the arcs and the arc complex of a hexagon. The dual of the codimension 0 and 1 simplices gives a convex polytope known as \emph{associahedron}. See Fig.\ref{asso} for the associahedron of dimension 3.

The following theorem about the arc complex of once-punctured polygons was proved by Penner in \cite{penner}.
\begin{thm}\label{idpac}~
	The arc complex $\ac{\punc n}$ of a punctured $n$-gon,($n\geq 2$), is homeomorphic to a sphere of dimension $n-2$.
\end{thm}
\subsection{Pruned arc complex of decorated polygons}
In this subsection, we shall prove that the pruned arc complexes of a decorated ideal polygon and a decorated once-punctured polygon are open manifolds. Since the permitted arcs in this case are allowed to have one endpoint on a decorated vertex, we consider the following abstract set up to cover all the cases at the same time. 

We start with the polygon $\poly {2n}$ (defined in the previous section) with $n\geq 2$ and partition its vertex set into two disjoint subsets $G$ and $R$ such that $|G|=|R|=n$ and for every pair of consecutive vertices, exactly one belongs to $G$ and the other one belongs to $R$. Such a polygon is said to have an \emph{alternate partitioning} and shall be denoted by $(\poly {2n}, C_{alt})$, where $C_{alt}:=(G,R)$. 
\begin{figure}[!h]
	\centering
	\includegraphics[width=12cm]{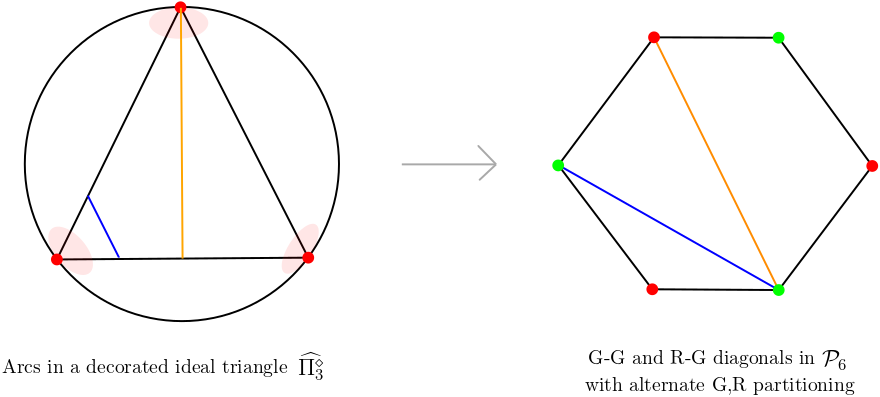}
	\caption{From decorated ideal $n$-gon to Euclidean $\poly{2n}$ with alternate partitioning}
	\label{arc2diagdeco}
\end{figure}
To every decorated polygon $\dep n$, one can associate the polygon $(\poly {2n}, C_{alt})$ in the following way:
\begin{itemize}
	\item a decorated vertex of $\dep n$ corresponds to a vertex of $\poly {2n}$ in $R$,
	\item an edge of $\dep n$ corresponds to a vertex of $\poly {2n}$ in $G$,
\end{itemize}
such that one $R$-vertex and one $G$-vertex are consecutive in $\poly {2n}$ if and only if the corresponding edge and decorated vertex of $\dep n$ are consecutive.
Again, we have the bijection:
\[\cur{\text{Isotopy classes of edge-to-edge arcs of } \dep n} \leftrightarrow \cur{G-G\text{ diagonals}}\]
\[\cur{\text{Isotopy classes of edge-to-vertex arcs of } \dep n} \leftrightarrow \cur{G-R\text{ diagonals} }\]

So the arc complex $\ac {\dep n}$ (resp. $\ac {\depu n}$ ) is isomorphic to the subcomplex $\acsub{\poly{2n}}$ of $\ac{\poly {2n}}$ (resp. $ \acsub{\puncp{2n}}$ of $\ac{\puncp{2n}}$) generated by the $G-G$ and $G-R$ diagonals. In the case of a polygon without puncture $\poly{2n}$, the $k+1$ diagonals of a filling simplex decompose the surface into $k+2$ smaller polygons none of which has more than one $R$-vertex. In the case of a punctured polygon $\puncp{2n}$, the $k+1$ diagonals of a filling simplex decompose the surface into $k+1$ smaller unpunctured polygons none of which has more than one $R$-vertex and exactly one smaller punctured polygon without any $R$-vertex. The boundary of $\acsub{\poly{2n}}$ as well as $\acsub{\puncp{2n}}$ consists of all the non-filling simplices. So the pruned arc complex $\sac {\dep n}$ (resp. $\sac{\depu n}$) is the interior of $\acsub{\poly{2n}}$ (resp. $\acsub{\puncp {2n}}$).\\
In the following theorem we prove that the interior of these subcomplexes form open manifolds of given dimensions.
\begin{thm}\label{sacdp}
	\begin{enumerate}
\item 	The interior of the simplicial complex $\acsub{\poly{2n}}$, ($n\geq 2$) of a polygon $\poly{2n}$ with an alternate partitioning is an open manifold of dimension $2n-4$. 
\item 	The interior of the simplicial complex $\acsub{\puncp {2n}}$ ($n\geq 1)$ of a once-punctured $\puncp{2n}$ with an alternate partitioning is an open manifold of dimension $2n-2$. 
	\end{enumerate}
\end{thm}
\begin{proof}
	
	\begin{enumerate}
	\item Let $x \in \acsub{\puncp{2n}}$ be point which lies in the interior of a unique simplex $\sigma_x$ of dimension say $k$. Here, $n-1\leq k \leq 2n-4$. We need to show that there is a neighbourhood of $x$ in $\sac{\poly{2n}}$ which is homeomorphic to an open ball of dimension $2n-4$. It suffices to prove that the link of $\sigma_x$ in the arc complex is a sphere of dimension $2n-5-k$.
	The $k+1$ arcs of the $0$-skeleton of $\sigma_x$ divide the polygon $\poly{2n}$ into $k+2$ smaller polygons $\poly {n_1},\ldots, \poly{n_{k+2}}$, with $3\leq n_i \leq 2n-1$ for every $i=1,\ldots,k+2$. 
	
	\begin{lem}\label{smalllemma}
		Let $s:=\sum_{r=1}^{k+2} n_r$ be the total number of vertices of all the smaller polygons. Then we have, 
		\[s=2(k+1)+2n.\]
	\end{lem}
	\begin{proof}
		For $p=1,\ldots,2N$, let $e_p$ be the total number of arcs that have endpoints on the $p$-th vertex. Let $w_p$ be the number of times the $p$-th vertex appears as a vertex of a smaller polygon. Then $w_p=e_p+1$ and $\sum_{p=1}^{2N} e_p=2(k+1)$. Hence we have $s=\sum_{p=1}^{2N} w_p=2(k+1)+2N$.
	\end{proof}
	
	Since $\sigma_x$ is a filling simplex, we have that none of the smaller polygons contain a $R-R$ diagonal. So each of their arc complexes is a sphere, from Theorem \eqref{acideal}. The link is then given by 
	\begin{align*}
		\Link{\ac{\poly{2n}}}{\sigma_x}&=\ac{\poly {n_1}}\Join\ldots \ac{\poly {n_{k+2}}}\\
		&=\s{n_1-4}\Join\ldots\s{n_{k+2}-4}& (\text{from Theorem \eqref{acideal}})\\
		&=\s{s-4(k+2)+k+1}\\
		&=\s{2n-5-k} &(\text{from Lemma \eqref{smalllemma}})
	\end{align*}
\item  Again we need to prove that the link of a $k$-dimensional filling simplex $\sigma$ in the arc complex is a sphere of dimension $2n-3-k$.
The $k+1$ arcs of the $0$-skeleton of $\sigma_x$ divide the punctured polygon $\puncp{2n}$ into $k+1$ smaller convex polygons with at most one $\poly {n_1},\ldots, \poly{n_{k+1}}$, with $3\leq n_i \leq 2n-1$ for every $i=1,\ldots,k+1$ and exactly one punctured polygon $\puncp{n_0}$, ($n_0\geq2$) without any $R$-vertex. So we have,
	\begin{align*}
	\Link{\ac{\poly{2n}}}{\sigma}&=\ac{\puncp{n_0}}\Join\ac{\poly {n_1}}\Join\ldots \ac{\poly {n_{k+1}}}\\
	&=\s{n_0-2}\Join \s{n_1-4}\ldots\s{n_{k+2}-4}& (\text{from Theorem \eqref{acideal}})\\
	&=\s{s-2-4(k+1)+k+1}\\
	&=\s{2n-5-k}, &(\text{from Lemma \eqref{smalllemma}}).
\end{align*}
\end{enumerate}
\end{proof}
\subsection{Tiles}\label{tilestypes}
	Let $S$ be a hyperbolic surface endowed with a hyperbolic metric $m\in \tei  S$. Let $\mathcal{K}$ be the set of permitted arcs for an arc complex $\ac S$ of the surface. Given a simplex $\sigma\subset \ac S$, the \emph{edge set} is defined to be the set \[ \ed:=\set{\al_g(m)\in \al | \al\in \sigma^{(0)}},\] where $\al_g(m)$ is a geodesic representative from its isotopy class. The set of all lifts of the arcs in the edge set in the universal cover $\wt S\subset \HP$ is denoted by $\led$. The set of connected components of the surface $S$ in the complement of the arcs of the edge set is denoted by $\tile$. The lifts of the elements in $\tile$ in $\HP$ are called \emph{tiles}; their collection is denoted by $\ltile$.
	\begin{rem}
		In the case of ideal polygons and decorated polygons, these components are homeomorphic to two-dimensional disks. In the case of punctured polygons, one of the components is a punctured disk.
	\end{rem}
The sides of a tile are either contained in the boundary of the original surface or they are the arcs of $\ed$. The former case is called a \emph{boundary side} and the latter case is called an \emph{internal side}. Two tiles $d,d'$ are called \emph{neighbours} if they have a common internal side. The tiles having finitely many edges are called \emph{finite}. 

If $\sigma$ has maximal dimension in $\ac S$, then the finite tiles can be of three types:
\begin{itemize}
	\item[Type 1:]  The tile has only one internal side, \ie it has only one neighbour.
	\item [Type 2:] The tile has two internal sides, \ie two neighbours.
	\item [Type 3:] The tile has three internal sides, \ie three neighbours.
\end{itemize}

\begin{rem}
	Any tile, obtained from a triangulation using a simplex $\sigma$, must have at least one and at most three internal sides. Indeed, the only time a tile has no internal side is when the surface is an ideal triangle. Also, if a tile has four internal sides, then it must also have at least four distinct boundary sides to accommodate at least four endpoints of the arcs. The finite arc that joins one pair of non-consecutive boundary sides lies inside $\mathcal{K}$. This arc was not inside the original simplex, which implies that $\sigma$ is not maximal. Hence a tile can have at most 3 internal sides.
\end{rem}
\begin{figure}
	\centering
	\includegraphics[height=18cm]{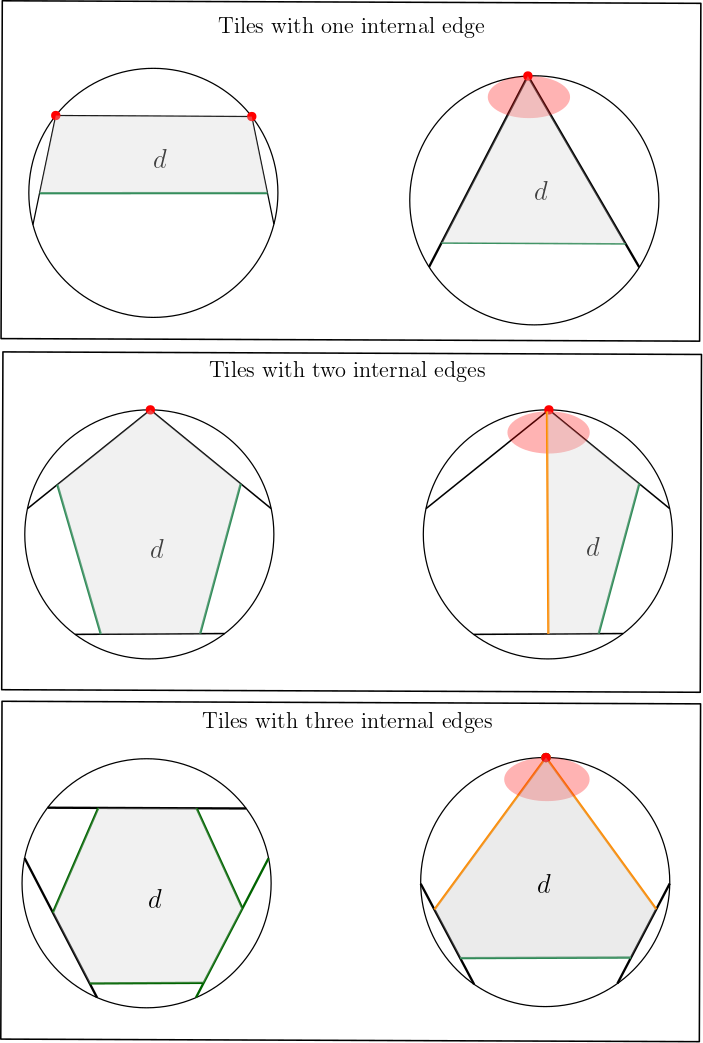}\vspace{0.3cm}
	\caption{Finite tiles for undecorated and undecorated polygons}
	\label{tiles}
\end{figure}

\paragraph{Undecorated polygons:}  There are three types of tiles possible after triangulating (cf. first column of Fig \ref{tiles}):
\begin{itemize}
	\item a hyperbolic quadrilateral with two ideal vertices and one permitted arc,
	\item a hyperbolic pentagon with one ideal vertex and two permitted arcs as alternating edges,
	\item a hyperbolic hexagon with three permitted arcs as alternating edges.
\end{itemize}
\paragraph{Decorated Polygons:} The different types of tiles possible in the case of a decorated polygon are shown in the last three columns of the table in Fig. \ref{tiles}.
\begin{itemize}
	\item When there is only one internal side of the tile, that side is an edge-to-edge arc of the original surface. The tile contains exactly one decorated vertex $\nu$ and two boundary sides. The three cases corresponding to the three possible types of the vertex are given in the first row of the table in Fig. \eqref{tiles}. 
	\item When there are two internal sides (second row in Fig. \eqref{tiles}), one of them is an edge-to-vertex and the other one is of edge-to-edge type. So the tile contains a decorated vertex. 
		\item There are two possibilities in this case: either all the three internal sides are of edge-to-edge type (fourth row in Fig. \eqref{tiles}) or two of them are edge-to-vertex arcs and one edge-to-edge arc (third row in Fig. \eqref{tiles}). In the former case, the tile does not contain any vertex whereas in the latter case it contains one.
\end{itemize}

\paragraph{Punctured polygons:} In the case of a punctured polygon, the possible connected components after cutting the surface along the arcs of the edgeset, can be of the three types as in the case of surfaces with non-decorated spikes and also a hyper-ideal punctured monogon. The lift of the latter in $\HP$ is an infinite polygon with one ideal vertex, as in Fig.\ \ref{punctile}. 
\begin{figure}
	\frame{	\includegraphics[width=15cm]{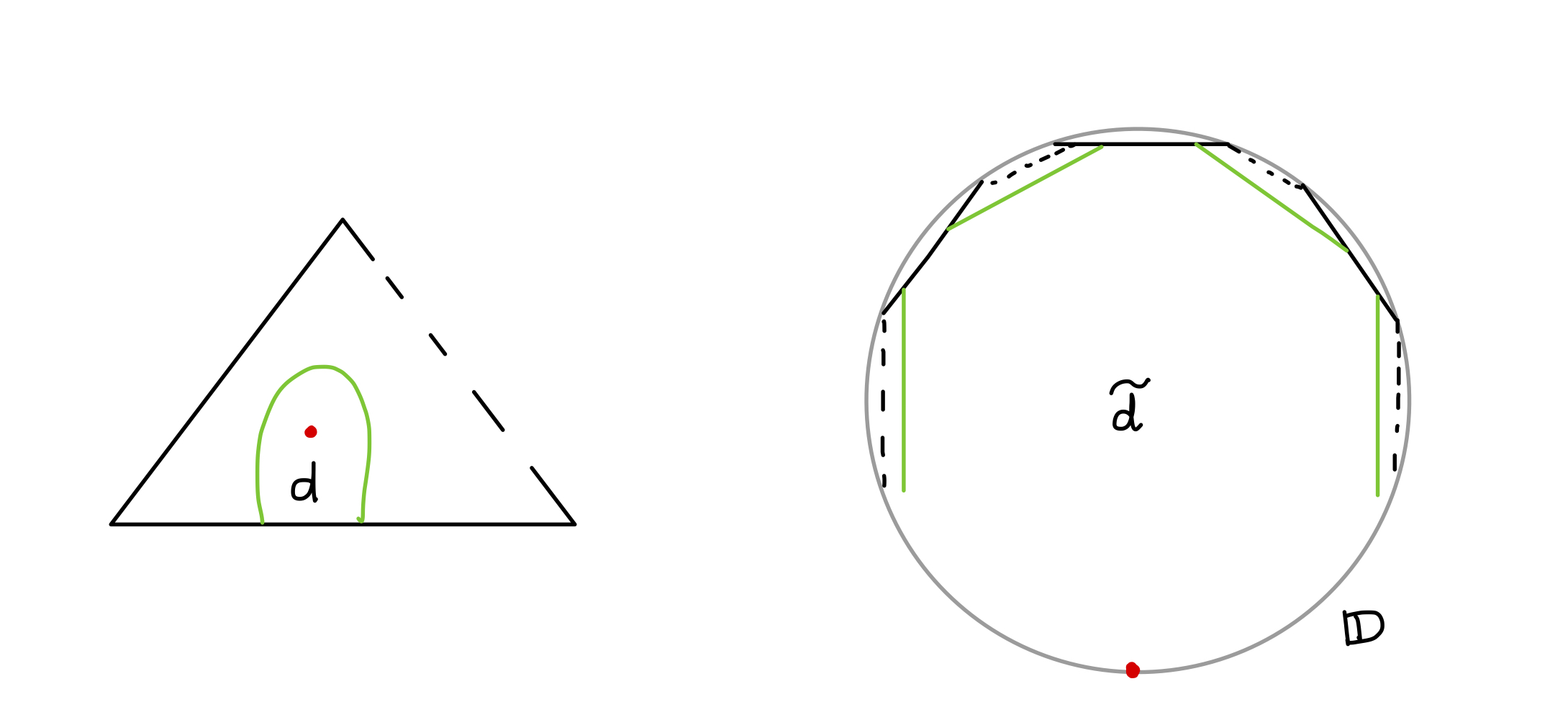}}
	\caption{Infinite tile containing the puncture}
	\label{punctile}
\end{figure}
\paragraph{Dual graph to a triangulation.} Let $\sigma\in \ac \Pi$ be a triangulation of a polygon ${\Pi}$. Then the corresponding \emph{dual graph} is a graph embedded in the universal cover of the surface such that the vertex set is $\ltile$ and the edge set is given by unordered pairs of lifted tiles that share a lifted internal edge. A vertex of the graph has valency 1 (resp. 2, 3) if and only if the corresponding tile is of the type 1 (resp. 2, 3).

\paragraph{Refinement.}  Let $\sigma$ be a top-dimensional simplex of an arc complex $\ac S$ of a hyperbolic surface $S$.  Let $\be$ be an arc such that $[\be]\in \ac S^{(0)}\smallsetminus \sigma^{(0)}$. So, $\be$ intersects every arc in the isotopy class of at least one arc in $ \sigma$. The set $\sigma\cup [\be]$  is called a \emph{refinement} of the triangulation $\sigma$. Let $\rtile$ be the set of connected components of $S\smallsetminus ( \be \bigcup\limits_{\al\in \ed} \al )$ and $\lrtile$ be the set of lifts of its elements. The elements of $\lrtile\smallsetminus\ltile $ are called \emph{small tiles}. 


	\section{Strip deformations}\label{Sd}
	In this section, we will introduce strip deformations, strip templates, tiles and tile maps. We shall also recapitulate the main ideas from the proof of the motivating theorem proved by Danciger-Guéritaud-Kassel in \cite{dgk}.
	
	Informally, a strip deformation of a polygon is done by cutting it along a geodesic arc and gluing a strip of the hyperbolic plane $\HP$, without any shearing. The type of strip used depends on the type of arc and the surface being considered. 
	\subsection{The different strips}
	Firstly, we define the different types of strips.
	Let $l_1$ and $l_2$ be any two geodesics in $\HP$. 
	Then there are two types of strips depending on the nature of their intersection:
	\begin{itemize}
		\item Suppose that $l_1$ and $l_2$ are disjoint in $\cHP$. Then the region bounded by them in $\HP$ is called a hyperbolic strip. The \emph{width} of the strip is the length of the segment of the unique common perpendicular $l$ to $l_1$ and $l_2$, contained in the strip. The \emph{waist} of the strip is defined to be the set of points of intersection $l\cap l_1$ and $l\cap l_2$.
		\item Suppose that $l_1$ and $l_2$ intersect in $\HPb$ at a point $p$. Let $h$ be a horocycle based at $p$. Then the region bounded by them inside $\HP$ is called a \emph{parabolic strip.} The waist in this case is defined to be the ideal point $p$ and the width (w.r.t $h$) is defined to be the length of the horocyclic arc of $h$ subtended by $l_1$ and $l_2$.
	\end{itemize}
	
	\subsection{Strip template}\label{template}
	Let $\Pi$ be a polygon endowed with a metric $m\in \tei \Pi$ on it. Let $\mathcal{K}$ be the set of permitted arcs (Definition \eqref{ac}). 
	A strip template is the following data:
	\begin{itemize}
		\item an $m$-geodesic representative $\al_g$ from every isotopy class $\al$ of arcs in $\mathcal{K}$, along which the strip deformation is performed,
		\item a point $p_\al\in \al_g$ where the waist of the strip being glued must lie.  
	\end{itemize} 
	A choice of strip template is the specification of this data. However, we shall see in the following section that even though we are allowed to choose the geodesic arcs in every case, the waists are sometimes fixed from beforehand by the nature of the arc being considered.
	\paragraph{Finite arcs:}Recall that finite arcs are embeddings of a closed and bounded interval into the surface with both the endpoints lying on the boundary of the surface. These arcs are present in the construction of every arc complex that we discuss. The strip glued along these arcs is of hyperbolic type. The representative $\al_g$ from the isotopy class of such an arc can be any geodesic segment from $v$ to that edge. In every case, including edge-to-edge arcs in decorated polygons, we are free to chose the geodesic representative and the waist of the hyperbolic strip.
	
	\paragraph{Infinite arcs:}Let $\al$ be the isotopy class of a permitted infinite arc of a decorated polygon $\dep n$. Then 
	An arc in $\al$ has one finite end lying on $\partial \dep n$ and one infinite end that escapes the surface through a spike. We can choose any geodesic arc $\al_g$ from $\al$ that does the same without any self intersection.

Now we will give a formal definition of a strip deformation and its infinitesimal version. 
	\begin{defi}
		Given an isotopy class $\al$ of arcs and a choice of strip template ($\al_g,\pal,\wal$), define the \emph{strip deformation} along $\al$ to be a map
		\[
		F_{\al}:\tei \Pi \longrightarrow \tei \Pi
		\]
		where the image $F_{\al}(m)$ of a point $m\in \tei \Pi$ is a new metric on the surface obtained by cutting it along the $m$-geodesic arc $\al_g$ in $\al$ chosen by the strip template and gluing a strip whose waist coincides with $\pal$. The type of strip used depends on the type of arc and the surface being considered. 
	\end{defi}
	\begin{defi}
		Given an isotopy class of arcs $\al$ of a polygon $\Pi$ and a strip template $\{(\al_g,\pal,\wal)\}_{\al\in \mathcal{K}}$ adapted to the nature of $\al$ for every $m\in \tei {S}$, 
		define the \emph{infinitesimal strip deformation}
		
		\[
		\begin{array}{cc}
			f_{\al}:&\tei {S} \longrightarrow T\tei{S}\\
			&m\mapsto [m(t)]
		\end{array}
		\]
		where the image $m(\cdot)$ is a path in $\tei {S}$ such that $m(0)=m$ and $m(t)$ is obtained from $m$ by strip deforming along $\al$ with a fixed waist $\pal$ and the width as $t\wal$. 
	\end{defi} 
	
	Let $m=( [\rho, \vec{x}])\in \tei S$ be a point in the deformation space of the surface, where $\rho$ is the holonomy representation and denote $\Gamma=\rho(\fg S)$. Fix a strip template $\{\st \}$ with respect to $m$. Let $\sigma$ be a simplex of $\ac S$. Given an arc $\al$ in the edgeset $\ed$, there exist tiles $\del, \del'\in \tile$ such that every lift $\wt \al$ of $\al$ in $\wt S$ is the common internal side of two lifts $\wt \del,\wt {\del'}$ of the tiles. Also, $p_{\ga\cdot\wt\al}=\ga\cdot p_{\wt\al}$, for every $\ga\in \Gamma$. Then the infinitesimal deformation $f_{\al}(m)$ tends to pull the two tiles $ \del$ and ${\del'}$ away from each other due to the addition of the infinitesimal strip.  Let $u$ be a infinitesimal strip deformation of $\rho$ caused by $f_{\al}(m)$. Then we have a $(\rho, u)$-equivariant \emph{tile} map $\phi:\ltile \rightarrow\lalg$ such that for every $\ga\in \Gamma$, 
	\begin{align}\label{tile}
		\phi(\h\ga\cdot \wt\del)-\phi(\h\ga\cdot\wt\del')=\h\ga\cdot  v_{\wt\al},
	\end{align}
	where $v_{\wt\al}$ is the Killing field in $\lalg\simeq \mathscr X$ corresponding to the strip deformation $f_{\wt\al}(m)$ along a geodesic arc $\wt\al_g$, isotopic to $\wt\al$, adapted to the strip template chosen, and pointing towards $\wt\del$:
	\begin{itemize}
		\item If $\isda$ is a hyperbolic strip deformation with strip template $\st$, then $v_{\wt\al}$ is defined to be the hyperbolic Killing vector field whose axis is perpendicular to $\wt\al_g$ at the point $\wt\pal$, whose velocity is $\wal$.
		\item If $\al$ is an infinite arc joining a spike and a boundary component, then $\isda$ is a parabolic strip deformation with strip template $\st$, and $v_{\wt\al}$ is defined to be the parabolic Killing vector field whose fixed point is the ideal point where the infinite end of $\wt\al$ converges and whose velocity is .
	\end{itemize}
	\begin{rem}
		Such a strip deformation $f_{\al}:\tei S\longrightarrow\tang S$ does not deform the holonomy of a general surface with spikes (decorated or otherwise) if $\al$ is completely contained outside the convex core of the surface. However, it does provide infinitesimal motion to the spikes.
	\end{rem}
	More generally, a linear combination of strip deformations $\sum_{\al} c_\al f_{\al}(m)$ along pairwise disjoint arcs $\{\al_i\}\subset \ed$ imparts motion to the tiles of the triangulation depending on the coefficient of each term in the linear combination. A tile map corresponding to it is a $(\rho, u)$-equivariant map $\phi:\ltile \rightarrow\lalg$ such that for every pair $\del, \del ' \in \tile$ which share an edge $\al\in \ed$, the equation \eqref{tile} is satisfied by $\phi$. 
	
	\begin{defi}\label{ism}
		The \emph{infinitesimal strip map} is defined as:
		\[
		\begin{array}{ccrcl}
			\mathbb{P}f& : &	\sac S & \longrightarrow & \ptan {S}\\
			& &\sum\limits_{i=1}^{\dim \tei S} c_i \al_i&\mapsto&\bra{\sum\limits_{i=1}^{\dim \tei S}c_i f_{\al_i}(m)} 
		\end{array}
		\]
	\end{defi}
	where $\sac S$ is the pruned arc complex of the surface (Definition \eqref{pac}). 
	
	 Two tile maps $\phi, \phi'$ are said to be equivalent if for all $d\in \tile$, \[\phi(d)-\phi'(d)=v_0\in\lalg .\] The set of all equivalence classes of tile maps corresponding to a  simplex $\sigma\subset \ac S$ is denoted by $\Phi$. 
	  \begin{figure}[!h]
	 	\begin{center}
	 		\includegraphics[height=3cm]{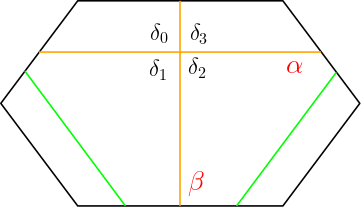}
	 		\caption{Refined tiles}
	 		\label{refine}
	 	\end{center}
	 \end{figure}
	 
	 Let $\sigma\cup [\be]$ be a refinement of $\sigma$. A \emph{consistent} tile map is a tile map $\phi: \mathcal{T}_{\sigma\cup [\be]}\longrightarrow\lalg$ that satisfies the consistency relation around every point of intersection: if the pairs $(\del_1, \del_0),(\del_3, \del_2)$ neighbour along $\al$ and the pairs $(\del_1, \del_3),(\del_0, \del_2)$ neighbour along $\be$ , then $\phi$ must satisfy
	 \begin{align}
	 	\phi(\del_1)-\phi(\del_0)=\phi(\del_3)-\phi(\del_2)=v_{\al},\\
	 	\phi(\del_1)-\phi(\del_3)=\phi(\del_0)-\phi(\del_2)=v_{\be},
	 \end{align}
	 
	 where $v_{\al}$ and $v_{\be}$ are the Killing vector fields adapted to the strip templates and the nature of $\al$ and $\be$.
	 The set of all equivalence classes modulo $\lalg$ of consistent tile maps is denoted by $\Phi^c$. Then there is a natural inclusion \[\Phi\subset\Phi^c.\] Also, we have the bijection between formal expressions of the form $\sum\limits_{\al \in \ed\cup \{\be\}}c_{\al}f_{\al}(m)$ and $\Phi^c$.

	 \begin{defi}\label{neutraltilemap}
	 	A \emph{neutral} tile map, denoted by $\phi_0$, is a tile map that fixes the decorated vertex or a spike of a tile whenever it has one and satisfies the equation 
	 	\begin{equation}
	 		\np{\gamma\cdot \delta})=\gamma\cdot \np \del, \text{ for every } \ga\in \Gamma.
	 	\end{equation}
	 \end{defi}
	 
	 Such a map belongs to the equivalence class corresponding to $0\in \tang S$.

		\subsection{Some useful estimates}
	Let $\Pi$ be a polygon with (possibly decorated) spikes with a metric $m$. Consider a strip deformation $\isda$ along a finite arc $\al$, with strip template $\st$. Then the strip added along $\al$ is hyperbolic.
	Let $w_\al(p)$ be the width of the strip at the point $p\in \al_g$. Let $\wt{\al_g}, \wt\pal,\wt p$ be the lifts of $\al_g,\pal,p$ such that $\wt p,\wt\pal\in \wt{\al_g}$. Suppose that $v_{\wt\al}$ is the Killing field acting across $\wt{\al_g}$ due to the strip deformation. Then, $\norm{v_{\wt\al}}=\wal$.
	
	In the hyperboloid model $\HP$, suppose that $v_{\wt\al}=(\wal,0,0)$ and let the plane containing $\wt{\al_g}$ be $\{(x,y,z)\in \R^3 \mid y=0\}$. So, $\wt\pal=(0,0,1)$. A point $p$ on the geodesic $\wt{\al_g}$ is of the form $(x,0,\sqrt{x^2+1})$, with $x\in \R$. Then we have 
	\begin{align}
		w_\al(p)=\norm{(v_{\wt\al}\mcp p}=\wal\sqrt{x^2+1}=-\wal\langle p, \wt\pal \rangle=\wal\cosh d_{\HP}(p,\wt\pal).
	\end{align}
	Now suppose that the arc $\al$ is joining a decorated spike and a boundary component of a decorated polygon. Then the infinitesimal strip added by $\isda$ is parabolic. Let $v_{\wt\al}=(\wal,0,\wal)$ be the corresponding parabolic Killing field. Then, 
	\begin{align*}
		w_\al(p)=\norm{v_{\wt\al}\mcp p}=\wal(\sqrt{x^2+1}-x).
	\end{align*}
	Let $L$ be the linear coordinate along the arc $\al$ such that $L<0$ if $p$ lies between the points $v_{\wt\al}$, $\pal$ and $L>0$ if $\pal$ lies between the points $v_{\wt\al}$ and $p$. Taking $x=\sinh L$ we get,
	$w_\al(p)=\e^{L}$.
	
	The point $\pal$ is called the point of \emph{minimum impact} because $w_\al(p_{\al})=\wal$.
	
	\begin{defi}
		Let $\Pi$ be a polygon with (possibly decorated) spikes with a metric $m$ and corresponding strip template $\{\st \}$. Let $x=\sum_{i=1}^{N_0} c_i\al_i$ be a point in the pruned arc complex $\sac \Pi$. Then the \emph{strip width function} is defined as:
		\[ 
		\begin{array}{rrcl}
			w_x:&\supp x&\longrightarrow&\R_{>0}\\
			& p&\mapsto&c_iw_{\al_i}(p),
		\end{array}
		\]
	\end{defi}
	\paragraph{Normalisation:} 
	Let $\Pi$ be a possibly decorated polygon and $\mathcal{K}$ be the set of permitted arcs. Then for every $\al\in \mathcal{K}$, we choose $\wal>0$ such that the following equality holds for every $x\in \sac \Pi$: 
	\begin{equation}\label{norma}
		\sum\limits_{p\in \partial \Pi\cap \supp x} w_x(p)=1.
	\end{equation}

	\begin{lem}\label{lenderiv}
		Let $\Pi$ be a decorated polygon endowed with a decorated metric $m$ and a corresponding strip template $\st$. Let $x\in \sac \Pi$ and $\ga$ be an edge or a diagonal of $\Pi$ instersecting $\supp x$. Then, 
		\begin{equation}\label{lengthderiv}
			\mathrm{d}l_{\ga}(f(x))=\sum\limits_{p\in \ga \cap \supp x} w_x(p) \sin \angle_p( \al_g, \supp x) > 0.
		\end{equation}  
	\end{lem}
\begin{proof}
Let $\supp x$ contain only one arc $\al$. Consider the universal cover of the surface inside the hyperboloid model of $\HP$. Suppose that a lift $\tilde{\ga}$ of $\ga$ is the horoball connection joining the two light-like points $u=(0,y,y), v=(0,-y',y')$, with $y,y'>0$. Then the length of $\tilde{\ga}$ is given by \[ l(\tilde{\ga})=\ln -\frac{\ang{u,v}}{2}.   \] Suppose that $\al_g$ intersects $\ga$ at $p=(0,0,1)$ at an angle $\angle_p( \al_g,\ga):=\theta\leq\frac{\pi}{2}.$ \\
 \begin{figure}
	\centering
	\frame{\includegraphics[width=9cm]{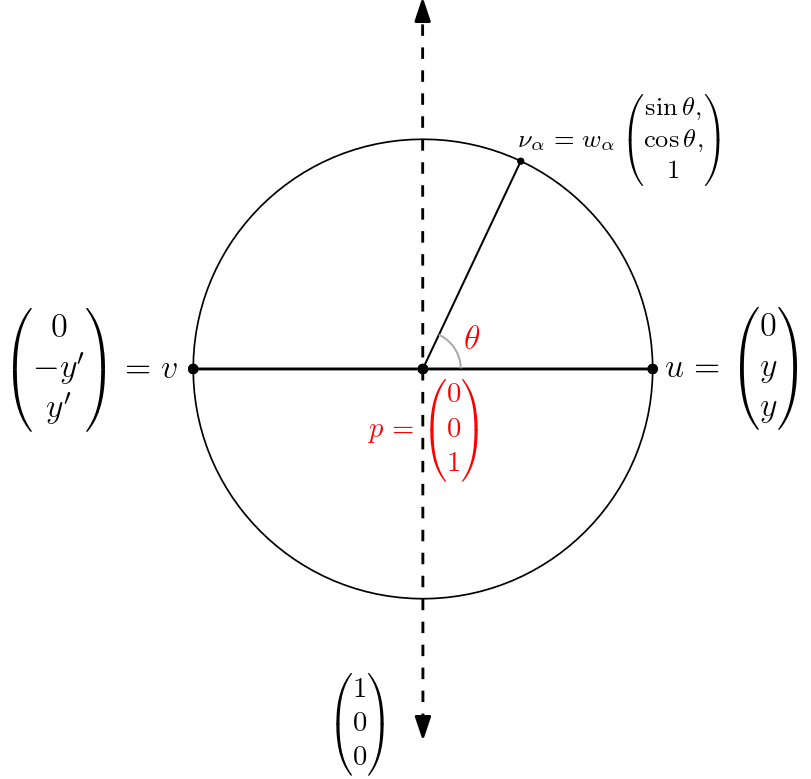}}
	\caption{Parabolic Killing field acting on a horoball connection}
	\label{parakvf}
\end{figure}
Firstly we consider the case when the arc $\al$ is of infinite type joining a spike and a boundary component of $\Pi$. Then the Killing field corresponding to the parabolic strip deformation $\isda$ along $\al_g$ with strip template $\st$ is given by $\nu_{\al}=(\wal\sin\theta,\wal\cos\theta,\wal)$. We need to show that \[ \mathrm{d}l_{\ga}(\isda)=\wal(p)\sin\theta=\wal\sin\theta.\]
The Killing field $\nu_{\al}$ pushes $v$ in the direction $w\mcp v$. 
The flow of the action of $\nu_{\al}$ on $v$ is given by \begin{align*}
	v_t&=v+t \nu_{\al}\mcp v + o(t)\\
	&=(0,-y',y')+t\wal y'(1+\cos\theta,-\sin\theta,\sin\theta)+ o(t)\\
	&=(t\wal y'(1+\cos\theta),-y'-t\wal y'\sin\theta,y'+t\wal y'\sin\theta) + o(t).
	\end{align*}
So the length of the horoball connection $\ga_t$ joining $u$ and $v_t$ is given by
\begin{align*}
l(\ga_t)&=\ln - \frac{\ang{u,v_t}}{2}\\
&= \ln yy'(1+t\wal \sin\theta)\\
&=\ln yy' + t\wal \sin\theta + o(t).
\end{align*}
Hence $\mathrm{d}l_{\ga}(\isda)=\frac{\mathrm d}{\mathrm {dt}}|_{t=0}  l(\ga_t)= \wal \sin\theta.$ \\
 \begin{figure}
 	\centering
 \frame{\includegraphics[width=9cm]{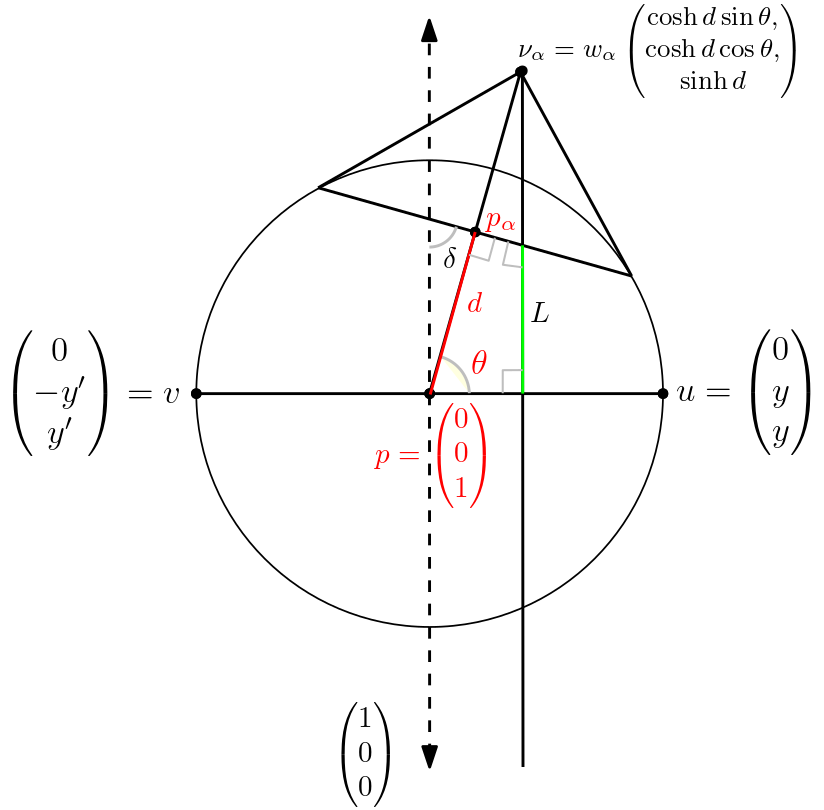}}
 \caption{Hyperbolic Killing field acting on a horoball connection}
 \label{hypkvf}
 \end{figure}
 Let us now suppose that $\al$ is a finite arc so that $\isda$ is a hyperbolic strip deformation with strip template $\st$. Let $d=d(p,p_\al)$ be the distance between the point of intersection and the waist. We need to prove that \[ \mathrm{d}l_{\ga}(\isda)=\wal(p)\sin\theta=\wal\cosh d\sin\theta.\] Then the Killing vector field corresponding to the strip deformation is given by \[\nu_{\al}=(\wal\cosh d\sin\theta, \wal\cosh d\cos\theta,\wal \sinh d).\] See Fig.\ref{hypkvf}
 We have, \begin{align*}
 	v_t&=v+t \nu_{\al}\mcp v + o(t)\\
 	&=(0,-y',y')+t\wal y'(\sinh d+\cosh d\cos\theta,-\cosh d\sin\theta,\cosh d\sin\theta)+ o(t)\\
 	&=(t\wal y'(\sinh d+\cosh d\cos\theta),-y'-t\wal y'\cosh d\sin\theta,y'+t\wal y'\cosh d \sin\theta) + o(t).
 \end{align*}
So the length of the horoball connection $\ga_t$ joining $u$ and $v_t$ is given by
\begin{align*}
	l(\ga_t)&=\ln - \frac{\ang{u,v_t}}{2}\\
	&= \ln yy'(1+t\wal\cosh d \sin\theta)\\
	&=\ln yy' + t\wal \cosh d \sin\theta + o(t).
\end{align*}
Hence $\mathrm{d}l_{\ga}(\isda)=\frac{\mathrm d}{\mathrm {dt}}|_{t=0}  l(\ga_t)= \wal\cosh d \sin\theta.$ \\
Finally, by linearity, we get the result for the general case with multiple arcs and intersection points.

\end{proof}

	\subsection{Summary of strip deformations of compact surfaces}\label{recap}
In this section, we will recall the statement of the parametrisation theorem proved by Danciger-Guéritaud-Kassel in \cite{dgk} for compact surfaces with totally geodesic boundary. We shall also give an idea of their proof, whose methods are going to be adapted to our case of surfaces with spikes. 

Let $S$ be a compact hyperbolic surface with totally geodesic boundary. Recall that when the surface is orientable (resp.\ non-orientable), it is of the form $S_{g,n}$ (resp.\ $T_{h,n}$), where $g$ is the genus (resp.\ $h$ is the total number of copes of projective plane) and $n$ is the total number of boundary components. Its deformation space $\tei {S}$ is homeomorphic to an open ball of dimension $N_0=6g-6+3n$ when $S$ is orientable and $N_0=3h-6+3n$ when $S$ is non-orientable. A point $m$ of the deformation space is expressed as $m=[\rho]$, where $\rho:\fg{S}\rightarrow\pgl$ is a holonomy representation of the surface. 
Given such an element $m\in \tei {S}$, its admissible cone $\adm m$ is the set of all infinitesimal deformations that uniformly lengthen every non-trivial closed geodesic. It is an open convex cone of the vector space $\tang{S}$. 

The arcs that are used to span the arc complex $\ac{S}$ of such a surface, are finite, non self-intersecting and their endpoints lie on boundary $\partial S$, like in the case of undecorated polygons. The pruned arc complex $\sac{S}$ of the surface $S$, given by the union of the interiors of all filling simplices, is an open ball of dimension $N_0-1$. Any point $x$ in $\sac{S}$ belongs to the interior of a unique filling simplex $\sigma_x.$ 

The strip deformations performed along the arcs are of hyperbolic type; their waists and widths are fixed by the choice of a strip template $\{\st \}_{\al\in\mathcal K}$. The infinitesimal strip map is given by  \[
\begin{array}{ccrcl}
	f& : &	\ac {S} & \longrightarrow & \tang {S}\\
	& &x=\sum\limits_{i=1}^{N_0} c_i \al_i&\mapsto&\sum\limits_{i=1}^{N_0}c_i f_{\al_i}(m),
\end{array}
\]
where $c_i\in [0,1]$ for every $i=1,\ldots,N_0$, and $\sum\limits_{i=1}^{N_0} c_i=1$.
Then the following result was proved in \cite{dgk}:

\begin{thm}\label{cpt}
	Let $S=S_{g,n}$ or $T_{h,n}$ be a compact hyperbolic surface with totally geodesic boundary. Let $m=([\rho])\in \tei {S}$ be a metric. Fix a choice of strip template $\{\st \}_{\al\in\mathcal K}$ with respect to $m$. Then the restriction of the projectivised infinitesimal strip map $\mathbb{P}f:\sac {S}\longrightarrow \ptan {S}$ is a homeomorphism on its image  $\mathbb{P}^+(\adm m)$.
\end{thm}
\paragraph{Structure of the proof.} Firstly, they show that the image of the map $\mathbb{P}f$ is given by the positively projectivised admissible cone. Since both the pruned arc complex and $\mathbb{P}^+(\adm m)$ are homeomorphic to open balls of the same dimension, it is enough to show that $\mathbb{P}f$ is a covering map. A classical result from topology states that a continuous map between two manifolds is a covering map if the map is proper and also a local homeomorphism. So the authors prove that the projectivised strip map $\mathbb P f$ satisfies these two properties. 

Firstly, they show the following theorem.
\begin{thm}\label{proper}
	The projectivised strip map $\mathbb{P}f:\sac {S}\longrightarrow \mathbb{P}^+(\adm m)$ is proper. 
\end{thm}
Secondly, they show that the map $\mathbb{P}f$ is a local homeomorphism around points $x\in \sac{S}$ such that $\codim{\sigma_x}\leq 2$, and then around points such that $\codim{\sigma_x}\geq 2$ by induction. This is done in the following steps:
\begin{itemize}
\item For points belonging to the interior of simplices with codimension 0, it is enough to show that the $f$-images of the vertices of any top-dimensional simplex form a basis in the deformation space of the surface. \begin{thm}
	\label{basis} 	Let $S$ be a compact hyperbolic surface with totally geodesic boundary, equipped with a metric $m\in \tei {S}$. Let $\sigma$ be a codimension zero simplex and let $\ed$ be the corresponding edge set. Then the set of infinitesimal strip deformations $B=\{   \isd| e\in \ed  \}$ forms a basis of $\tang {S}$.
\end{thm}
\item Next let $x\in \sac{S}$ such that $x\in \inte{\sigma_x}$ where $\sigma_x$ is a filling simplex with $\codim{\sigma_x}=1$. Since $\sac S $ is an open ball, there exist two simplices $\sigma_1,\sigma_2$ such that 
\begin{itemize}
	\item $\codim{\sigma_1}=\codim{\sigma_2}=0$,
	\item $\sigma_x=\sigma_1\cap\sigma_2$.
\end{itemize} The following theorem gives a sufficient condition for proving local homeomorphism of the projectivised strip map around points like in this case.
\begin{thm} \label{disjint}
	Let $S$ be a compact hyperbolic surface with totally geodesic boundary, equipped with a metric $m\in \tei {S}$. Let $\sigma_1 ,\sigma_2 \in \ac {S}$ be two top-dimensional simplices such that $$\codim {\sigma_1 \cap \sigma_2}=1\text{  and }\inte{\sigma_1 \cap \sigma_2}\subset \sac {S}.$$ Then we have that,
	\begin{equation}
		\inte{\mathbb{P}f(\sigma_1)}\cap \inte{\mathbb{P}f(\sigma_2)}=\varnothing.
	\end{equation}
\end{thm}
	
	\item The case for $\codim{\sigma_x}=2$ follows from the following theorem and lemma:
	
	\begin{thm}\label{codim2}
		Let $S$ be a compact hyperbolic surface with totally geodesic boundary, equipped with a metric $m\in \tei {S}$. Let $\sigma_1,\sigma_2$ be two simplices of its arc complex $\ac {S}$ satisfying the conditions of Theorem \ref{disjint}. Then there exists a choice of strip template such that $\mathbb{P}f(\sigma_1)\cup \mathbb{P}f(\sigma_2)$ is convex in $\ptan {S}$.
	\end{thm}
	\begin{lem}\label{cd2link}
		Let $S$ be a compact hyperbolic surface with totally geodesic boundary, equipped with a metric $m\in \tei {S}$.  Let $x\in\ac {S}$ such that $\codim{\sigma_x}= 2$. Then, $\mathbb{P}f|_{\Link{\ac {S}}{\sigma_p}}$ is a homeomorphism.
	\end{lem}
\begin{proof}[Idea of the proof of Lemma \ref{cd2link}:] Since $\codim{\sigma_x}=2$, there is space to put two more arcs that are disjoint from all the arcs of $\sigma_x$. There are two possibilities:
	\begin{itemize}
		\item there exist exactly two disjoint regions (hyperideal quadrilaterals) in the complement of $\supp x$; every other connected component is a hyperideal triangle. Each of these regions can be decomposed into hyperideal triangles in two ways by a diagonal exchange such that the exchanges are independent of each other. So the sub-complex $\Link{\ac {S}}{\sigma_x}$ of $\ac{S}$ is a quadrilateral in this case. 
		\item there exists exactly one region  in the complement of $\supp x$, which is not a hyperideal triangle. This region can be decomposed into three hyperideal triangles by two additional arcs that are pairwise disjoint from the rest. These two arcs can be chosen in 5 ways, using the "pentagonal moves". As a result, the sub-complex $\Link{\ac {S}}{\sigma_x}$ is a pentagon in this case.
	\end{itemize}
	So in both the cases, the restriction of the projectivised infinitesimal strip map to the link gives a P-L map $$\mathbb{P}f|_{\Link{\ac {S}}{\sigma_p}}: \s1\longrightarrow \s1.$$ Using Theorem \ref{codim2}, the authors prove that this map has degree one, which proves it to be a homemorphism.
\end{proof}
\item Finally the cases $\codim{\sigma_x}\geq2$ follow from the following theorem:

\begin{thm}
	Let $x\in\sac {S}$ such that $\codim{\sigma_x}\geq 2$. Let $V\subset \tang{S}$ be the vector subspace generated by the infinitesimal strip deformations $\{f_{\al}(m)\}_{\al\in \sigma_x^{(0)}}$. Then, the restriction map $$\mathbb{P}f: \Link{\ac {S}} {\sigma_x}\longrightarrow \mathbb{P}^+({\tang {S}}/V)$$ is a homeomorphism. 
\end{thm}
We recall the proof of the above theorem as done in \cite{dgk}. We will use the same reasoning for our surfaces with spikes.
\begin{proof}
	Firstly, we note that the $V$ is a subspace of dimension $N_0-2$ because from Theorem \ref{basis} we get that $\{f_{\al}(m)\}_{\al\in \sigma_x^{(0)}}$ is linearly independent. So the space $\mathbb{P}^+({\tang {S}}/V)$ is homeomorphic to $\s{\codim{\sigma_x}-1}$.
	The statement is verified for $\codim{\sigma_x}=2$. Suppose that the statement holds for $2,\ldots,d-1$. We need to show that $$\mathbb{P}f|_{\Link{\ac {S}} {\sigma_x}}:\s{d-1}\longrightarrow \s{d-1}$$ is a local homeomorphism. Let $x\in \Link{\ac {S}} {\sigma_x}$. Then $x$ is contained in the interior of a simplex $\sigma_x$ whose codimension in $\Link{\ac {S}} {\sigma_x}$ is $d-1-\dim \sigma_x$, which is less than $d$. So by induction hypothesis, the map  $\mathbb{P}f|_{\Link{\ac {S}} {\sigma_x}}$ restricted to $\Link{\mathbb{P}f|_{\Link{\ac {S}} {\sigma_x}}}{\sigma_x}$ is a homeomorphism. This proves that $\mathbb{P}f|_{\Link{\ac {S}} {\sigma_x}}$ is a local homeomorphism. Since $\s{d-1}$ is compact and simply-connected for $d\geq 3$, it follows that $\mathbb{P}f|_{\Link{\ac {S}} {\sigma_x}}$ is a homeomorphism. 
\end{proof}

\end{itemize}

\section{Parametrisation of infinitesimal deformations of polygons}\label{proof}
	The goal of this section is to prove our parametrisation theorems for four types of polygons — ideal polygons, ideal once-punctured polygons, decorated polygons and decorated  once-punctured polygons. 
	Let $\Pi$ be the surface of any of these polygons and let $N_0:=\dim\tei\Pi$. 
	Recall from Definition \ref{ism} that the projectivised infinitesimal strip map for a fixed $m\in \tei \Pi$ is defined as:
	\[
	\begin{array}{ccrcl}
		\mathbb{P}f& : &	\ac \Pi & \longrightarrow & \ptan {\Pi}\\
		& &\sum\limits_{i=1}^{N_0} c_i \al_i&\mapsto&\bra{\sum\limits_{i=1}^{N_0}c_i f_{\al_i}(m)} 
	\end{array}
	\] where for every $i=1,\ldots,N_0$, $c_i\in [0,1]$ and $\sum\limits_{i=1}^{N_0} c_i=1$.
	The rest of the section is dedicated to proving the following four theorems, which constitute our main contribution.
	\begin{thm} \label{Mainideal}
		Let $\ip n$ ($n\geq 4$) be an ideal $n$-gon with a metric $m\in \tei {\ip n}$. Fix a choice of strip template. Then,	the infinitesimal strip map $$\mathbb{P}f:\ac {\ip n}\longrightarrow \ptan {\ip n}$$ is a homeomorphism.
	\end{thm}
	
	\begin{thm} \label{Mainpunc}
		Let $\punc n$ ($n\geq 2$) be an ideal once-punctured $n$-gon with a metric $m\in \tei {\punc n}$. Fix a choice of strip template. Then,	the infinitesimal strip map $$\mathbb{P}f:\ac{\punc n} \longrightarrow \ptan {\punc n}$$ is a homeomorphism.
	\end{thm}
	\begin{thm} \label{Maindeco}
		Let $\dep n$ ($n\geq 3$) be a decorated $n$-gon with a metric $m\in \tei {\dep n}$. Fix a choice of strip template. Then the infinitesimal strip map $\mathbb{P}f$, when restricted to the pruned arc complex $\sac {\Pi}$, is a homeomorphism onto its image $\mathbb{P}^+(\adm m)$, where $\adm m$ is the set of infinitesimal deformations that lengthens all edges and diagonals of the polygon.
	\end{thm}
	\begin{thm} \label{Maindecopunc}
	Let $\depu n$ ($n\geq 2$) be a decorated once-punctured polygon with a metric $m\in \tei {\depu n}$. Fix a choice of strip template. Then the infinitesimal strip map $\mathbb{P}f$, when restricted to the pruned arc complex $\sac {\depu n}$, is a homeomorphism onto its image $\mathbb{P}^+(\adm m)$, where $\adm m$ is the set of infinitesimal deformations that lengthens all edges and diagonals of the polygon.
\end{thm}
	
	\paragraph{Idea of the proofs.} Each of the proofs of the four theorems follows the same strategy as discussed at the end of Section \ref{Sd}. Firstly, we show that the map $\psm$ is a local homeomorphism. 
	Since the sphere is compact, we have that $\psm$ is a covering map for the first two cases — $\ip n$, $\punc n$. Finally, for $n\geq 6$ (ideal $n$-gon) and $n\geq 4$ (punctured $n$-gon), the spheres $\s{n-4}$ and $\s{n-2}$ are simply-connected, so the maps are homeomorphisms. The cases $\ip 4,\ip5,\punc 2,\punc3$ will be treated separately. For the decorated polygons $\dep n, \depu n$, we show properness in order to get a covering map. Their arc complexes are contractible, hence we get a global homeomorphism.
	
	Let $\Pi$ be the topological surface of any hyperbolic polygon. Every point $p\in \ac \Pi$ belongs to a unique open simplex, denoted by $\sigma_p$. Like in \cite{dgk}, we prove that $\psm$ is a local homeomorphism for points $p$ such that $\codim{\sigma_p}=0,1,2$ and for $p$ with $\codim{\sigma_p}\geq 2$, the proof is by induction.
	\subsection{Local homeomorphism: codimension 0 faces}
	In this section, for each of the four types of polygons, we shall prove the local homeomorphism of the projectivised strip maps around points that belong to the interior of codimension 0 simplices in their respective arc complexes.
	\subsubsection{Ideal polygons}
	\begin{thm} \label{thmip}Let $m\in\tei {\ip n}$ be a metric on an ideal $n$-gon $\ip n$, with $n\geq 4$. Fix a choice of strip template. Let $\sigma$ be a top-dimensional simplex of its arc complex $\ac{\ip n}$ and let $\ed$ be the corresponding edge set. Then the set of infinitesimal strip deformations $B=\{   \isd| e\in \ed  \}$ forms a basis of the tangent space $\tang{\ip n}$.
	\end{thm}
	\begin{proof}
		Since $\dim \tang \Pi=\#\ed=n-3$, it is enough to show that the set $B$ is linearly independent. We proceed by contradiction: suppose that there exists reals $c_{e}$, not all equal to 0, such that
		\begin{equation}\label{lindep}
			\sum_{e\in \ed} c_{e}f_{e}(m)=0.
		\end{equation}
		Then we get an equivalence class of tile maps, up to an additive constant in $\lalg$, which do not deform the polygon. 
		From this class, we can choose a neutral tile map $\phi_0:\tile\rightarrow \lalg$ (see definition \ref{neutraltilemap} in Section \ref{Sd}), which fixes all ideal vertices of the tiles in $\tile$.
		The following lemma finds a permitted region for the $[\np d]$ any type of tile $d$. 
		
		\begin{lem}\label{lemideal} Let $\sigma$ be a top-dimensional simplex of $\ac{\ip n}$. Let $\phi_0:\tile\rightarrow \lalg$ be a neutral tile map corresponding to the linear combination eq.\eqref{lindep}. Let $e\in \ed$ be an internal edge of a tile $d \in \tile$ such that $\np d\neq 0$. Then the point $[\np d]\in \pp$ lies in the interior of the projective triangle, based at the infinite geodesic $\ol e$ carrying $e$, that contains the tile $d$.
		\end{lem}
		
		\begin{proof}
			Consider the dual graph of the triangulation of the surface by the top-dimensional simplex $\sigma$. It is a tree the valence of whose vertices is at most 3. Let $\tau$ be the sub-tree spanned by the tiles that are on the same side of $e$ as $d$. 
			Define $M(d)$ as the length of the longest path in $\tau$ joining $d'$ and a leaf (quadrilateral). The lemma will be proved by induction on $M$.
			When $M(d)=0$, the tile $d$ is a quadrilateral. The neutral tile map $\phi_0$ fixes the two ideal vertices of $d$. Applying Corollary \ref{tang} to these vertices, we get that $[\np d]$ is the point of intersection of the tangents to $\HPb$ at these ideal vertices. Lastly, the convexity of $\HPb$ implies that $[\np d]$ lies in the interior of $\Delta$. 
			
			Next, we suppose the statement to be true for $M(d)=0,\ldots,k$. Let $d\in \ltile$ be a tile such that $M(d)=k+1$. Then the tile $d$ can be either a hexagon or a pentagon because a quadrilateral has only one neighbouring tile and it must lie outside the triangle $\Delta$. We will treat the two cases separately below:
			\begin{itemize}
				\item If $d$ is a hexagon, then apart from $e$, it has two other internal edges $e', e''\in \ed$ along which $d$ neighbours two tiles $d', d''$, respectively. We note that both $d',d''$ lie inside $\Delta$. 
				\begin{itemize}
					\item Suppose that both $\np{d'}, \np{d''}$ are non-zero. See Fig.\ \ref{bothnonzero}. Denote by $\oarr{t_1},\oarr{t_2},\oarr{t_3},\oarr{t_4}$, the tangents to $\HPb$ at the endpoints $P,Q$ and $R,S$ of $e',e''$, respectively. Label the following points
					\[
					\begin{array}{ccc}
						X:=\oarr {t_1}\cap \oarr{t_2}, & Y:=\oarr {t_3}\cap \oarr{t_4},&A:=\oarr {t_2}\cap \oarr{t_3},\\
						B:=\oarr {t_1}\cap \oarr{t_3},&C:=\oarr {t_1}\cap \oarr{t_4}, & D:=\oarr {t_2}\cap \oarr{t_4}.
					\end{array}
					\]
					By the induction hypothesis, the points $[\np{d'}]$ and $[\np{d''}]$ lie inside the projective triangles $\Delta PQX$ and $\Delta RYS$ that contain $d'$ and $d''$, respectively. Since these two triangles are disjoint, $\np d$ cannot be equal to $\np{d'}$ as well as $\np{d''}$. In other words, the coefficients  $c_{e'}, c_{e''}$ in $\eqref{lindep}$ cannot be simultaneously equal to zero. Without loss of generality, suppose that $c_{e'}=0\neq c_{e''}$. So, $\np d=\np {d'}\neq \np {d''}$. Consequently, $\np d$ lies inside $\Delta PQX$ and $\np{d}-\np{d''}$ is a hyperbolic Killing vector field whose projective image lies on $\oarr{e''}\backslash \cHP$, i.e, the straight line joining the points $[\np{d}]$ and $[\np{d''}]$ intersects $e''$ outside $\cHP$. Using Property \ref{line} for $e''$, we know that the $[\np{d}]$ must be contained in the region $K_1:=\pp$ $\backslash \Delta''$ where $\Delta''$ is the projective triangle based at $e''$ that does not contain $d''$.  
					Since $\HPb$ is convex, $\Delta''$ is disjoint from $K_1$, which implies that $\np d\neq \np d'$, which is a contradiction. So we must have $c_{e'}\neq 0$. Using the same argument as in the case of $e''$, we get that $[\np d]$ lies in the region $K_2:=\pp$ $\backslash \Delta'$, where $\Delta'$ is the projective triangle based at $\ol {e'}$, not containing $d'$.
					Hence, the point $[\np d]$ must lie inside the intersection $R_1\cap R_2$, which is the quadrilateral $ABCD$ entirely contained in $\Delta\backslash e$, as required.
					
					\begin{figure}[!h]
						\begin{center}
							\frame{\includegraphics[height=8cm]{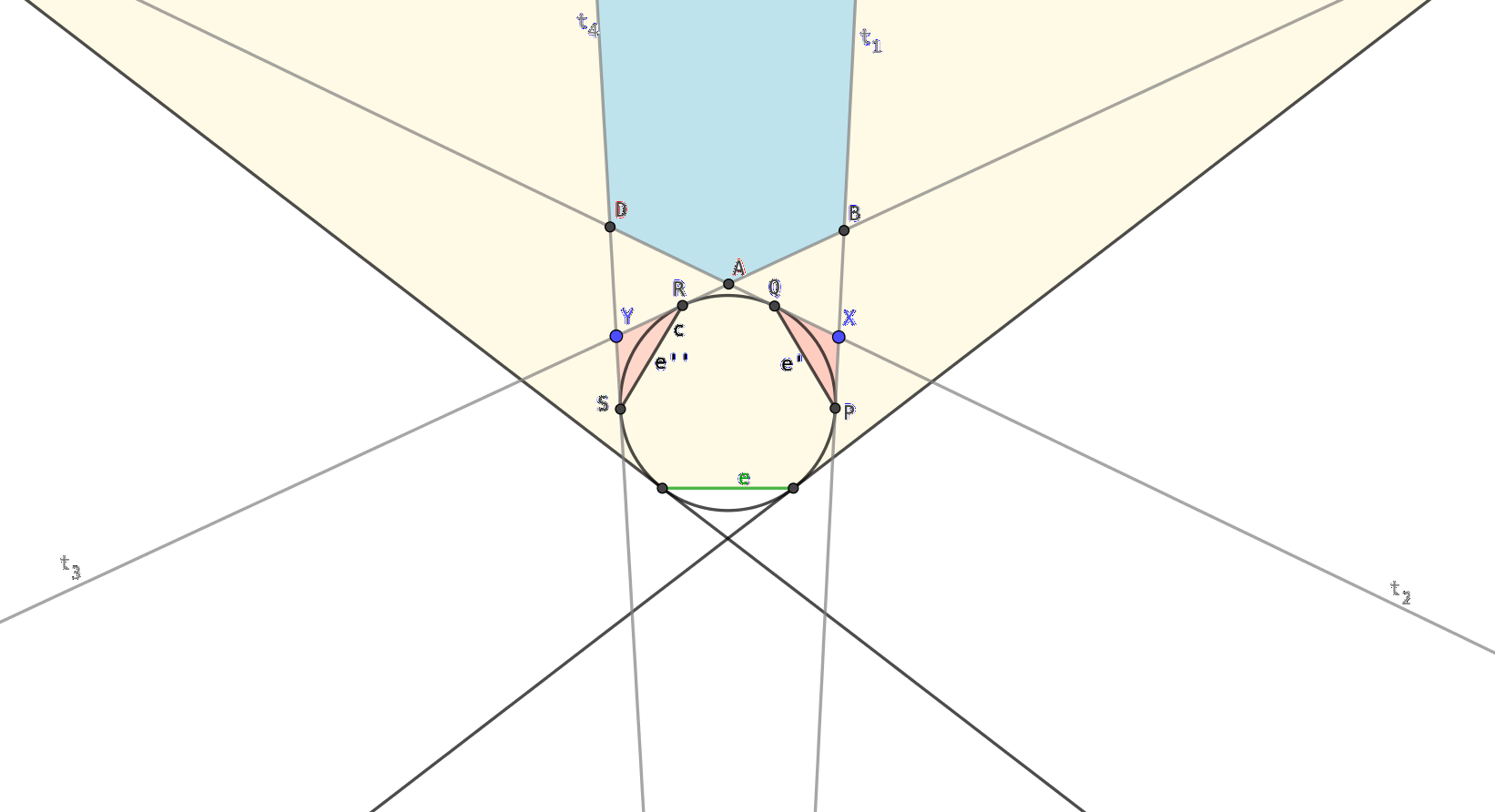}}
							\caption{$\np{d'}, \np{d''}\neq 0$}
							\label{bothnonzero}
						\end{center}
					\end{figure}
					\item Next we suppose that $\np {d'}= 0$ and $\np {d''}\neq0$. See Fig.\ \ref{onenonzero}. Again, by using the induction hypothesis on the tile $d''$ and the edge $e''$, we get that $\np{d''}$ lies in the triangle $\Delta RYS$, containing $d''$. 
					Using the same argument and notation of the previous case, we have that the region where the point $[\np d]$ must lie so that the straight line joining $[\np d]$ and $[\np {d''}]$ intersects $\oarr{e''}$ outside $\cHP$, is given by $K_1$. Label the points $\oarr{e'}\cap \oarr{t_3}, \oarr{e'}\cap \oarr{t_4}$ as $T,O$, respectively. Since $\np d\neq 0$, the coefficient $e'$ is non-zero. So, $[\np d]\in \oarr{e'}\backslash \cHP$. Hence, the point $[\np d]$ must lie in the intersection $(\oarr{e'}\backslash \cHP )\cap K_1$ which is a segment (coloured blue in the figure) completely contained inside $\Delta$. 
					\begin{figure}[!h]
						\begin{center}
							\frame{\includegraphics[height=8cm]{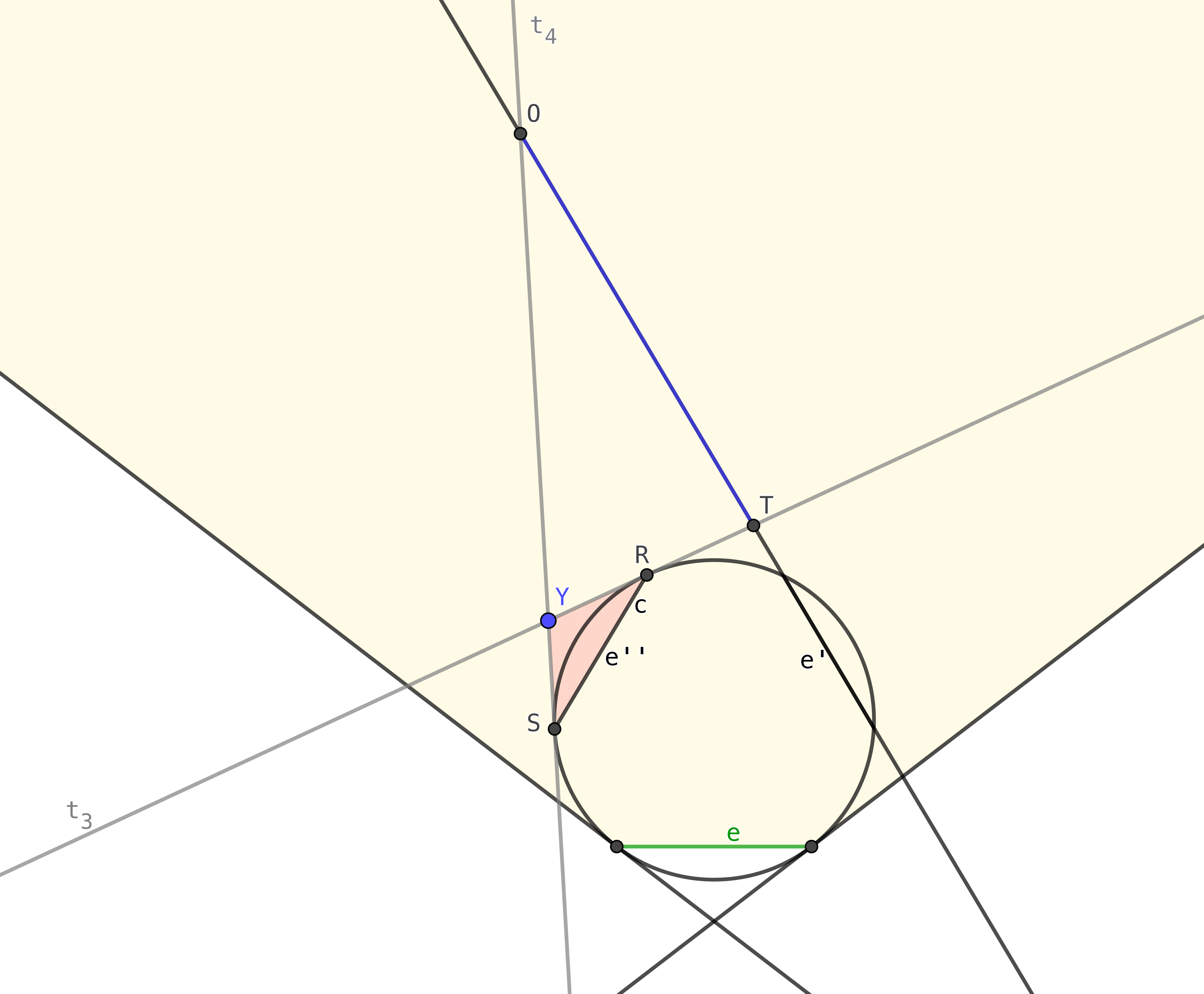}}
							\caption{$\np{d'}= 0$}
							\label{onenonzero}
						\end{center}
					\end{figure}
					\item Finally, we suppose that $\np {d'}=0=\np {d''}$. Again, $\np d\neq 0$ implies that $c_{e'}, c_{e'' }\neq 0$. Then the point $[\np \del]$ is given by the intersection of the two straight lines $\oarr{e'}, \oarr{e''}$. Since $e',e''$ are disjoint, the intersection point is hyperideal and lies inside $\Delta\backslash e$.
					
					\begin{figure}[h]
						\begin{center}
							\frame{\includegraphics[height=10cm]{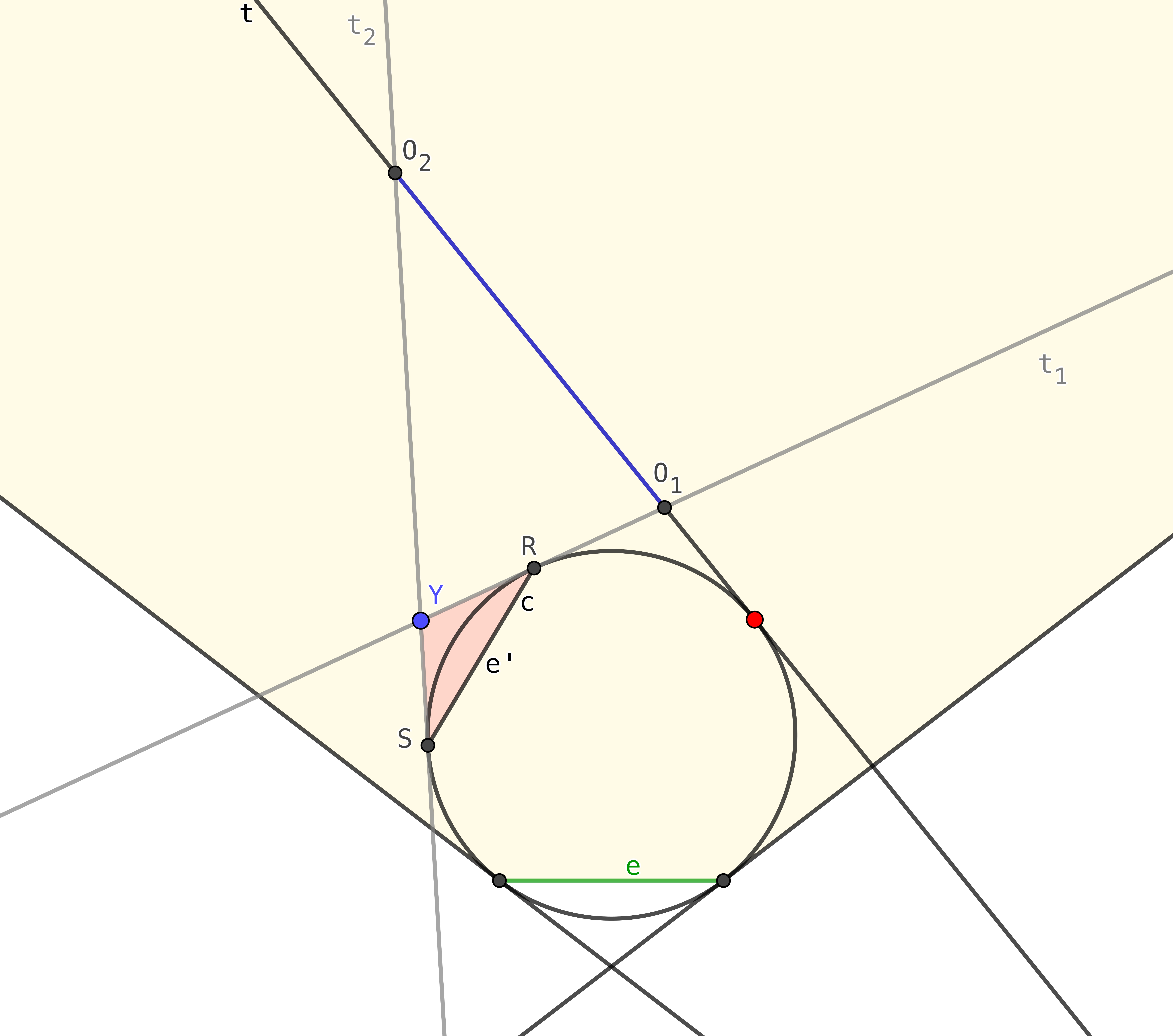}}
							\caption{$d$ is a pentagon, $\np{d'}\neq 0$}
							\label{pentnon}
						\end{center}
					\end{figure}
					
				\end{itemize} 
				\item If $d$ is a pentagon, then Corollary \ref{tang} implies that $\np d$ must lie on the tangent $\oarr t$ to the ideal vertex of $d$. Also, this tile has exactly one neighbour $d'$ that is contained in $\Delta$. Let $e'\in \ed$ be the common internal edge  of $d,d'$.
				\begin{itemize}
					\item If $\np {d'}=0$, then $[\np d]\in \oarr{e'}$. So we have  $[\np d]=\oarr{e'}\cap \oarr t$, which lies inside $\Delta$, by convexity of $\HPb$. 
					\item If $\np {d'}\neq0$, then by the induction hypothesis, $[\np {d'}]$ lies inside the projective triangle $\Delta'$ based at $e'$ that doesn't contain $d'$. See Fig.\ \ref{pentnon} Again by Property $\ref{line}$, the point $[\np d]$ is contained in the region $K:=\pp\backslash \Delta '$.  Let $\oarr{t_1}, \oarr{t_2}$ be the tangents to $\HPb$ at the endpoints of $\ol {e'}$. Label the points $\oarr {t_1}\cap \oarr t, \oarr {t_2}\cap \oarr t$ by $O_1, O_2$ respectively. Then ${\np d}$ is contained in the segment $\ol {O_1O_2}$, which lies in the interior of $\Delta$.
				\end{itemize}
			\end{itemize}
			This proves the induction step and hence the lemma for ideal polygons. 
		\end{proof}
		Now, we come back to the proof of the theorem. Let $e \in \ed$ be an arc such that $c_{e} \neq 0$. Let $d,d'$ be the two tiles with common edge $e$. Then, $\np d\neq \np {d'}$, and the point $[\np d- \np {d'}]$ belongs to $\oarr e\backslash \cHP$. Let $\Delta, \Delta'$ be the projective triangles based at $\ol e$. Let $d,d'$ be the tiles in $\tile$ neighbouring along $e$ such that $d\subset \Delta$ and $d'\subset \Delta'$. 
		
		If both $\np d, \np{d'}$ are non-zero, then the above lemma applied to the pairs $d,e$ and $d',e$ gives us that $[\np d]\in \inte{\Delta}$ and $[\np {d'}]\in \inte {\Delta '}$. Using \ref{line}, we get that the line joining $[\np d]$ and $[\np {d'}]$ intersects $\oarr e$ inside $\HPb$, which is a contradiction. 
		
		If $\np {d'}=0$, then $\np d \in \oarr e\backslash\cHP$, which is disjoint from the interior of $\Delta$. So we again reach a contradiction. Hence we must have $c_e=0$ for every $e=0$. 
		This concludes the proof.
	\end{proof}

\subsubsection{Punctured polygons}
\begin{thm} \label{thmpunc}
	Let $m\in\tei {\punc n}$ be a metric on an ideal once-punctured $n$-gon $\punc n$, with $n\geq 2$. Fix a choice of strip template. Let $\sigma$ be a top-dimensional simplex of its arc complex $\ac{\punc n}$ and let $\ed$ be the corresponding edge set. Then the set of infinitesimal strip deformations $B=\{\isd\mid e\in \ed\}$ forms a basis of the tangent space $\tang{\punc n}$.
\end{thm}
\begin{proof}
	Like in the case of ideal polygons, we have that $\dim \tang {\punc n}=\#\ed=n-1$. So we only need to prove the linear independence of $B$. Again we start with an equation as in \eqref{lindep} with a corresponding neutral map $\phi_0:\ltile\longrightarrow\lalg$. This map is $\rho(\ga)$-invariant, where $\ga$ is the generator of the fundamental group of the surface. So $\phi_0$ satisfies the following equation: 
	\begin{equation}\label{inv}
	\rho(\ga)	\cdot\np d=\np {\rho(\ga)\cdot d}, \text{ for every } d\in \ltile.
	\end{equation}
W assume that $\rho(\ga)$ is given by the matrix $T=\begin{pmatrix}
	1&1\\
	0&1
\end{pmatrix}$.  
	Recall from Section \ref{arc} that the permitted arcs generating the arc complex are finite arcs with their endpoints on the boundary. There is exactly one maximal arc $e_M$ (separates the puncture from the spikes) in every triangulation. The surface is decomposed into four types of tiles. The first three types (quadrilateral, pentagon, hexagon) are finite hyperbolic polygons and the fourth one is a tile containing the puncture. It lifts to a tile, denoted by $d_\infty$, with infinitely many edges, each given by a lift of the unique maximal arc $e_M\in \ed$ of the triangulation, and  exactly one ideal vertex, denoted by $p$ that corresponds to the puncture. See Fig. \ref{infini}.
	
	 	\begin{figure}[h]
	 	\begin{center}
	 		\includegraphics[width=12cm]{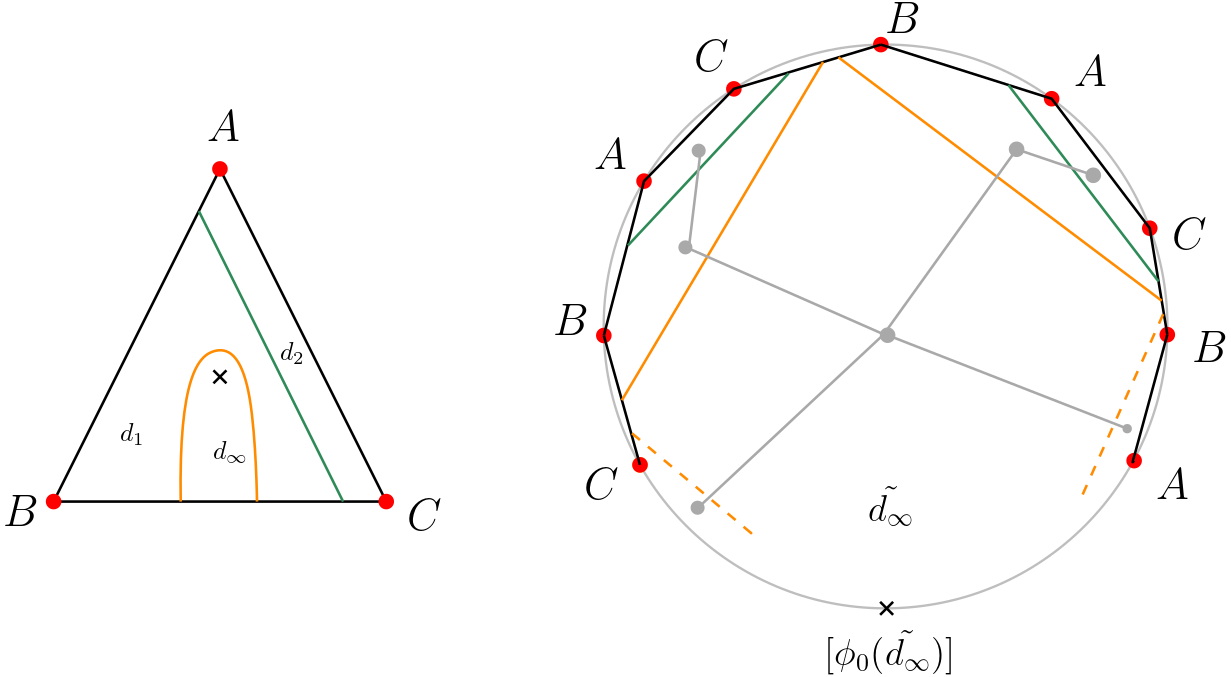}
	 		\caption{Punctured triangle and its universal cover.}
	 		\label{infini}
	 	\end{center}
	 \end{figure}
	
	Now, we show that the Killing field  $\np {d_{\infty}}$ associate to the unique infinite tile $d_\infty\in \ltile$, is either zero or a parabolic element with fixed point $p\in \HPb$ that corresponds to the puncture. We know that $d$ is invariant under the action of the isometry $T$: 
	\begin{equation}\label{invtile}
		\np {d_\infty}=\np {T^n\cdot d_\infty} \text{ for every } n\in\Z.
	\end{equation}
	Using the isomorphism between the Lie algebra $\lalg$ and $\Min$, we have that $\np d$ is represented by the matrix 
	$	\begin{pmatrix}
		y & x+z \\
		x-z & -y
	\end{pmatrix}$. The generator $T$ acts on $p$ by conjugation:
	\begin{align*}T\cdot \np {d_\infty}=&\begin{pmatrix}
			1&-1\\
			0&1
		\end{pmatrix}\begin{pmatrix}
			y & x+z \\
			x-z & -y
		\end{pmatrix}\begin{pmatrix}
			1&1\\
			0&1
		\end{pmatrix}\\
		=&\begin{pmatrix}
			y-x+z&2(y+z)\\
			x-z&x-z-y
		\end{pmatrix}
	\end{align*}
	From eqs.\ \eqref{inv} and \eqref{invtile}, we get that $y=0, x=z$. Hence, $\np {d_\infty}$ is either zero or a parabolic element, fixing the light-like line $\R p$ and $[\np {d_\infty}]=p$. 
	
	We now prove an analogous version of Lemma \ref{lemideal} for a punctured polygon.
	\begin{lem}\label{lempunc}
		Let $\sigma$ be a top-dimensional simplex of $\ac{\punc n}$. Let $\phi_0:\ltile\rightarrow \lalg$ be a neutral tile map corresponding to the linear combination \eqref{lindep}. Let $e\in \led$ be an internal edge of a tile $d \in \ltile$ such that $\np d\neq 0$. Then the point $[\np d]\in \pp$ lies in the interior of the projective triangle, based at the geodesic $\ol e$ carrying $e$, that contains the tile $d$.
	\end{lem}
	\begin{proof}
		Let $d \in \ltile$ such that $\np d\neq 0$ and let $e\in \led$ be an internal edge of $d$.
		Consider the dual graph of the triangulation of the universal cover of the surface by $\sigma$. It is an infinite tree invariant by the action of $\ang g$. It can be seen as the countable union of finite trees and rooted at the infinite tile $d_\infty$. The latter has infinitely many edges, each given by a lift of the unique maximal arc $e_M\in \ed$ of the triangulation. There are two possibilities — either $d=d_\infty$ or there exists a unique lift $\wt {e_M}$ that separates $d$ from $d_\infty$. 
		Let $\tau$ be the finite rooted sub-tree spanned by the tile $d_\infty$ and all those tiles that are separated by $\wt {e_M}$ from $d_\infty$. Define $M(d)$ as the length of the longest path on $\tau$ joining $d$ and a quadrilateral tile or the root tile $d_\infty$ such that the path does not cross the edge $e$ of $d$. Then the lemma is proved by induction on $M$. 
		
		When $M(d)=0$, $d$ is either a quadrilateral or the tile $d_\infty$. In the former case, we know that $\np d$ is a hyperbolic Killing field with fixed points as the two ideal vertices of the quadrilateral; the point $[\np d]$ is given by the intersection of the two tangents to the boundary circle $\HPb$ at the ideal vertices. So the lemma is verified in this case. Next we suppose that $d=d_\infty$. Then from the discussion before the lemma, we have that $[\np {d_\infty}]=p$ which lies inside the desired triangle. So the statement of the lemma is satisfied in this base case. 
		
		Now suppose that the statement is true for $M=1,\ldots, k$. Consider a tile $d$ inside $\tau$ such that $M(d)=k+1$. Then $d$ is either a pentagon with one ideal vertex and two internal edges (both finite) or a hexagon with three internal edges and no spikes. Also, there exists a finite path of length $k+1$ in the tree $\tau$ starting from $d$ and ending at a vertex which is either a quadrilateral or the root tile. By proceeding in the exact same way as in the induction step of Lemma \ref{lemideal} for ideal polygons, we get that the induction step is verified in this case well.
		This finishes the proof of the lemma.
	\end{proof}
	Now suppose that the coefficient $c_e$ of $\isd$ is non-zero for some $e\in \led$. Let $d,d'$ be the two tiles with common edge $e$. Then, $\np d\neq \np {d'}$, and the point $[\np d- \np {d'}]$ belongs to $\oarr e\backslash \cHP$. Let $\Delta, \Delta'$ be the projective triangles based at the geodesic carrying the arc $e$ such that $d\subset \Delta$ and $d'\subset \Delta'$. 
	
	If both $\np d, \np{d'}$ are non-zero, then the above lemma applied to the pairs $d,e$ and $d',e$ gives us that $[\np d]\in \inte{\Delta}$ and $[\np {d'}]\in \inte {\Delta '}$. Using \ref{line}, we get that the line joining $[\np d]$ and $[\np {d'}]$ intersects the projective line $\oarr e$ carrying the arc $e$ inside $\HPb$, which is a contradiction. 
	
	If $\np {d'}=0$, then $\np d \in \oarr e\backslash\cHP$, which is disjoint from the interior of $\Delta$. So we again reach a contradiction. 
	
	Hence, we have $c_e=0$ for every arc $e\in \led$, which proves Theorem\ref{thmpunc}. 
	\end{proof}
\subsubsection{Decorated Polygons}
Firstly we shall prove the linear independence in the case of decorated polygons without a puncture.
\begin{thm} \label{thmdeco}Let $m\in\tei {\dep n}$ be a metric on a decorated $n$-gon $\dep n$, with $n\geq 3$. Fix a choice of strip template. Let $\sigma$ be a top-dimensional simplex of its arc complex $\ac{\dep n}$ and let $\ed$ be the corresponding edge set. Then the set of infinitesimal strip deformations $B=\{\isd\mid e\in \ed\}$ forms a basis of the tangent space $\tang{\dep n}$.
\end{thm}
\begin{proof}
	Again, we have that  $\dim \tang {\dep n}=\#\ed=2n-3$. So, it is enough to show that the above set is linearly independent. Since every decorated polygon is simply connected, we have that $\ed=\led$ and $\tile=\ltile$. Suppose that $\sum_{e\in \ed} c_{e}f_{e}(m)=0$, with not all $c_{e}$'s equal to 0. Let $\phi_0:\ltile\rightarrow \lalg$ be a neutral tile map; by definition, it fixes the decorated vertices of every tile.
	Suppose a tile $d$ has a decorated vertex $\nu$ (Fig. \ref{decodim0}). The Killing field $\np d$ fixes the ideal point as well as the horoball decoration.  If $\np d\neq 0$, then the point $[\np d]$  contained in the interior of the desired triangle, due to the convexity of $\HPb$. 
\begin{figure}
	\centering
	\begin{subfigure}{0.5\textwidth}
		\centering
		\includegraphics[width=10cm]{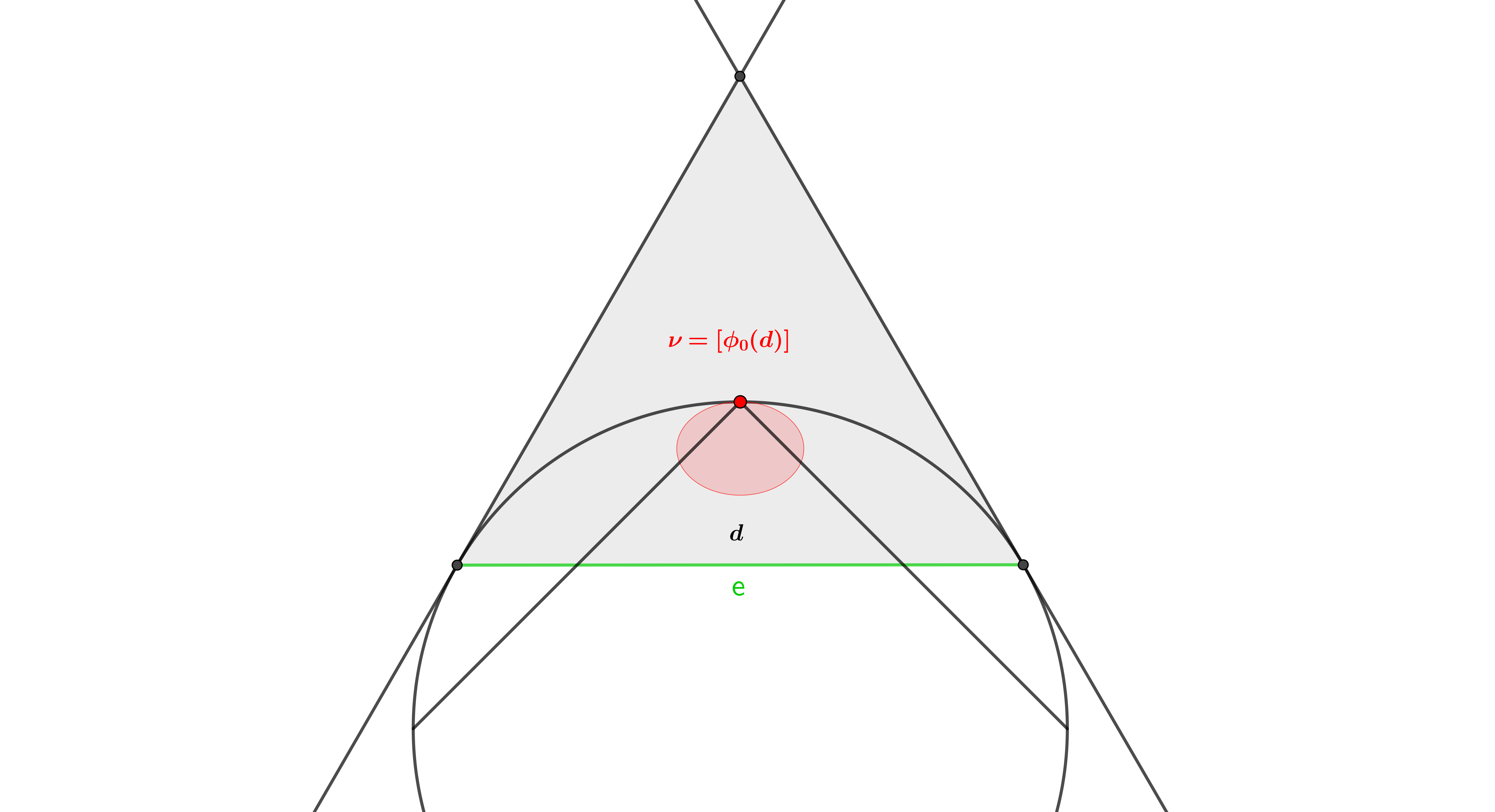}
		\caption{Type 1}
		\label{decodim0quad}
	\end{subfigure}%
	\begin{subfigure}{.5\textwidth}
		\centering
		\includegraphics[width=\linewidth]{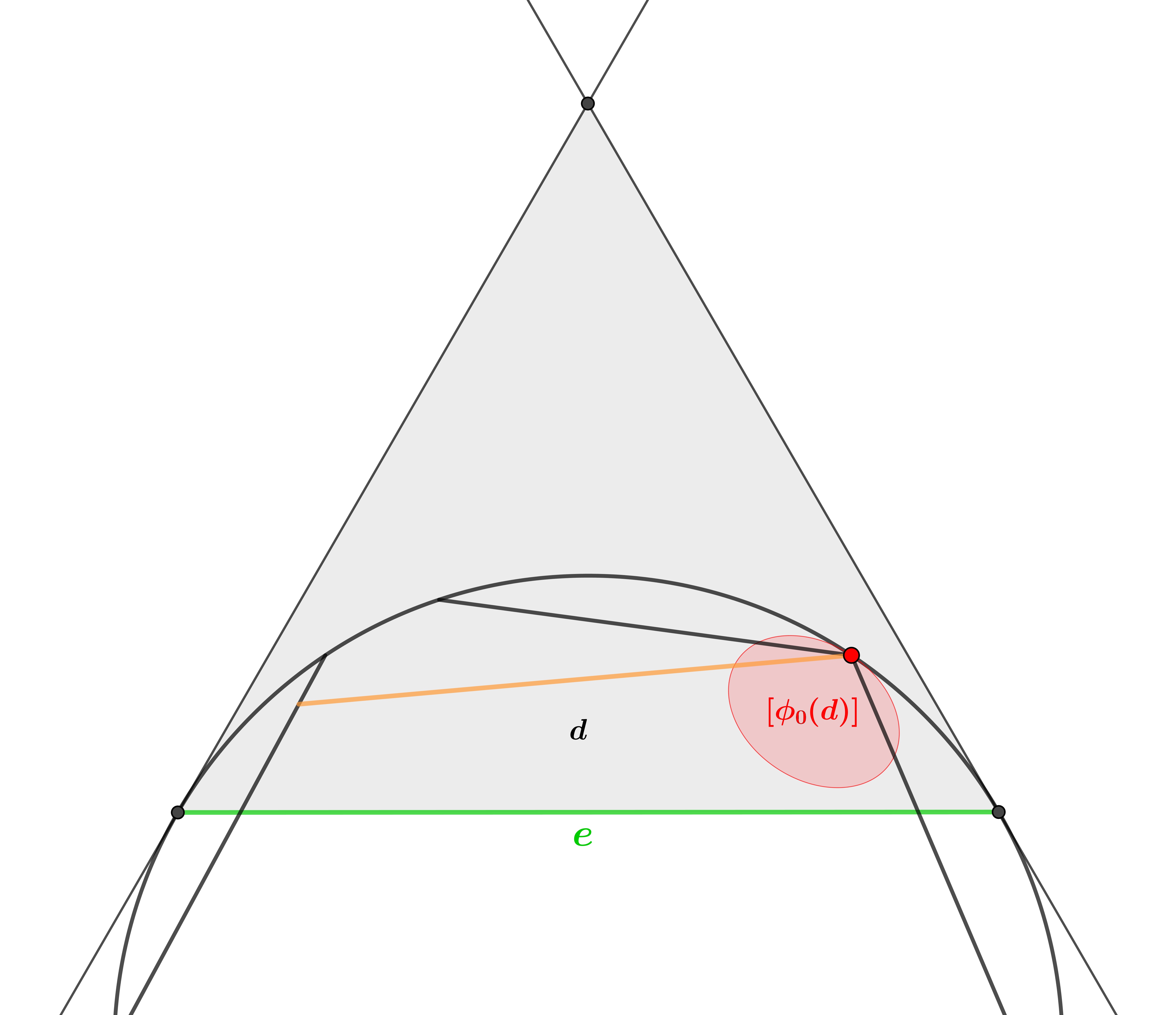}
		\caption{Type 2}
		\label{decodim0penta}
	\end{subfigure}
	\begin{subfigure}{\textwidth}
	\centering
	\includegraphics[width=0.5\linewidth]{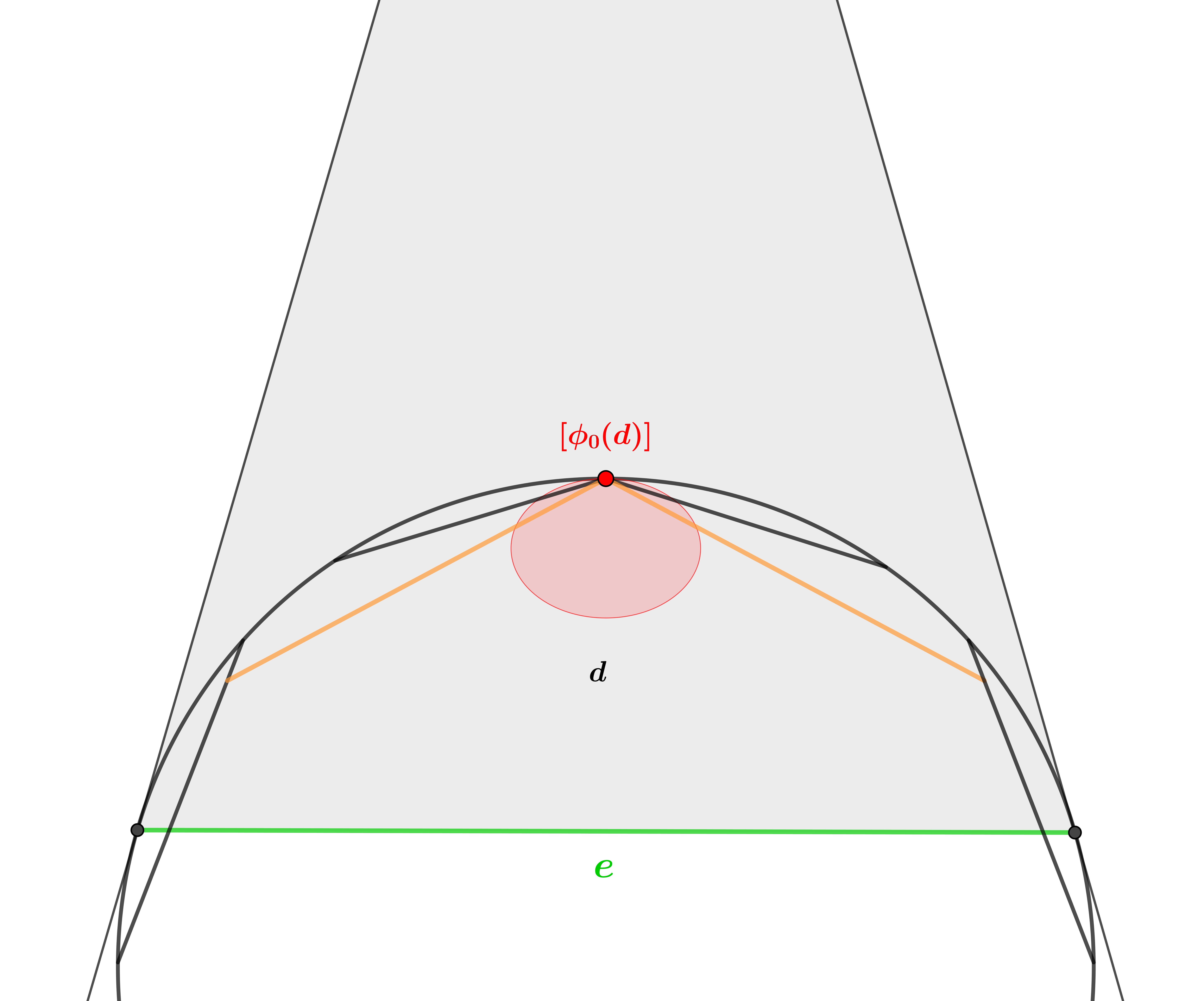}
	\caption{Type 2}
	\label{decodim0hexa}
\end{subfigure}
	\caption{$\phi_0$-images of tiles 1,2,3}
	\label{decodim0}
\end{figure}


	\begin{lem}\label{lemdeco}
		Let $\sigma$ be a top-dimensional simplex of $\ac{\dep n}$. Let $\phi_0:\ltile\rightarrow \lalg$ be a neutral tile map corresponding to the linear combination \eqref{lindep}. 
		Let $e$ be an internal edge-to-edge arc of a tile $d\in \tile$ such that $\np{d}\neq 0$. Then, $[\np{d}]$ is contained in the interior of the projective triangle in $\pp$, based at the geodesic $\ol e$ carrying $e$, that contains $d$.
	\end{lem}
	
	\begin{proof}
		For every triangulation $\sigma$, there is at least one tile of type one and every tile has at least one internal edge-to-edge arc. Consider the dual graph of the triangulation of the decorated polygon by $\sigma$. It is a finite tree. Let $\tau$ be the finite rooted sub-tree crossing the arc $e$ with root at the tile $d$. We will now prove that every tile on this sub-tree satisfies the lemma. 
		Let $d\in \tile$ be any tile and $e$ be an internal edge-to-edge arc. We define $M(d)$ to be the longest path in $\tau$ joining $d$ and a tile containing one decorated vertex.
		The proof is done by induction on $M$. 
		
		When $M(d)=0$, the tile $d$ is of type one (one decorated vertex and one internal edge. From the discussion before the lemma, we get $[\np d]$ lies in the desired triangle. 
		
		Now, let the statement be true for $M(d)=0,\ldots,k$. Again, if $d$ is a tile with a decorated vertex then we know already that the statement is verified. So we assume that $d$ is a hexagon without any decorated vertex, such that $\np d\neq 0$. Then it has two neighbouring tiles $d', d''$ contained in $\Delta$, with common arcs $e', e''$ respectively. Both $e',e''$ are edge-to-edge arcs.  The proof is then identical to that of Lemma \ref{lemideal}.
		This proves the induction step.
	\end{proof}
	
	Now we prove by contradiction that the coefficient $c_e$ of any edge-to-edge arc $e$ has to be zero. Let $e\in \ed$ be an edge-to-edge arc, that is common to the two neighbouring tiles $d_1, d_2$. Let $c_{e} \neq 0$. Then, $[\phi(d_1)-\phi(d_2)]\in \oarr{e}\backslash \cHP$. Since both $\np{d_1}$ and $\np{d_2}$ cannot be simultaneously equal to zero, we have two cases:
	\begin{enumerate}
		\item Let $\np{d_1}$ and $\np{d_2}$ be both non-zero. By the above lemma, $\np{d_1}$ and $\np{d_2}$ belong to two disjoint triangles associated to $e$. By Property \ref{line}, we have $[\np{d_1}-\np{d_2}]$ must intersect $e$ inside $\HP$, which is a contradiction. 
		\item Suppose that $\np{d_1}=0\neq \np{d_2}$. Then, the point $[\np{d_1}-\np{d_2}]=[\np{d_2}]$ does not intersect $\oarr{e}$, which is again a contradiction. 
	\end{enumerate}
	So, we have $\np{d_1}=\np{d_2}$, whenever two tiles $d_1,d_2\in \tile$ have a common edge-to-edge arc.
	\\
	Let $d, d'\in\tile$ be two tiles with different decorated vertices $\nu,\nu'$ such that $d$ and $d'$ can be joined by a path in the dual tree that crosses only edge-to-edge arcs. Then, from the above discussion we have that $[\np{d}]=[\np{d'}]$. But $\np{d'}$ must fix $\nu'$ which is different from $\nu$.  So we get $\np{d}=\np{d'}=0$. Since every tile has an edge-to-edge arc and there is more than one decorated vertex, we get that $\np d=0$ for every $d\in \tile$. So we get that $c_e=0$ for every $e\in \ed$, which proves the theorem.
\end{proof}	
Finally we will consider the case of decorated once-punctured polygons.

\begin{thm} \label{thmdecopunc}Let $m\in\tei {\depu n}$ be a metric on a decorated once-punctured polygongon $\depu  n$, with $n\geq 2$. Fix a choice of strip template. Let $\sigma$ be a top-dimensional simplex of its arc complex $\ac{\depu n}$ and let $\ed$ be the corresponding edge set. Then the set of infinitesimal strip deformations $B=\{   \isd| e\in \ed  \}$ forms a basis of the tangent space $\tang{\depu n}$.
\end{thm}
\begin{proof}
	Again, we have that  $\dim \tang {\dep n}=\#\ed=2n-1$. So, it is enough to show that the above set is linearly independent. We start with an equation as in \eqref{lindep} with a corresponding $\rho(\ga)$-invariant neutral map  $\phi_0:\ltile\longrightarrow\lalg$. This map is $\rho(\ga)$-invariant, where $\ga$ is the generator of the fundamental group of the surface. From the proof of Theorem \ref{thmpunc}, we know that $\np {d_\infty}$ is either zero or a parabolic element, fixing the light-like line $\R p$ and $[\np {d_\infty}]=p$. We also know that  a tile $d$ has a decorated vertex $\nu$ (Fig. \ref{decodim0}). The Killing field $\np d$ fixes the ideal point as well as the horoball decoration.  If $\np d\neq 0$, then the point $[\np d]$  contained in the interior of the desired triangle, due to the convexity of $\HPb$. The following lemma follows from the Lemmas \ref{lempunc} and \ref{lemdeco}.
		\begin{lem}\label{lemdecopunc}
		Let $\sigma$ be a top-dimensional simplex of $\ac{\depu n}$. Let $\phi_0:\ltile\rightarrow \lalg$ be a neutral tile map corresponding to the linear combination \eqref{lindep}. 
		Let $e$ be an internal edge-to-edge arc of a tile $d\in \tile$ such that $\np{d}\neq 0$. Then, $[\np{d}]$ is contained in the interior of the projective triangle in $\pp$, based at the geodesic $\ol e$ carrying $e$, that contains $d$.
	\end{lem}
Using the argument after the end of the proof of Lemma \ref{lemdeco}, we get that $\np{d_1}=\np{d_2}$, whenever two tiles $d_1,d_2\in \tile$ have a common edge-to-edge arc. Since the infinite tile $\tilde{d_\infty}$ has no vertex-to-edge arc, we conclude that $c_e=0$ for every $e\in \ed$, which proves the theorem.

\end{proof}

\subsection{Local homeomorphism: Codimension 1}
In this section we show that the projectivised strip map $\mathbb Pf:\sac \Pi\longrightarrow \ptan \Pi$ is a local homeomorphism around points belonging to the interiors of simplices of codimension 1. 

\begin{thm}\label{codim1}
	Let $\Pi$ be any one of the four types of polygons — ideal $n$-gons $\ip n$, once-punctured $n$-gons $\punc n$, decorated $n$-gons $\dep n$and decorated once-punctured $n$-gons $\depu n$. Let $m\in \tei {\Pi}$ be a metric. Let $\sigma_1 ,\sigma_2 \in \ac \Pi$ be two top-dimensional simplices such that $$\codim {\sigma_1 \cap \sigma_2}=1 \text{ and } \inte{\sigma_1 \cap \sigma_2}\subset\sac\Pi.$$ Then,
	\begin{equation}\label{interior}
		\inte{\mathbb{P}f(\sigma_1)}\cap \inte{\mathbb{P}f(\sigma_2)}=\varnothing.
	\end{equation}
	Moreover, there exists a choice of strip template such that $\mathbb{P}f(\sigma_1)\cup \mathbb{P}f(\sigma_2)$ is convex in $\ptan \Pi$.
\end{thm}

Firstly, we will give a general idea of the proof for any type of polygon and then we will give the proof in each case in the subsequent sections \ref{ipdim1}-\ref{decodim1punc}.
\paragraph{Idea of the proof:}	Let $\edo$ and $\edt$ be the edge sets of $\sigma_1$ and $\sigma_2$ respectively. Since the simplex $\sigma_1\cap \sigma_2$ has codimension one, we have that $\edo\backslash\edt$ (resp. $\edt\backslash\edo$) has exactly one arc, denoted by $\al_1$ (resp. $\al_2$). There are different possibilities for the pair $\{\al_1,\al_2 \}$ in the case of every polygonal surface. Let $\lred$ be the refined edgeset of $\ledo$ obtained by considering the refinement $\sigma:=\sigma_1\cup \{\al_2\}$. Let $\lrtile$ be the refined tile set of $\ltile$.

In every case, we shall give a choice of strip template and then construct a tile map that represents the following linear combination for a chosen strip template and is coherent around every point of intersection: 
\begin{equation} \label{lincomb}
	c_{\al_1}f_{\al_1}(m)+c_{\al_2}f_{\al_2}(m)+\sum_{\be\in \edo\cap\edt}c_{\be}f_{\be}(m)=0, 
\end{equation} 
where $c_{\be}\leq 0$ for every $\be\in \edo\cap\edt$ and $c_{\al_1},c_{\al_2}>0$. Drawing all arcs of $\sigma_1\cup\sigma_2$ subdivides the surface into a system of tiles that refines both the triangulations. We will choose strip templates and assign Killing fields equivariantly to these tiles in a way that expresses this linear combination. 
The construction is done in the upper half plane model $\HP$.  We shall use the identification $\lalg\simeq\R_2[z]$ from section 1.5.

\subsubsection{Ideal Polygons}\label{ipdim1}

Since ideal polygons are simply-connected, $\ledo=\edo, \ledt=\edt, \lrtile=\rtile$. We choose  an embedding of the polygon into the upper half plane so that the point $\infty$ is distinct from all the vertices of the polygon, for $n\geq 5$. We shall consider the following strip template:
\begin{itemize}
	\item For every isotopy class, choose the geodesic representative $\al_g$ which intersects the boundary of the polygon perpendicularly;
	\item  For every isotopy class $\al$, the waist $\pal$ is given by the projection of $\infty$ on $\al_g$. 
	\item For every isotopy class $\al$, we take the width of the strip deformation $\wal=1$. 
\end{itemize}
Then every geodesic arc used in the triangulation is carried by a semi-circle. 
		\begin{figure}
	\centering
	\begin{tabular}[c]{c@{\qquad\qquad}c}
		\includegraphics[width=.35\textwidth]{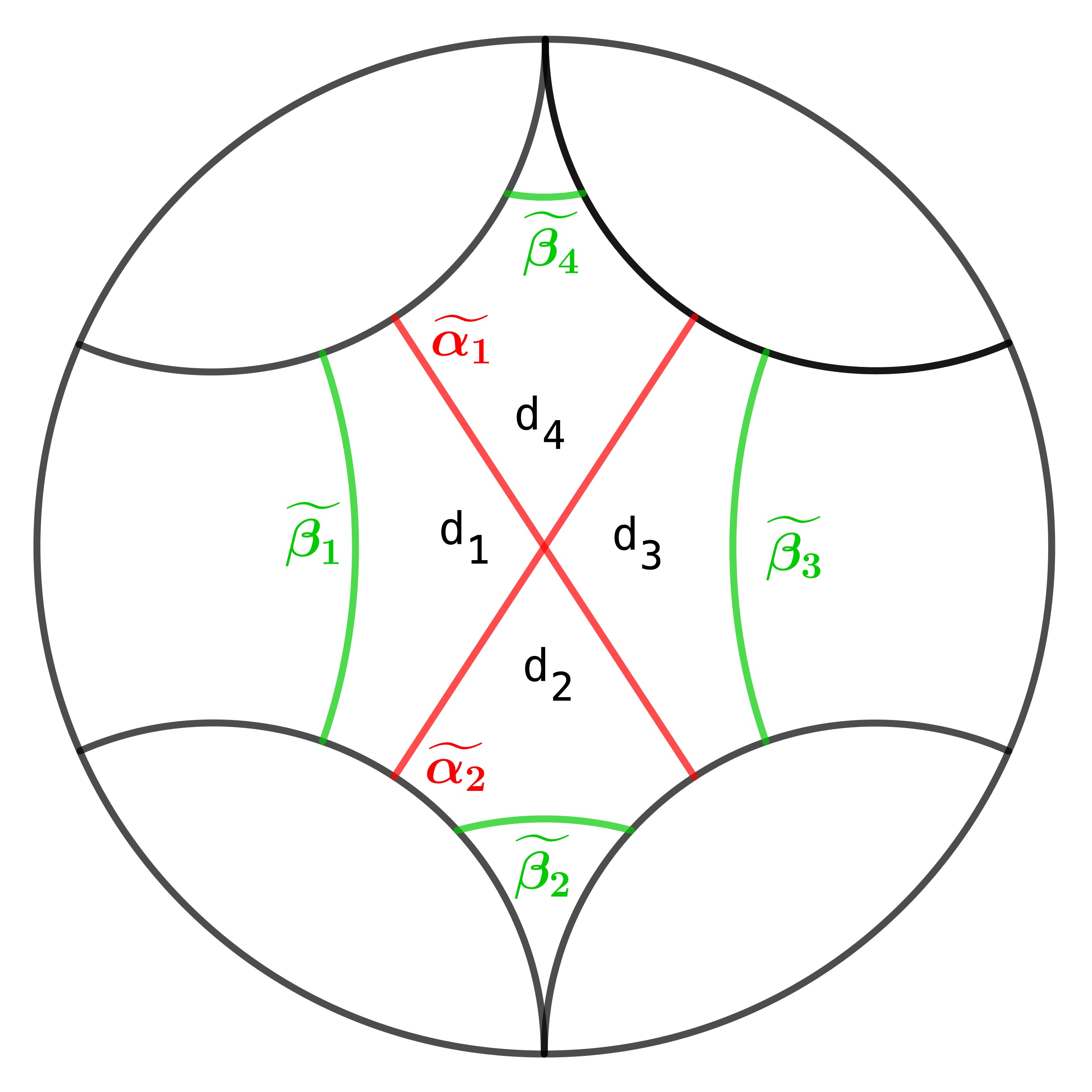}
		&
		\includegraphics[width=.35\textwidth]{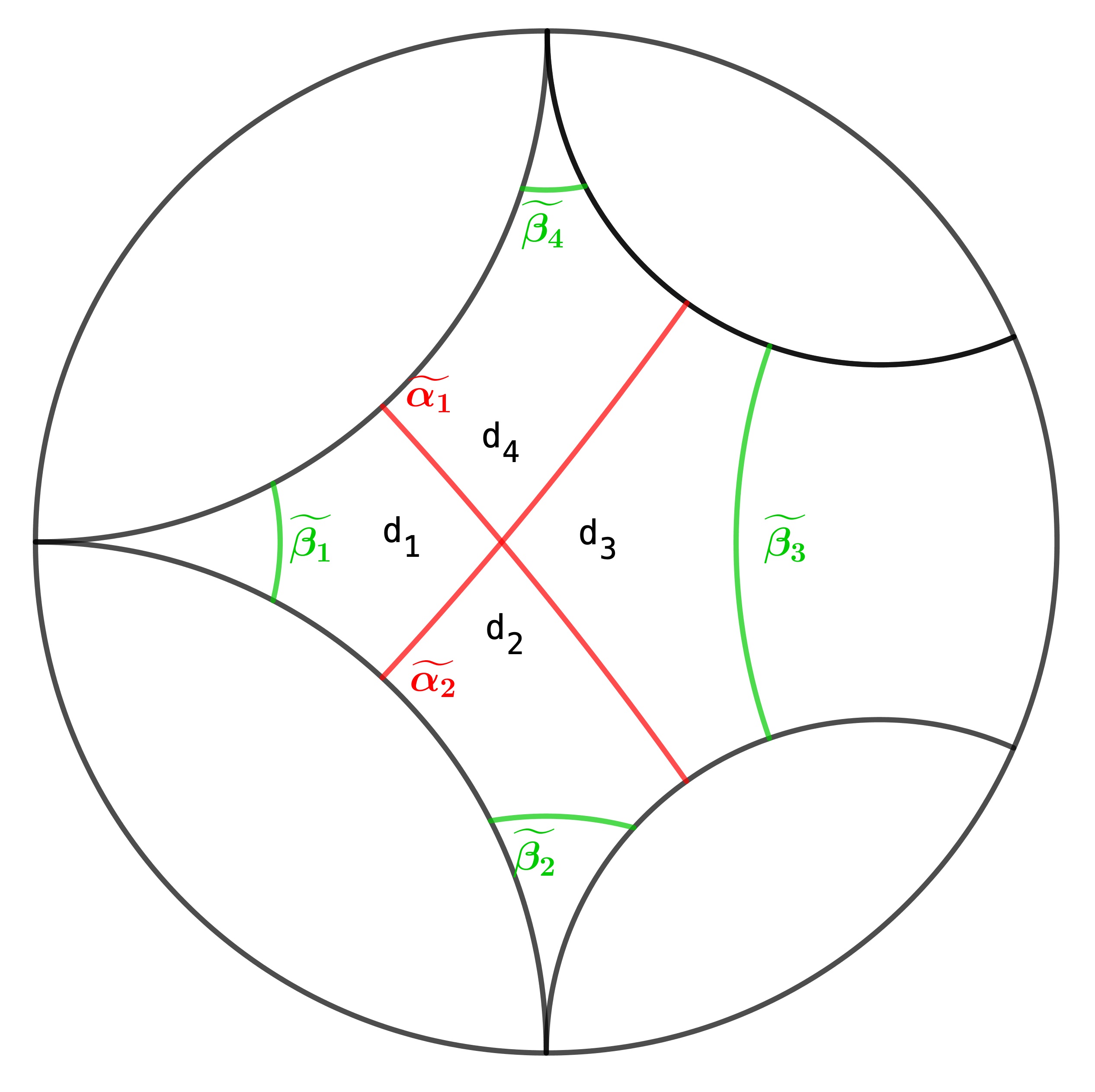}\\[2em]
		\includegraphics[width=.35\textwidth]{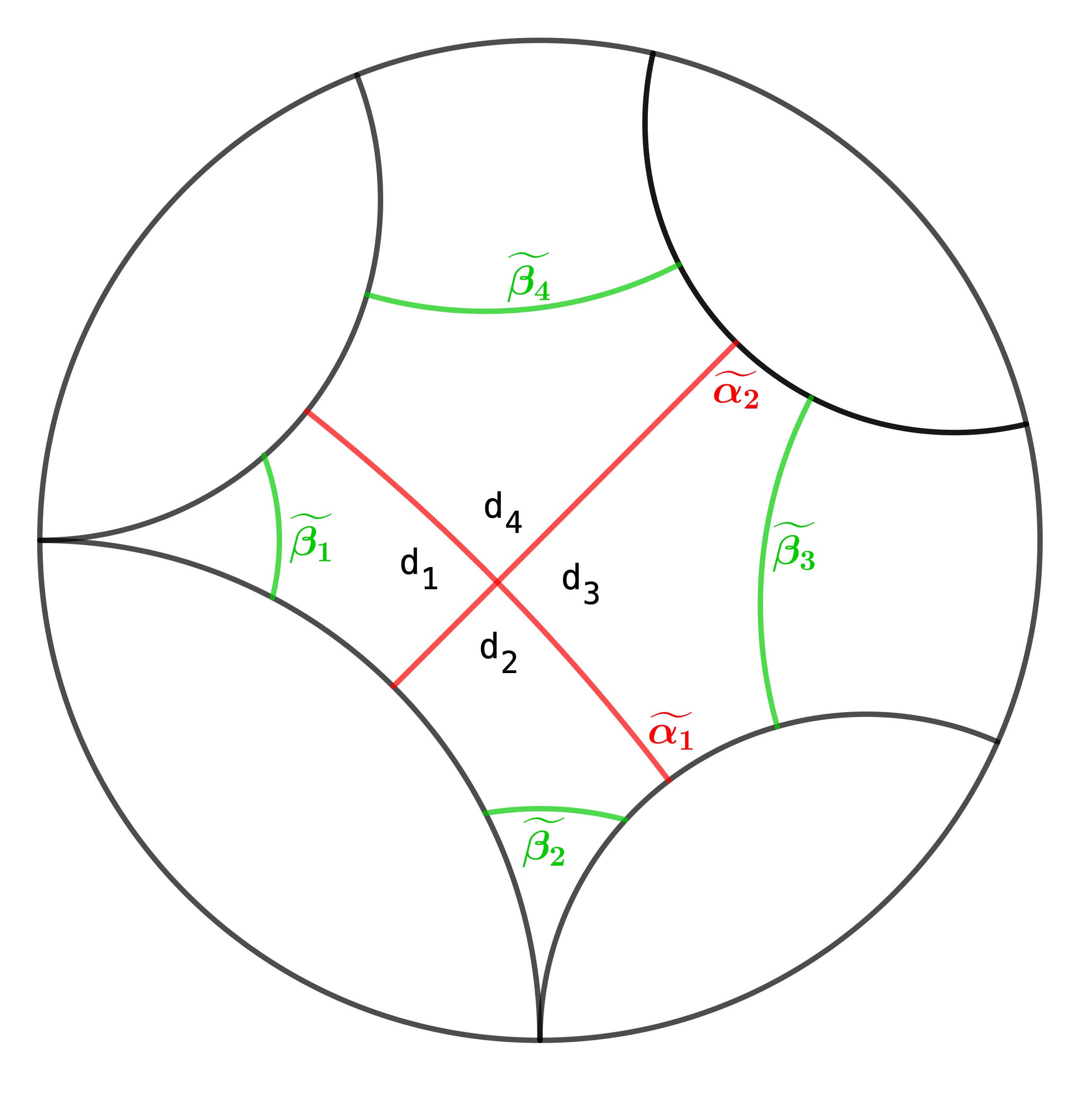} 
		&
		\includegraphics[width=.35\textwidth]{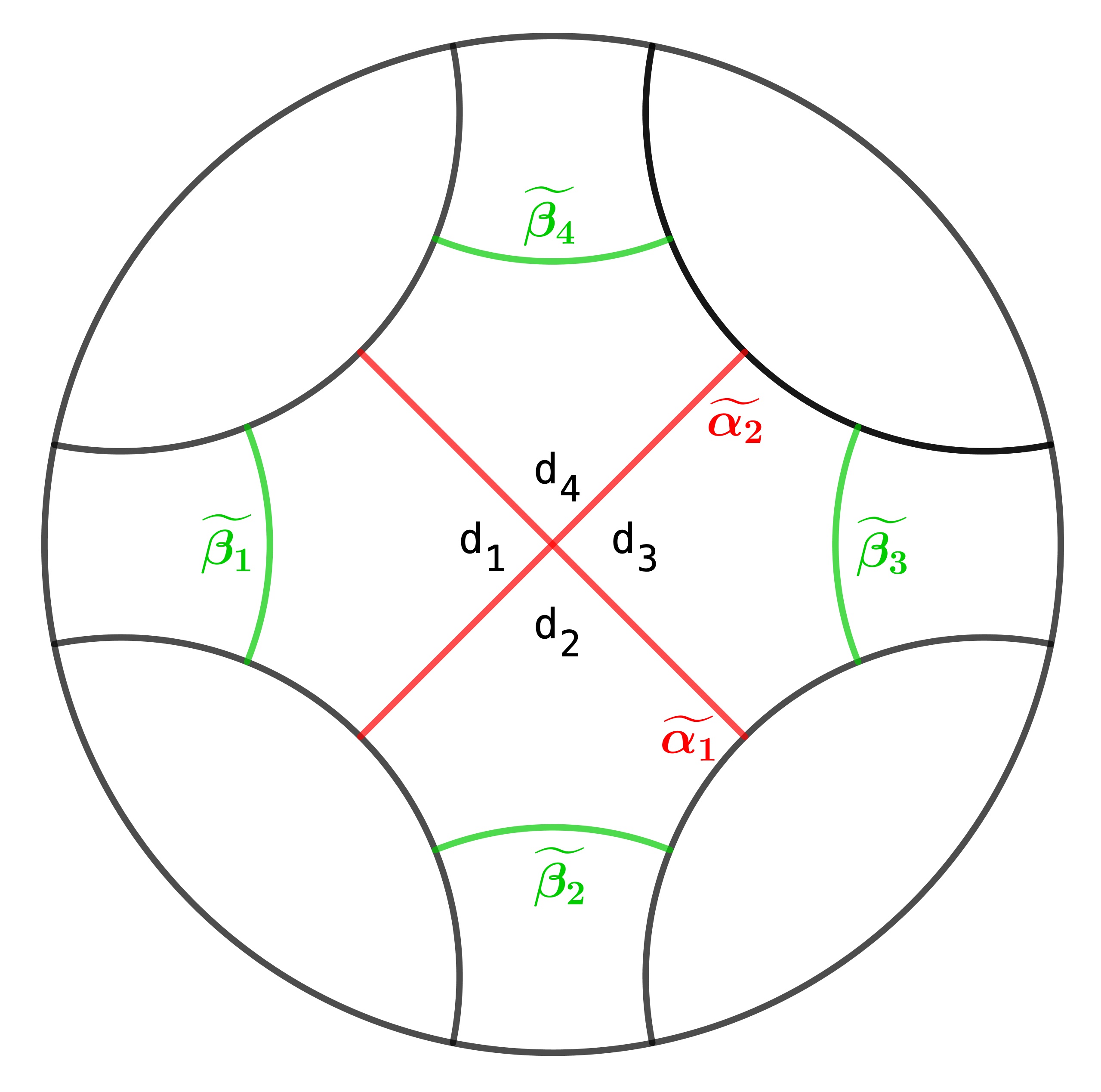}\\ [2em]
		\includegraphics[width=.35\textwidth]{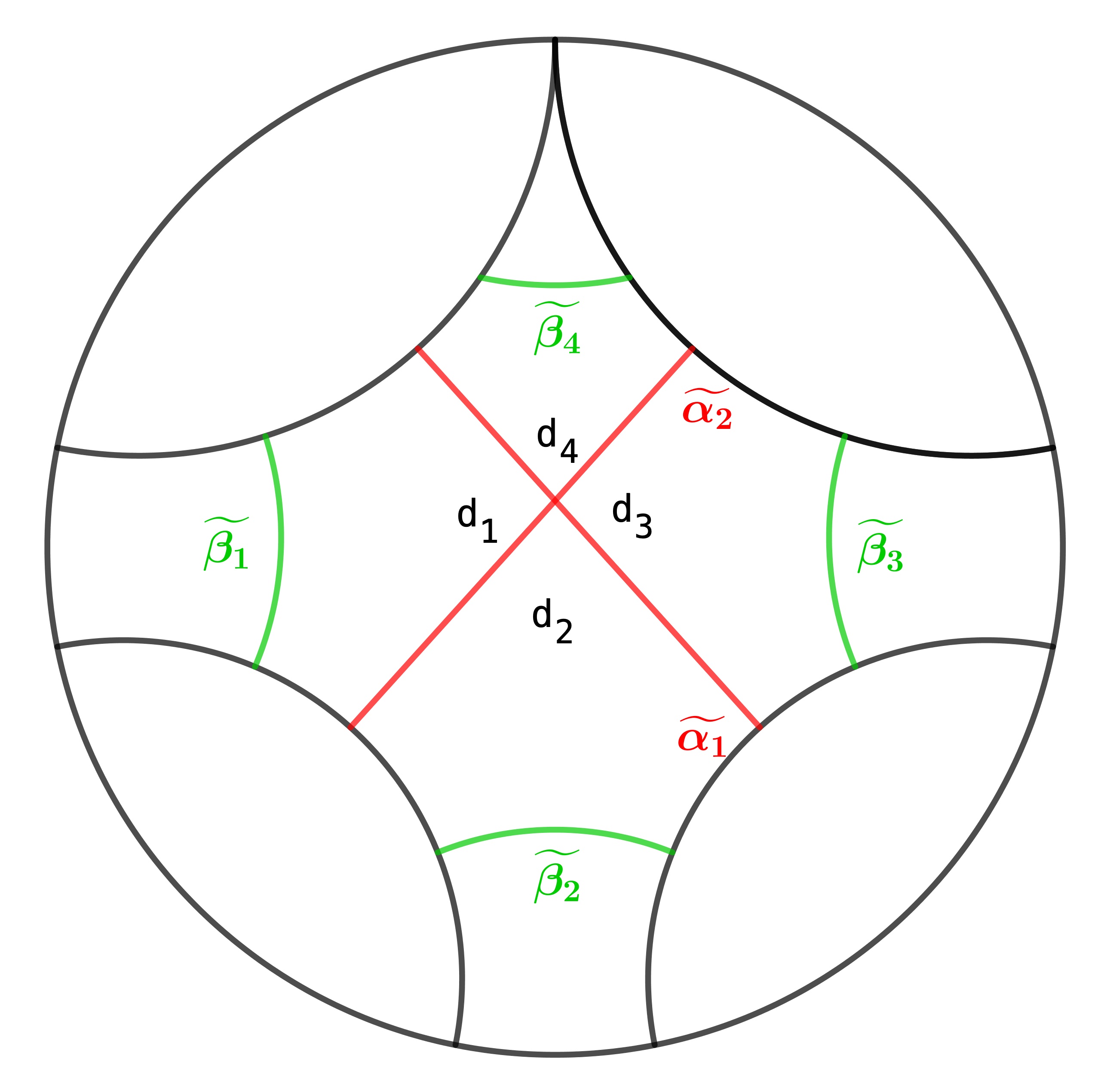}
		&
		\includegraphics[width=.35\textwidth]{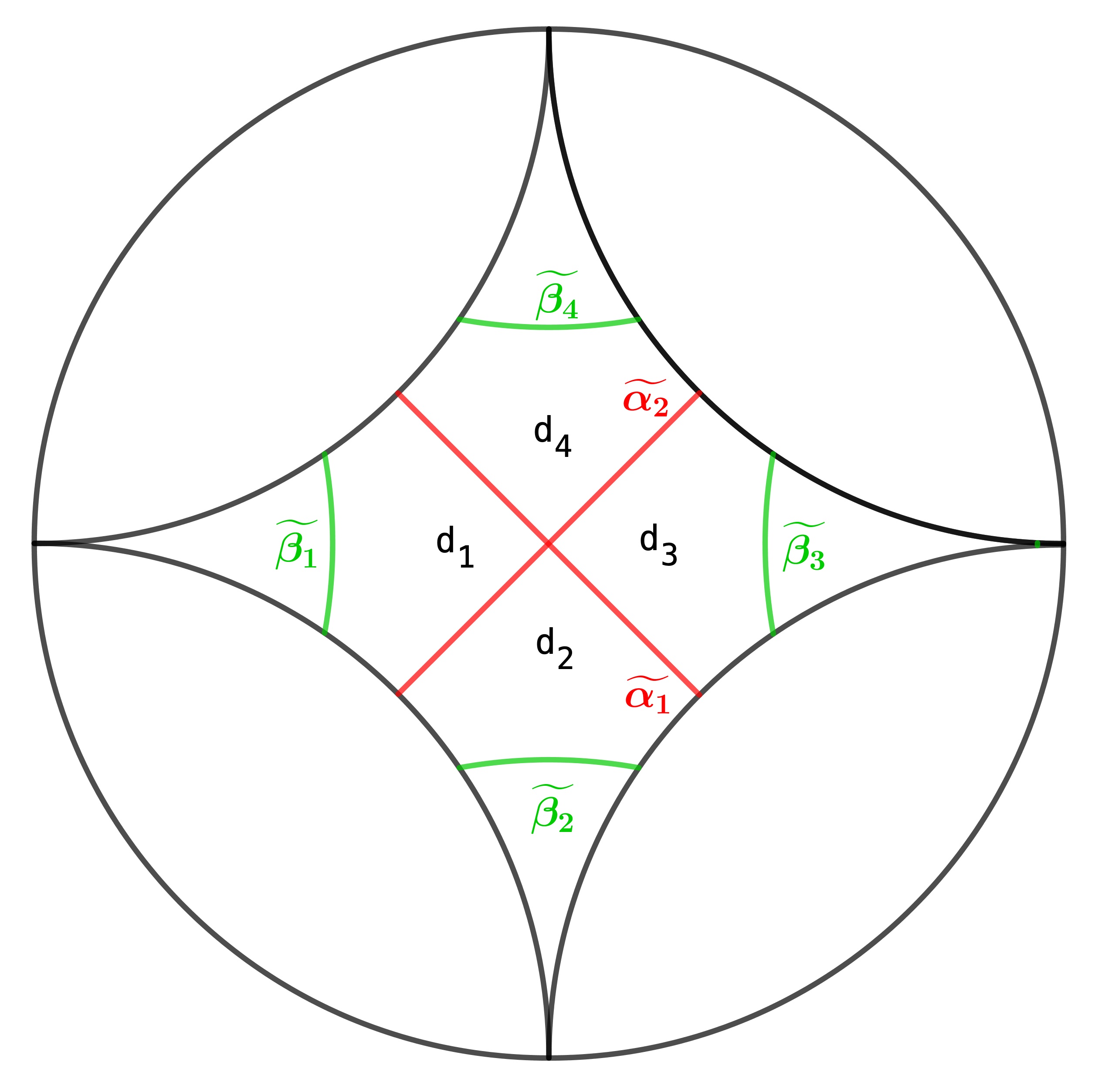}
	\end{tabular}
	\caption{The six possible configuration}
	\label{cd1decoee}
\end{figure}
In an ideal polygon, any two arcs intersect at most once. See Fig.\ref{cd1decoee}. The geodesic arcs that are coloured green in the figure are common to both $\sigma_1$ and $\sigma_2$. Let $o$ be the point of intersection of $\al_1, \al_2$. In each of the six cases, there are four small tiles formed around $o$, namely $d_j$, $j=1,2,3,4$, labeled anti-clockwise such that $d_1,d_2$ lie below the semi-circle carrying $\al_1$. 
For each $j$, $d_j$ is either a quadrilateral with exactly one ideal vertex and two internal edges contained in $\al_1$ and $\al_2$, or it is a pentagon with exactly three internal edges: $\al_1,\al_2$ and a third arc $\be_j\in \edt\cap\edo$. Let $\mathcal J\subset \{1,\ldots,4\}$ be such that the tile $d_j$ is pentagonal if and only if $j\in \mathcal J$. Note that the arc $\be_j$  intersects the boundary of the polygon perpendicularly, due to the choice of strip template. For $i=1,2$, let $x_i\in \R$ be the centre of the semi-circle carrying the geodesic arc $\al_i$. For $j=1,\ldots,4$, let $y_j\in \R$ denote the ideal vertex of $d_j$ or the centre of the semi-circle carrying the geodesic arc $\be_j$. 
\begin{figure}[!h]
	\begin{center}
		\frame{\includegraphics[width=\textwidth]{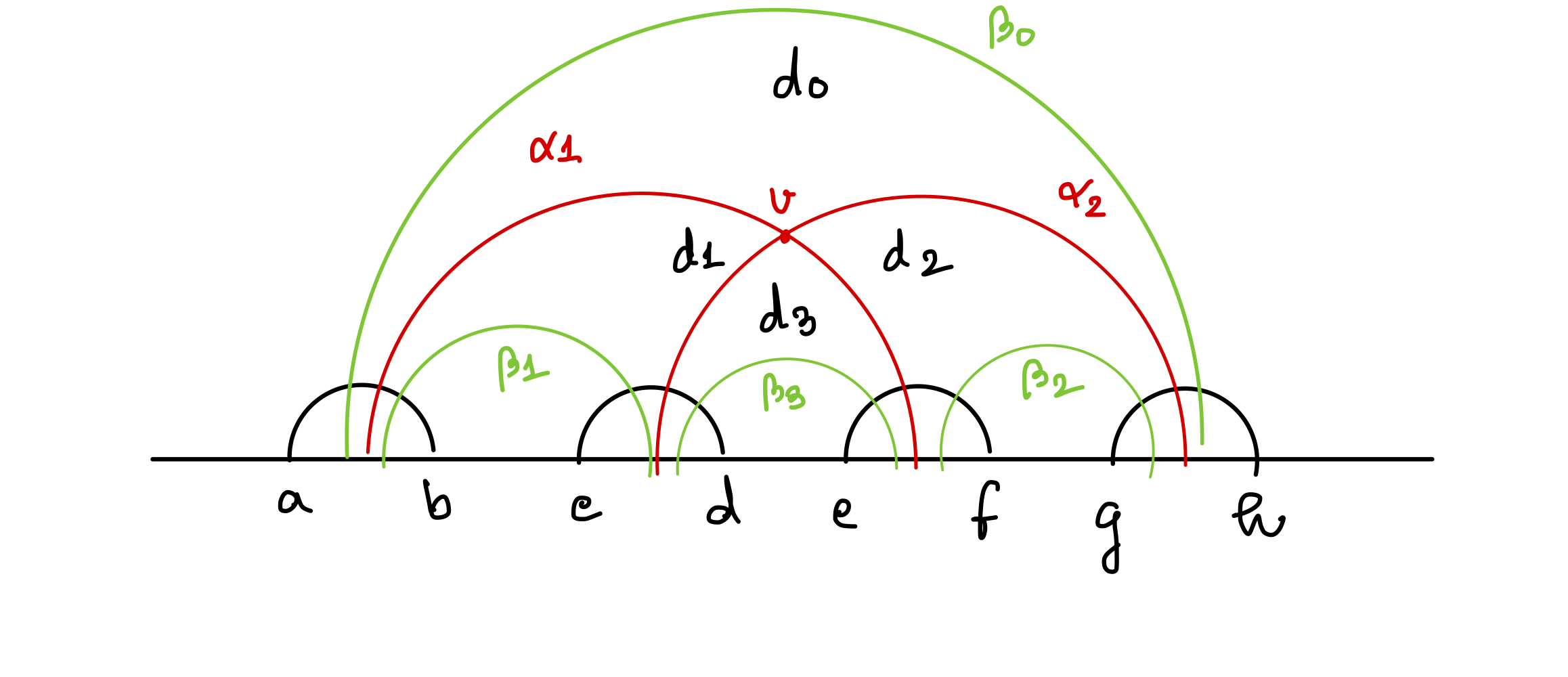}}
		\caption{Ideal polygons: The generic case}
		\label{idcodim1}
	\end{center}
\end{figure}
We shall construct a tile map  corresponding to the following linear combination:
\begin{equation}\label{lincombi}
	c_{\al_1}f_{\al_1}(m)+c_{\al_2}f_{\al_2}(m)+\sum_{j\in \mathcal J} c_{\be} f_{\be}(m)=0,
\end{equation}
with $c_{\al_1},c_{\al_2}>0$ and $c_{\be_j}< 0$ for every $j\in \mathcal J$. 
\begin{proper}\label{propid}
	A neutral tile map $\phi_0:\rtile\longrightarrow\R_2[z]$ represents the linear combination \eqref{lincombi} if and only if it verifies the following properties:
	\begin{enumerate}
		
		\item \label{idneutral} The polynomial $\np \del$ vanishes at every ideal vertex of $\del\in \rtile$ whenever it has one.
		\item \label{idcoh}The tile map is coherent around the intersection point $o$: $ \np{d_4}-\np{d_1}= \np{d_3}-\np{d_2}.$
		\item \label{idalp}Let $\del,\del'\in \rtile$ be two tiles with common internal edge contained in $\al_i$ for $i=1,2$ such that $\del$ lies above and $\del'$ lies below the semi-circle carrying the common internal edge. Then $\np \del-\np {\del'}$ is a hyperbolic Killing field with attracting fixed point at $\infty$ and repelling fixed point at $x_i$. In particular, its axis intersects $\al_i$ at $p_{\al_i}$. In terms of polynomials, we must have $$\np \del-\np {\del'}=(z\mapsto A(z-x_i)), \, \text{for some } A>0.$$
		
		where $v_i$ is a hyperbolic Killing vector field that whose axis is perpendicular to $\al_i$ at $p_{\al_i}$ and whose direction is towards the tile that lies above $\al_i$.
		\item If $d,d'\in \rtile$ are two tiles with common internal edge contained in $\al_i$ for $i=1,2$, then $\np \del-\np {\del'}$ is a hyperbolic Killing vector field whose axis is perpendicular to $\al_i$ at $p_{\al_i}$ and whose direction is towards $\del$.
		\item \label{idbe}Suppose $\del,\del'\in \rtile$ are two tiles with common internal edge $\be_j$ for $j\in\mathcal J$, such that $\del$ lies above $\be_j$. Then $\np \del-\np {\del'}$ is a hyperbolic Killing field with attracting fixed point at $y_j$ and repelling fixed point at $\infty$. In particular, its axis intersects $\be_j$ at $p_{\be_j}$. In terms of polynomials, we must have $\np \del-\np {\del'}=(z\mapsto B(z-y_j)),\, \text{for some } B<0.$
	\end{enumerate}
\end{proper}
Suppose that the endpoints of $\al_1$ lie on the boundary geodesics $(a,b)$ and $(e,f)$ and those of $\al_2$ lie on $(c,d)$ and $(g,h)$ such that the following inequalities hold for $n\geq 5$:
\begin{align}
	&a<b\leq c<d\leq e<f\leq g<h.
\end{align}
We shall treat the case $n=4$ separately.
polynomial with positive leading coefficient which vanishes at $x_i$. is a polynomial with negative leading coefficient which vanishes at $y_j$. 

Using Lemma \ref{centre}, we get that
\begin{align}
	x_1=\frac{ef-ab}{e+f-a-b}, &\quad x_2=\frac{gh-cd}{g+h-c-d},\\
	y_1=\frac{cd-ab}{c+d-a-b},&\quad y_2=\frac{ef-cd}{e+f-c-d},\\
	y_3=\frac{gh-ef}{g+h-e-f}, &\quad y_4=\frac{gh-ab}{g+h-a-b}.
\end{align}
definition of tile map
For $j=1,\ldots,4$, define 
\begin{eqnarray*}
	\phi_0&:\tile\longrightarrow &\twop\\
	&\del \longmapsto &\left\{ 
	\begin{array}{ll}
		(z\mapsto a_j(z-y_j)), & \text{if } \del=d_j\\
		0, & \text{otherwise,}
	\end{array}
	\right.
\end{eqnarray*}
where \[\begin{array}{llll}
	a_1=\frac{x_1-y_4}{x_1-y_1},&a_2=\frac{(x_1-y_4)(x_2-y_4)-(y_4-y_1)(y_4-y_3)}{(x_1-y_1)(x_2-y_3)}&a_3=\frac{x_2-y_4}{x_2-y_3} &a_4=1.
\end{array} \]
The $a_i's$ as defined above are a nontrivial solution to the following system of linear equations in four unknowns:
\begin{align}
	a_1-a_2+a_3-a_4&=0,\label{one}\\
	a_1y_1-a_2y_2+a_3y_3-a_4y_4&=0,\label{two}\\
	a_1(x_1-y_1)-a_4(x_1-y_4)&=0\label{three}\\
	a_3(x_2-y_3)-a_4(x_2-y_4)&=0.\label{four}
\end{align}
Applying Lemma \eqref{ineqasym} to the geodesics $(a,b), (c,d),(e,f)$ and then to $(c,d),(e,f),(g,h)$ we get that for $j=2,4$, $y_1<x_1<y_j<x_2<y_3,$
So, $a_1, a_2,a_3<0$.
\begin{rem}
	Note that for every $\del\in \rtile$, $\np \del\in \onep$. This is a consequence of our choice of normalisation and strip template.
\end{rem}

\begin{proof}[Verification of the Properties \ref{propid}:]~\\
	\vspace{-1em}
\begin{enumerate}
	
	\item Suppose that $\del$ is a tile with an ideal vertex. If $\del\in \supp{\phi_0} $, then $\del=d_j$  for some $j\in \set{1,\ldots,4}$, so that ideal vertex is given by $y_j$. From the definition of the tile map we have that $\np {\del}=P_j$ which vanishes at $y_j$.  If $\del\notin \supp{\phi_0} $, then $\np \del=0$, which automatically fixes its ideal vertex.
	\item  From eqs. (\ref{one}) and (\ref{two}) it follows that, 
	\[\begin{array}{ll}
		(\np{d_4}-\np {d_1})(z)&=(a_4-a_1)z-a_4y_4+a_1y_1\\
		&=(a_3-a_2)z-a_3y_3+a_2y_2\\
		&=(\np{d_3}-\np{d_2} )(z).
	\end{array}\]
	\item The tiles that share an edge carried by $\al_1$ are the pairs $\{d_1,d_4 \}$ and $\{d_2,d_3\}$. The tiles that share an edge carried by $\al_2$ are the pairs $\{d_1,d_2 \}$ and $\{d_3,d_4\}$. From the coherence property \eqref{idcoh}, it is enough to verify the property for $\{d_1,d_4 \}$ and $\{d_3,d_4 \}$. The tile $d_4$ lies above both the semicircles carrying the arcs $\al_1,\al_2$, respectively. From the definition of $\phi_0$ we have that,
	\begin{eqnarray*}
		(\np{d_4}-\np {d_1})(z)&=&(a_4-a_1)z+a_1y_1-a_4y_4,\\
		(\np{d_4}-\np {d_3})(z)&=&(a_0-a_3)z+a_3y_3-a_4y_4.
	\end{eqnarray*}
	Since $a_4>0$ and $a_1,a_3<0$, the leading coefficients $a_4-a_1$ and $a_4-a_3$ are both positive. The polynomials $\np{d_4}-\np {d_1}$ and $\np{d_4}-\np {d_3}$ vanish at $x_1$ and $x_2$ respectively, by eq.\eqref{three} and eq.\eqref{four} . 
	\item Suppose that the two tiles $\del,\del'$ have a common edge of the form $\be_j$ for $j\in\mathcal J$, with $\del$ lying above $\be_j$. Then either $\del'=d_4$ and $\del\notin \supp {\phi_0}$ or $\del=d_j$ for $1\leq j \leq 3$ and $\del'\notin \supp {\phi_0}$. 
	
	In the first case, 
	$(\np \del- \np{d_4})(z)=-a_4(z-y_0).$
	Since $-a_0<0$, the property is verified.
	This is a polynomial of degree 1, with negative leading coefficient, and which vanishes at $y_0$. So it is a hyperbolic element in $\lalg$ whose repelling fixed point is $y_0$ and attracting fixed point at $\infty$. 
	In the second case, for $1\leq j \leq 3$, $(\np {d_j}-\np {\del'})(z)= a_j(z-y_j).$
	Since $a_j<0$ for every $j=1,2,3$, the leading coefficient is negative.  \qedhere
	\end{enumerate}
\end{proof}
\subsubsection{Punctured Polygons}	
Next, we shall prove Theorem \ref{codim1} for undecorated punctured polygons.

Let $\sigma_1$ and $\sigma_2$ be as in the hypothesis with $\al_1\in\edo\backslash \edt$ and $\al_2\in\edt\backslash\edo$. These two arcs intersect either exactly once at a point $o$ (when both are non-maximal) or twice at the points $o_1,o_2$ (when both are maximal). We suppose that the ideal point corresponding to the puncture is at $\infty$ in the upper half plane model of $\HP$ and that  $\Gamma=\rho(\fg{\punc n})$ is generated by the parabolic element $T:z\mapsto z+1$, after normalisation. Let $\lred$ and $\lrtile$ be the refined edge set and tile set respectively for the refinement $\sigma=\sigma_1\cup \{\al_2 \}$. We take the following strip template: 
\begin{itemize}
	\item From every isotopy class of arcs, we choose the geodesic arc which intersects the boundary of the surface perpendicularly.
	\item For every geodesic arc, the waist is chosen to be the point of projection of $\infty$. This choice of waist is $\Gamma$-equivariant because $T$ fixes $\infty$.
\end{itemize}
We have the two following cases, depending on the maximality of $\al_1, \al_2$. 
\begin{figure}[!h]
	\begin{center}
		\frame{\includegraphics[width=\textwidth]{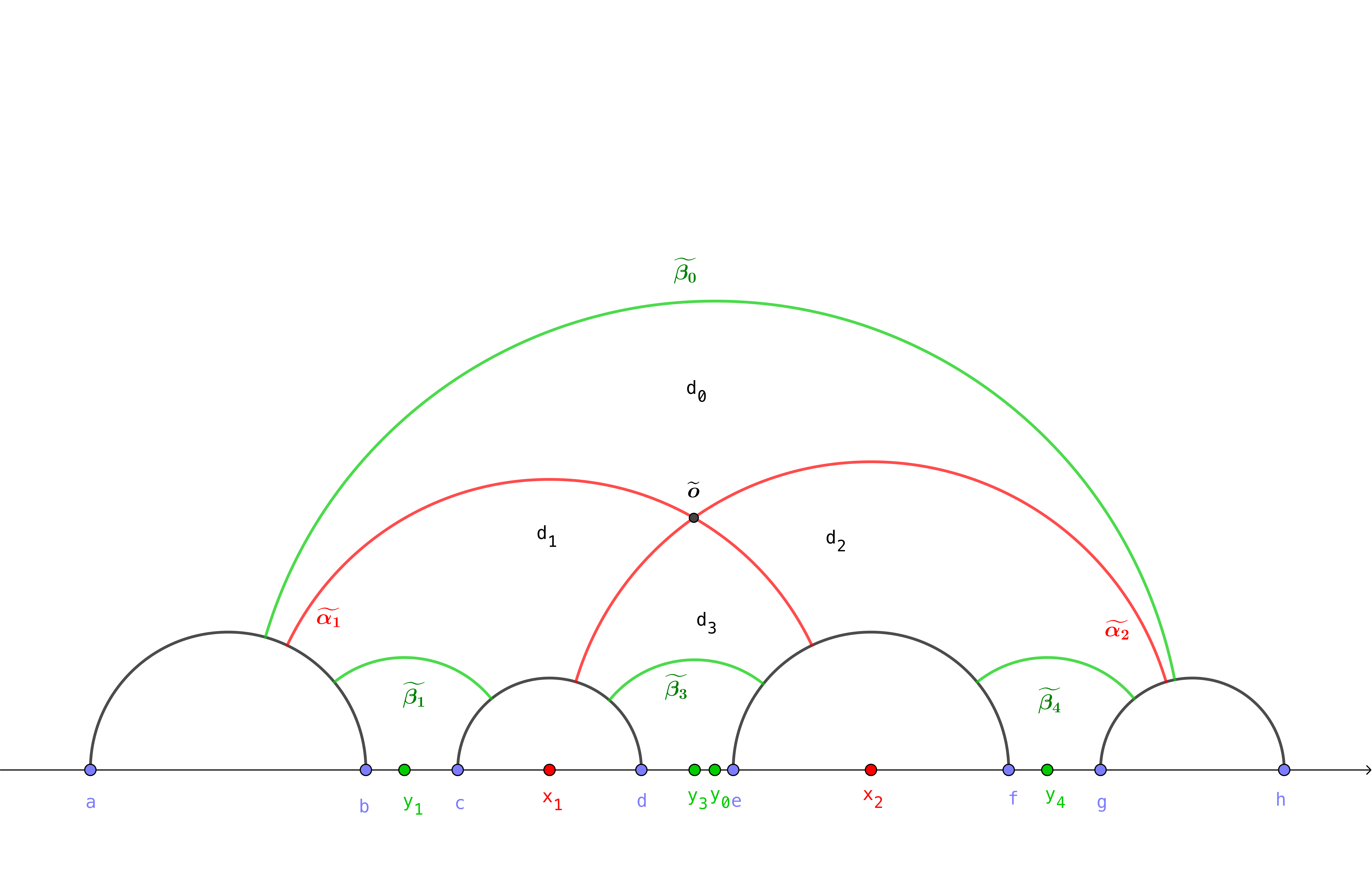}}
		\caption{$\al_1$ and $\al_2$ are non-maximal arcs}
		\label{non-mini}
	\end{center}
\end{figure}
\begin{enumerate}
	\item See Fig. \ref{non-mini}. When $\al_1, \al_2$ are both non-maximal, the construction is very similar to that in the case of the ideal polygons. Let $\wt o$ be a lift of the point $o$. Then $\wt o= \alo\cap \alt$ for two lifts  of $\al_1$ and $\al_2$ respectively. There are four finite tiles formed around $\wt o$, denoted by $d_j$, for $j=0,\ldots, 3$. For each $j\in \set{0,\ldots,3}$, the tile $d_{j}$ is either a quadrilateral with an ideal vertex and exactly two arc edges carried by $\alo$ and $\alt$, or it is a pentagon with exactly three arc edges carried by $\alo,\alt$ and a third arc $\wt{\be_j}$, which is a lift of an arc $\be_j\in \edt\cap\edo$. Let $\mathcal J\subset \{1,\ldots,4\}$ be such that the tile $d_j$ is pentagonal if and only if $j\in \mathcal J$. For $i=1,2$, let $x_i$ be the centre of the semi-circle containing $\wt{\al_i}$. For $j=0,\ldots,3$, let $y_j$ denote the ideal vertex of $d_j$ or the centre of the semi-circle containing $\wt{\be_j}$.
	In this case, a tile map representing the linear combination \eqref{lincombi} is a map $\phi_0: \lrtile\longrightarrow \R_2[z]$ that satisfies the following properties:
	\begin{proper}\label{proppunc}
		\begin{enumerate}
			\item $\phi_0$ is $\Gamma$-equivariant: for every $m\in \Z$, $\np {T^m\cdot \del}(z)=\np \del (z-m).$
			\item\label{proppuncneu} The polynomial $\np \del$ vanishes at every ideal vertex of $\del\in \rtile$ whenever it has one.
			\item \label{proppuncoh}The tile map is coherent around every point of intersection of the lifts of $\al_1$ and $\al_2$.
			\item \label{proppuncal}Let $\del,\del'\in \lrtile$ be two tiles neighbouring along an edge contained in a lift $T^m\cdot\wt{\al_i}$ of $\al_i$, for some $i\in \set{1,2}$, such that $\del$ lies above $\wt{\al_i}$. Then the difference $\np \del-\np {\del'}$ is a hyperbolic Killing vector field with attracting and repelling fixed points at $\infty$ and $x_i+m$, respectively. The axis intersects $\wt{\al_i}$ at $p_{\wt{\al_i}}$. In other words, 
			$$\np \del-\np {\del'} (z)=A(z-x_i-m),\, \text{for some } A>0.$$
			\item \label{proppuncbe}Let $\del,\del'\in \rtile$ be two tiles neighbouring along an edge $\wt\be\in \ledo\cap \ledt$ such that $\del$ lies above the edge. If $\wt\be=T^m\cdot \wt\be_j$ for some $j=1,\ldots,4$ then the difference $\np \del-\np {\del'}$ is a hyperbolic Killing vector field with attracting and repelling fixed points at $x_j+m$ and $\infty$, respectively. The axis intersects $\wt{\be_j}$ at $p_{\wt{\be_j}}$. In other words, 
			$\np \del-\np {\del'} (z)=B(z-y_j-m), \, \text{for some }B<0.$ Otherwise, $\np \del=\np {\del'}$.
			
		\end{enumerate}
	\end{proper}
	
	Let $(a,b)$ and $(e,f)$ be the two boundary geodesics that are joined by $\alo$. Similarly, let $(c,d)$ and $(g,h)$ be the two boundary geodesics joined by $\alt$ such that
	\begin{align}
		&a<b\leq c<d\leq e<f\leq g<h\leq a+1.
	\end{align}
	
	We consider the non-trivial solution $(a_0, a_1,a_2,a_3)$ of the system of equations \eqref{one}-\eqref{four} defined in the ideal polygon proof.
	For $j=0,\ldots,3$ and $m\in \Z$, define 
	\begin{eqnarray*}
		\phi_0&:\lrtile\longrightarrow &\lalg\\
		&\del \longmapsto &\left\{ 
		\begin{array}{ll}
			(z\mapsto a_j(z-y_j-m)), & \text{if } \del=T^m\cdot d _j \\
			0, & \text{otherwise.}
		\end{array}
		\right.
	\end{eqnarray*}
	
		\begin{proof}[Verification of Properties \ref{proppunc}:]~\\
			\vspace{-1em}
	\begin{enumerate}
			\item For every $m\in \Z$, $(T^m)'(z)=1$. If $\del= d_j$ for some $j\in\set{1,\ldots,4}$, then from the definition of $\phi_0$ we have that for every $m\in \Z$, $$\np {T^m\cdot \del} (z)=a_j(z-y_j-m)= \np \del (z-m).$$So, the equivariance condition is satisfied in this case. For $\del\notin \supp{\phi_0}$, the condition holds trivially.
		\end{enumerate}
		Since the map has been proved to be $\Gamma$-equivariant, it suffices to verify the properties \eqref{proppuncneu}-\eqref{proppuncbe} around the point $\wt o$. This is identical to the proof in the case of ideal polygons. This finishes the proof in the non-maximal case.
			\end{proof}
			\item Let $\al_i$ be maximal for $i=1,2$. See Fig. \ref{mini}.  Let $\wt {o_1},\wt {o_2}$ be two lifts of $o_1,o_2$ respectively, such that $\wt {o_1}=\alo\cap\alt$ and $\wt{o_2}=(T\cdot\alo)\cap\alt$ for two lifts $\alo, \alt$ of $\al_1,\al_2$, respectively. Let $d_0$ be the infinite tile around $\vo$,  and $d_1$ and $d_2$ be the tiles neighbouring $d_0$ along edges carried by $\alo$ and $\alt$, respectively. The fourth tile formed at the vertex $\vo$ is denoted by $d_3$. The tiles around $\vt$ are $d_0,d_2, T\cdot d_1$ and a fourth tile denoted by $d_4$. Again, for $ j=3,4$, the tile $d_{i}$ is either a quadrilateral with one ideal vertex, formed by two boundary edges of the punctured polygon and with two arc edges respectively contained in $\alo,\alt$, or it is a pentagon with two arc edges respectively contained in $\alo,\alt$ and a third arc edge $\lbe j$ which is a lift of an arc $\be_j\in\ledo\cap\ledt$. For $i=1,2$, let $x_i$ denote the centre of the semi-circle containing $\wt{\al_i}$. Let $y_3$ denote the ideal vertex of $\del_3$ or the centre of the semi-circle containing $\be_3$. Similarly, let $y_4$ denote the ideal vertex of $d_4$ or the centre of the semi-circle containing $\wt{\be_4}$. 
		Let $a,b,c,d\in \R$ be such that $\alo$ joins the two geodesics $(c-1,d-1)$ and $(c,d)$, and $\alt$ joins $(a,b)$ and $(a+1,b+1)$. Then $a,b,c,d$ satisfies \[a<b\leq c<d\leq a+1.\] Again, from Lemma \ref{centre}, we have that
		\begin{align}
			x_1=\frac{c+d-1}{2},&\quad x_2=\frac{a+b+1}{2}\\
			y_3=\frac{cd-ab}{c+d-a-b},&\quad y_4=\frac{cd-(a+1)(b+1)}{c+d-a-b-2}
		\end{align}
	\begin{eqnarray*}
	\phi_0&:\lrtile\longrightarrow& \R_2[z]\\
	&\del \mapsto &\left\{ 
	\begin{array}{lll}
		(z\mapsto a_i(z-x_i-m)), & \text{if } \del=T^m\cdot d _i, &i=1,2,\\
		(z\mapsto (a_1+a_2)(z-m)-a_1x_1-a_2x_2), & \text{if } \del=T^m\cdot d_3&\\
		(z\mapsto (a_1+a_2)(z-m)-a_1x_1-a_2x_2-a_1), & \text{if } \del=T^m\cdot d_4\\
		0, & \text{otherwise},
	\end{array}
	\right.
\end{eqnarray*}
where $a_1=-1, a_2=\frac{y_3-x_1}{y_3-x_2}.$ In particular, $\np{d_0}=0$. We note that $a_1,a_2$ satisfy the following two equations \begin{align}
	a_1(y_3-x_1)+a_2(y_3-x_2)&=0\label{One}\\
	a_1(y_4-x_1-1)+a_2(y_4-x_2)&=0.\label{Two}
\end{align}
\begin{figure}[!h]
	\begin{center}
		\frame{\includegraphics[width=15cm]{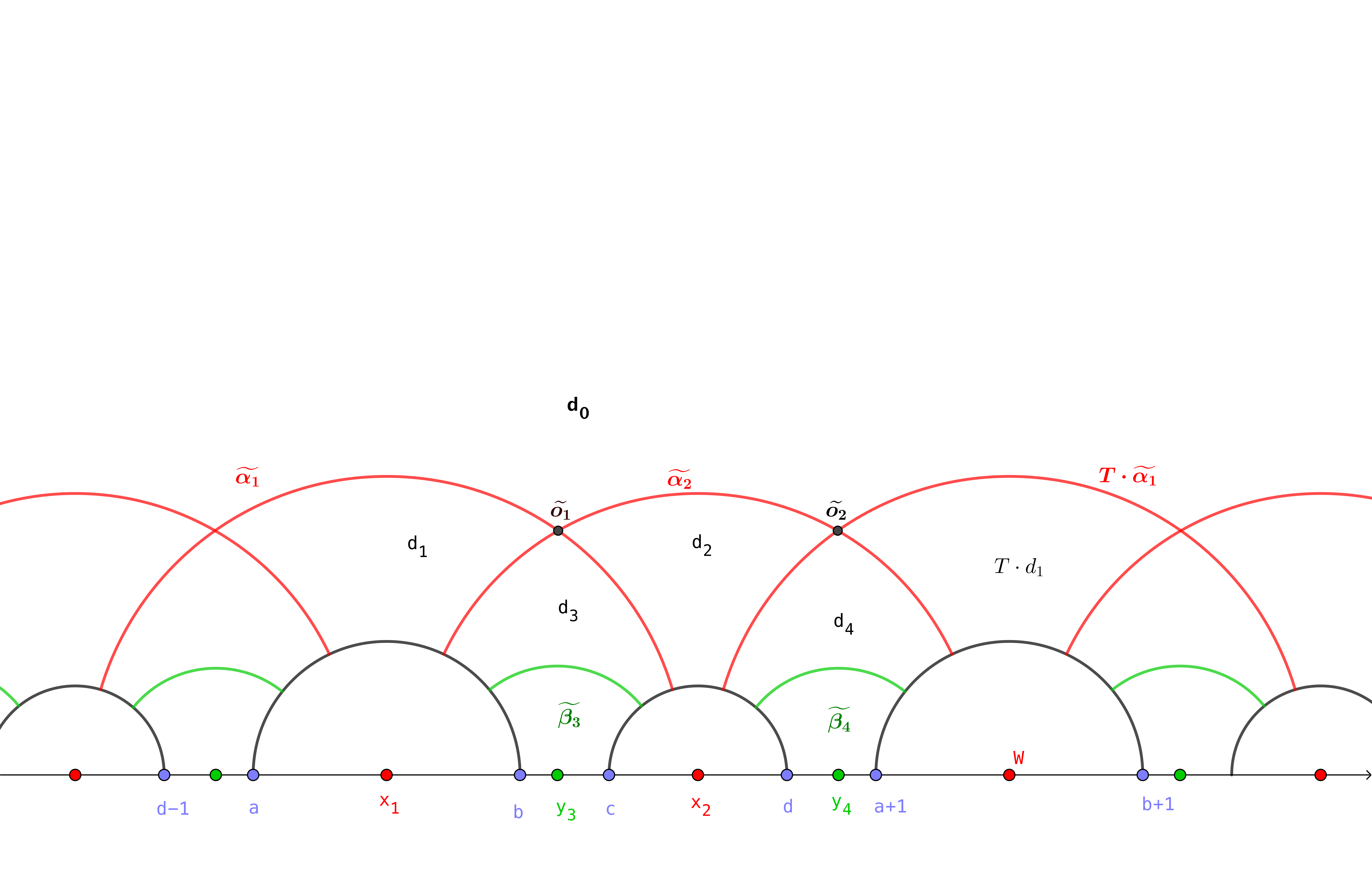}}
		\caption{$\al_1,\al_2$ are maximal}
		\label{mini}
	\end{center}
\end{figure}
Applying Lemma \eqref{ineqcentres} to the two triples of geodesics $(c-1,d-1),(a,b),(c,d)$ and \\$(a,b),(c,d),(a+1,b+1)$, we get that $x_1<y_2<x_2<y_3$. As a result, $a_2<0$. 				

\begin{proof}[Verification of Properties \eqref{proppunc}:]~\\
	\vspace{-1em}
	\begin{enumerate}
		\item The equivariance follows from the definition of $\phi_0$.
		
		\item Suppose that $\del$ is a tile with an ideal vertex. If $\del\in \supp{\phi_0} $, then $\del=d_j$ for some $j\in \set{3,4}$, then that ideal vertex is given by $y_j$. From the equations \eqref{One} and \eqref{Two}, we get that $\np \del$ vanishes at $y_j$.  If $\del\notin \supp{\phi_0} $, then $\np \del=0$, which automatically fixes its ideal vertex.
		
		\item We show that the map is coherent around $\vo$:
		\begin{align*}
			(\np {d_0}-\np{d_1}) (z)&=-a_1z+a_1x_1\\
			&=a_2z-a_2z-a_1z+a_1x_1+a_2x_2-a_2x_2\\
			&=a_2(z-x_2)-(a_1+a_2)(z)-a_1x_1-a_2x_3\\
			&=	(\np {d_2}-\np{d_3}) (z).
		\end{align*}
		Around the point $\vt$:
		\begin{align*}
			(\np {d_0}-\np{d_2}) (z)&=-a_2z+a_2x_2\\
			&=a_1z-a_1z-a_1z+a_2x_2+a_1x_1-a_1x_1+a_1-a_1\\
			&=a_1(z-x_1-1)-(a_1+a_2)(z)-a_1x_1-a_2x_2-a_1\\
			&=	(\np {T\cdot d_1}-\np{d_4}) (z).
		\end{align*}
		Hence by equivariance, the map $\phi_0$ is coherent around every intersection point.
		\item Let $\del,\del'\in \lrtile$ be two tiles with a common internal edge of the form $T^m \cdot \wt{\al_i}$, for some $m\in \Z$ and $i=1,2$, such that $\del$ lies above the edge. Using the equivariance and coherence of the map, it suffices to verify the property when $m=0, \del=d_0,\del'=d_i$, $i=1,2$.

			From the calculations of the proof of the coherence property, we have that 
		\begin{align*}
			(\np {d_0}-\np{d_1}) (z)&=-a_1(z-x_1)\\
			(\np {d_0}-\np{d_2}) (z)&=-a_2(z-x_2).
		\end{align*}
		Since $a_1,a_2<0$, the difference is of the desired form for every $i=1,2$.

		\item If the two tiles $\del,\del'$ have a common internal edge of the form $ \wt{\be_j}$ for $m\in \Z$ and $j=3,4$ such that $\del$ lies above the edge, then $\del=\wt{\be_j}$ and $\del'\notin \supp {\phi_0}$. 
		\begin{itemize}
		\item For $j=3$, $(\np {d_3}-\np{\del'})(z)=(a_1+a_2)z-a_1x_1-a_2x_2.$ 
		\item 		For $j=4$, $(\np {d_4}-\np{\del'})(z)=(a_1+a_2)z-a_1x_1-a_2x_2-a_1.$
		\end{itemize}
			Since $a_1,a_2<0$, both of these polynomials have negative leading coefficient $a_1+a_2$. By eqs. \eqref{One}, \eqref{Two}, the polynomial $\np {d_j}-\np{\del'}$ vanishes at $y_j$, for $m\in \Z$ and $j=3,4$. 
		
		If two tiles $d_1,d_2\in \lrtile$ have a common internal edge $\be\in \ledo\cap \ledt$, which is not of the above form, then $d_1,d_2\notin \supp{\phi_0}$. So, $\np {d_1}-\np{d_2}=0$.\qedhere
	\end{enumerate}
\end{proof}
So $\phi_0$ is a $\Gamma$-equivariant refined tile map that realises the required linear combination.
\end{enumerate}
	\subsubsection{Decorated ideal polygons}\label{decodim1}
In this section, we will prove Theorem \ref{codim1} fo decorated polygons without punctures.

The surface is contractible. So, $\lred=\red,\lrtile=\rtile$. 
Firstly, we remark that at most one of the two intersecting arcs $\al_1,\al_2$ can be of the edge-to-vertex type. Indeed, if for every $i=1,2$ the arc $\al_i$ joins the vertex $v_i$ with the edge  $e_i$ then these four vertices must be cyclically ordered as $e_1,v_2,v_1,e_2$ and no two are consecutive. In particular, $v_1,v_2$ are not consecutive. So either there exists an arc $\be$ in $\edo$ that has one endpoint on $e_1$ and another on a vertex or an edge that lies between $v_2$ and $v_1$ or there exists an arc $\be'\in \edo\cap \edt$ joining $v_1,v_2$. In the first case, the arc $\al_2$ must intersect $\be$, hence $\codim {\sigma_1\cap \sigma_2}>1$ which contradicts our hypothesis. The second case is not possible because there are no vertex-to-vertex arcs in this arc complex. 

So we have the following two cases:
\begin{itemize}
	\item The proof, in the case where both $\al_1$ and $\al_2$ are edge-to-edge arcs, is identical to that for ideal polygons.
	\item Let $\al_1$ be an edge-to-edge arc joining two edges $e_1, e_3$ and let $\al_2$ be an edge-to-vertex arc joining the edge $e_2$ and the decorated ideal vertex $\nu$. As shown in the Fig.\ref{deco4possi}, there are four configurations  
	possible depending on whether $e_1$ or $e_3$ are consecutive to $p$ or not. Fig. \ref{codim1deco} focuses on the generic possibility.
	\begin{figure}[!h]
		\begin{center}
			\includegraphics[width=12cm]{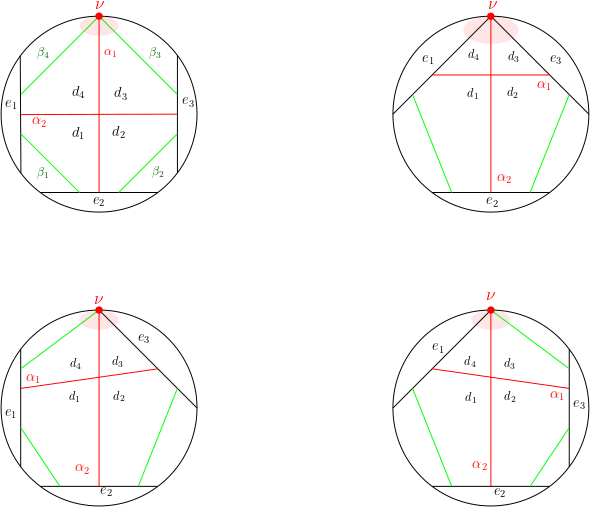}
			\caption{The four possible configurations}
			\label{deco4possi}
		\end{center}
	\end{figure}
\begin{figure}[!h]
	\begin{center}
		\frame{\includegraphics[width=12cm]{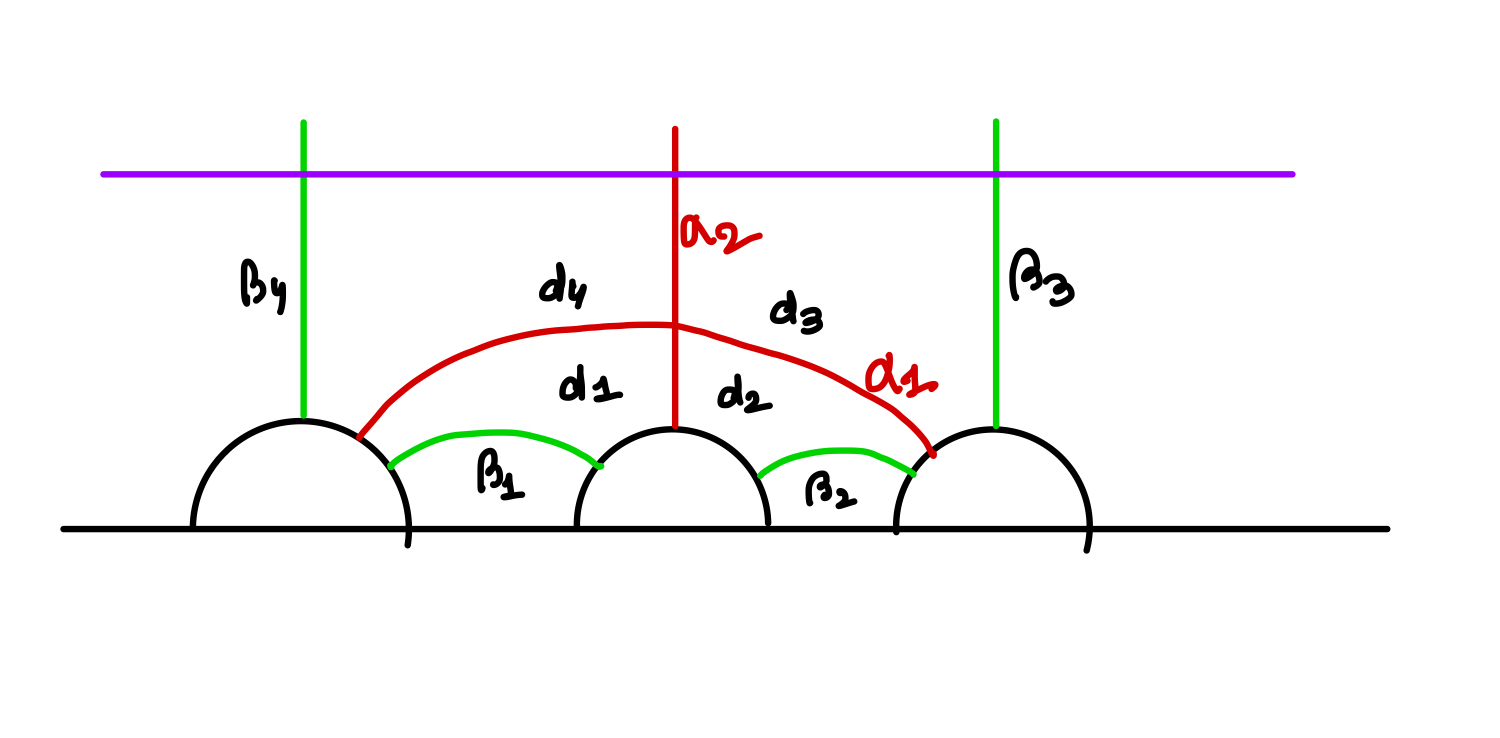}}
		\caption{Codimension 1: the generic case}
		\label{codim1deco}
	\end{center}
\end{figure}
	Since neither $e_1,e_2$ nor $e_2,e_3$ can be consecutive, there always exist two edge-to-edge arcs $\be_1$ and $\be_2$ in $\edo\cap\edt$ that respectively join these two pairs.
	Again, if $e_1,\nu$ or $e_3,\nu$ are not consecutive, there must exist two edge-to-vertex arcs $\be_4$ and $\be_3$ in $\edo\cap\edt$, joining the pairs, respectively. 
	Let $d_1,\ldots,d_4$ be the smaller tiles of the refinement $\sigma_1\cup \{\al_2\}$, such that $\be_i$ is an internal edge of $d_i$ whenever $\be_i$ exists. 
	
	We may suppose that $\nu=\infty$. We make the following choice of strip template:
	\begin{itemize}
		\item The arcs are chosen from the isotopy classes so that they intersect the boundary perpendicularly.
		\item For edge-to-edge arcs, the waists are chosen to the point of projection of the ideal point $\infty$. For edge-to-vertex arcs, the waist is always the point $\infty$.
	\end{itemize}
	
	The arcs $\be_1,\be_2$ are semi-circular with centres $y_1,y_2$ whereas $\be_3,\be_4$ are vertical lines. Let $x_0$ be the centre of the semi-circle carrying $\al_1$. Using Lemma \eqref{centre}, we have that \begin{equation}\label{ineqcentredeco}
		y_1<x_0<y_2.
	\end{equation}
	We shall construct a tile map corresponding to the linear combination \eqref{lincomid}.
	
	\begin{proper}\label{propdeco}
		A neutral tile map $\phi_0:\rtile\longrightarrow\R_2[z]$ represents the linear combination \eqref{lincomid} if and only if it verifies the following properties:
		\begin{enumerate}
			\item \label{ideconeut}The polynomial $\np d$ vanishes at every ideal vertex of $\del\in \rtile$.
			\item \label{idecoch}The tile map is coherent around the intersection point $o$.
			\item Let $\del,\del'\in \lrtile$ be two tiles with common internal edge contained in $\al_1$ for such that $\del$ lies above $\al_1$. Then $\np \del-\np {\del'}$ is a hyperbolic Killing field with attracting fixed point at $\infty$ and repelling fixed point at $x_1$. In particular, its axis intersects $\al_1$ at $p_{\al_1}$. In terms of polynomials, we must have $\np \del-\np {\del'}=(z\mapsto A(z-x_1)),\text{ for some }A>0,$ a linear polynomial with positive leading coefficient.
			\item \label{idecoalhori} Let $\del,\del'\in \rtile$ be two tiles with common internal edge contained in $\al_2$ for $i=1,2$ such that $\del$ lies to the left of the edge. Then $\np \del-\np {\del'}$ is a parabolic Killing vector field with fixed point at $\infty$, pointing towards $\del$. In terms of polynomials, we must have $\np \del-\np {\del'}=(z\mapsto B),\text{ for some }B<0,$ a constant polynomial.
			
			\item \label{idecobehori}Suppose $\del,\del'\in \rtile$ are two tiles with common internal edge $\be_j$ for $j=1,2$, such that $\del$ lies above $\be_j$. Then $\np \del-\np {\del'}$ is a hyperbolic Killing vector field with attracting fixed point at $y_j$ and repelling fixed point at $\infty$. In particular, its axis intersects $\be_j$ at $p_{\be_j}$. In terms of polynomials, we must have $\np \del-\np {\del'}=(z\mapsto C(z-y_j)), \,C<0.$
			\item\label{idecobevert} Suppose $\del,\del'\in \rtile$ are two tiles with common internal edge $\be_j$ for $j=3,4$, such that $\del$ lies to the left $\be_j$. Then $\np \del-\np {\del'}$ is a parabolic Killing vector field with fixed point at $\infty$, pointing away from $\del$. In terms of polynomials, we must have $\np \del-\np {\del'}=(z\mapsto D),\, D>0.$
		\end{enumerate}
	\end{proper}
	We define the tile map
	\begin{eqnarray*} 
		\phi_0&:\lrtile\longrightarrow &\R_2[z]\\
		&\del \mapsto &\left\{ 
		\begin{array}{lll}
			(z\mapsto a_j(z-y_j)), & \text{if } \del=d_j, &i=1,2,\\
			(z\mapsto a_3), & \text{if } \del=d_3\\
			(z\mapsto a_4), & \text{if } \del=d_4\\
			0, & \text{otherwise},
		\end{array}
		\right.
	\end{eqnarray*}
	where
		$a_1=a_2=-1, \quad a_3=y_2-x_0>0,\quad a_4=y_1-x_0<0.$

	\begin{lem}
		The tile map defined above satisfies Properties \eqref{propdeco}.
	\end{lem}
	\begin{proof}
		\begin{enumerate}
			\item Suppose that $\del$ is a tile with a decorated ideal vertex. If $\del\in \supp{\phi_0} $, then $\del=d_j$  for some $j\in \set{3,4}$ and that ideal vertex is given by $\infty$. From the definition of $\phi_0$, we get that $\np {d_j}$ is a constant polynomial. Hence, it fixes infinity and any horoball centered at $\infty$. If $\del\notin \supp{\phi_0} $, then $\np \del=0$, which automatically fixes its ideal vertex.
			\item Consistency around $\wt o$:
			\begin{align*}
				\np{d_1}-\np{d_2} (z)&=y_1-y_2\\
				&= y_1-x_0+x_0-y_2\\
				&=	\np{d_4}-\np{d_3} (z).
			\end{align*}
			\item Let $\del,\del'\in \rtile$ be two tiles with a common internal edge of the form $\al_1$ such that $\del$ lies above the edge. Using the coherence of $\phi_0$, it suffices to verify the property when $\del=d_4,\del'=d_1$. Substituting the values of $a_1,a_4$ and using the definition of $\phi_0$, we get that
		$\np{d_4}-\np{d_1} (z)=z-x_0,$
			which is of the desired form.
			\item Let $\del,\del'\in \rtile$ be two tiles with a common internal edge of the form $\al_2$ such that $\del$ lies to the left of the edge. Using the coherence of $\phi_0$, it suffices to verify the property when $\del=d_4,\del'=d_3$. From eq. \eqref{ineqcentredeco},
		$\np{d_4}-\np{d_3} (z)=y_1-y_2<0.$
		So the property is verified.
			\item Let $\del,\del'\in \rtile$ be two tiles with a common internal edge of the form $\be_j$ for $j=1,2$ such that $\del$ lies above the edge. Then $\del=d_j$ for $j=1,2$ and $\del'\notin \supp {\phi_0}$.
			For $j=1,2$, we have that
				$\np{d_j}-\np{\del'} (z)=-z+y_j,$
			which is of the desired form.
			\item Let $\del,\del'\in \rtile$ be two tiles with a common internal edge of the form $\be_j$ for $j=3,4$ such that $\del$ lies to the left of the edge. Then either $\del\notin \supp {\phi_0}, \del'=d_4$ or $\del=d_3, \del'\notin \supp {\phi_0}$.
			In the first case, we have that
				$\np{\del}-\np{d_4} (z)=x_0-y_1>0.$
			In the second case, we have that 
				$\np{d_3}-\np{\del'} (z)=y_2-x_0>0.$
		\end{enumerate}
		So we have a neutral tile map representing the linear combination \eqref{lincomid} for the chosen strip template. 
	\end{proof}
	
\end{itemize}
This concludes the proof of Theorem \ref{codim1} for decorated ideal polygons.

\subsubsection{Decorated once-punctured polygons}\label{decodim1punc}
Finally, we will prove Theorem \ref{codim1} for decorated once-punctured polygons.
	
	Let $\sigma_1$ and $\sigma_2$ be as in the hypothesis with $\al_1\in\edo\backslash \edt$ and $\al_2\in\edt\backslash\edo$. These two arcs intersect either exactly once at a point (when both are non-maximal) or twice (when both are maximal). In the first case: at most one of the two intersecting arcs $\al_1,\al_2$ can be of the edge-to-vertex type.
	
	\begin{enumerate}
		\item Suppose that both the geodesic arcs $\al_1$ and $\al_2$ are finite.  We choose  an embedding of the universal cover of the polygon in the upper half plane so that the point $\infty$ is distinct from the endpoints of all the arcs in the lifted edgesets $\ledo, \ledt$. Then every geodesic arc used in the triangulation is carried by a semi-circle. 
	Let $o$ be the point of intersection of $\al_1, \al_2$.  Let $\wt o$ be a lift of the point $o$. Then $\wt o= \alo\cap \alt$ for two lifts  of $\al_1$ and $\al_2$ respectively. There are four tiles formed around $\wt o$, denoted by $d_j$, for $j=1,\ldots, 4$. For each $j\in \set{1,\ldots,4}$, the tile $d_{j}$ is either a quadrilateral with a decorated vertex and exactly two arc edges carried by $\alo$ and $\alt$, or it is a pentagon with exactly three arc edges carried by $\alo,\alt$ and a third arc $\wt{\be_j}$, which is a lift of an arc $\be_j\in \edo\cap\edt$.  Let $\mathcal J\subset \{1,\ldots, 4\}$ be such that the tile $d_j$ is a pentagon if and only if $j\in \mathcal J$ .
	For $i=1,2$, let $x_i$ be the centre of the semi-circle containing $\wt{\al_i}$. For $j=1,\ldots,4$, let $y_j$ denote the ideal vertex of $d_j$ or the centre of the semi-circle containing $\wt{\be_j}$.
	For each $j$, the tile $d_j$ is either a quadrilateral with exactly one ideal vertex and two internal edges contained in $\al_1$ and $\al_2$, or it is a pentagon with exactly three internal edges: $\al_1,\al_2$ and a third arc $\be_j\in \edt\cap\edo$. 
	We shall construct a tile map  corresponding to the following linear combination:
	\begin{equation}\label{lincomid}
		c_{\al_1}f_{\al_1}(m)+c_{\al_2}f_{\al_2}(m)+\sum_{j\in \mathcal J}c_{\be_j} f_{\be}(m)=0,
	\end{equation}
	with $c_{\al_1},c_{\al_2}>0$ and $c_{\be_j}< 0$ for every $j\in \mathcal J$. In other words, we define a neutral tile map $\phi_0:\lrtile\longrightarrow\R_2[z]$ that is $\Gamma$-equivariant and satisfies Properties \eqref{idneutral}-\eqref{idbe}, as defined in Subsection \ref{ipdim1}. 
	
	Suppose that the endpoints of $\wt{\al_1}$ lie on the boundary geodesics $(a,b)$ and $(e,f)$ and those of $\wt{\al_2}$ lie on $(c,d)$ and $(g,h)$ such that the following inequalities hold: 
	\begin{align}
		&a<b\leq c<d\leq e<f\leq g<h.
	\end{align}
	For $j=1,\ldots,4$, define 
	\begin{eqnarray*}
		\phi_0&:\lrtile\longrightarrow &\twop\\
		&d \longmapsto &\left\{ 
		\begin{array}{ll}
			P_j:=(z\mapsto a_j(z-y_j)), & \text{if } d=d_j\\
			(\ga\cdot z\mapsto \frac{\mathrm{d}\ga(z)}{\mathrm{d}z}P_j(z))), & \text{if } d=\ga\cdot d_j, \ga\in \Gamma\smallsetminus\{Id\}\\
			0, & \text{otherwise,}
		\end{array}
		\right.
	\end{eqnarray*}
	where \[\begin{array}{llll}
		a_1=\frac{x_1-y_4}{x_1-y_1},&a_2=\frac{(x_1-y_4)(x_2-y_4)-(y_4-y_1)(y_4-y_3)}{(x_1-y_1)(x_2-y_3)}&a_3=\frac{x_2-y_4}{x_2-y_3} &a_4=1.
	\end{array} \]
	So, $a_1, a_2,a_3<0$.
	
	The tile map $\phi_0$ is $\Gamma$-equivariant by definition. From the subsection we get that it verifies all the properties except possibly for a sub-case of \eqref{idbe} when there exist $j,j'\in \mathcal J$ and $\ga\in \Gamma\smallsetminus\{Id\}$ such that $\wt{\be_j'}=\ga\cdot\wt{\be_j}$. Without loss of generality, we suppose that $j<j'$. We shall now prove that $\phi_0$ verifies this property for this case as well. 
	
	The geodesic arc $\ga\cdot \be_j$ is the common internal edge of $d_{j'}$ and $\ga\cdot d_j$. Suppose that $d_{j'}$ lies above $\ga\cdot d_j$. Since $j<j'$, we have that $j,j'\in \{1,2,3\}$. From the definition of $\phi_0$, we get that $\np{d_j},\np{d_{j'}}$ are hyperbolic Killing fields whose axes are perpendicular respectively to $ \be_j,\be_{j'}$, with attracting fixed points at $y_j,y_{j'}$. Since both $d_j$ and $d_{j'}$ lie above $\be_j$ and $\be_{j'}$, this means that the two Killing fields are directed towards these edges. Now $\ga$ maps $\be_j$ to $\be_{j'}$ and $\ga\cdot d_j$ lies below it. Since $\phi_0$ is $\Gamma$-equivariant, $\np{\ga\cdot d_j}$ is a hyperbolic Killing field directed towards $\ga\cdot\be_{j}$ and hence it points upwards. So the difference vector $\np{d_{j'}}-\np{\ga\cdot d_j}$ is hyperbolic, normal to $\be_{j'}$, directed downwards and towards $\ga\cdot d_j$, as required.  
	\begin{figure}
		\centering
		\frame{\includegraphics[width=15cm]{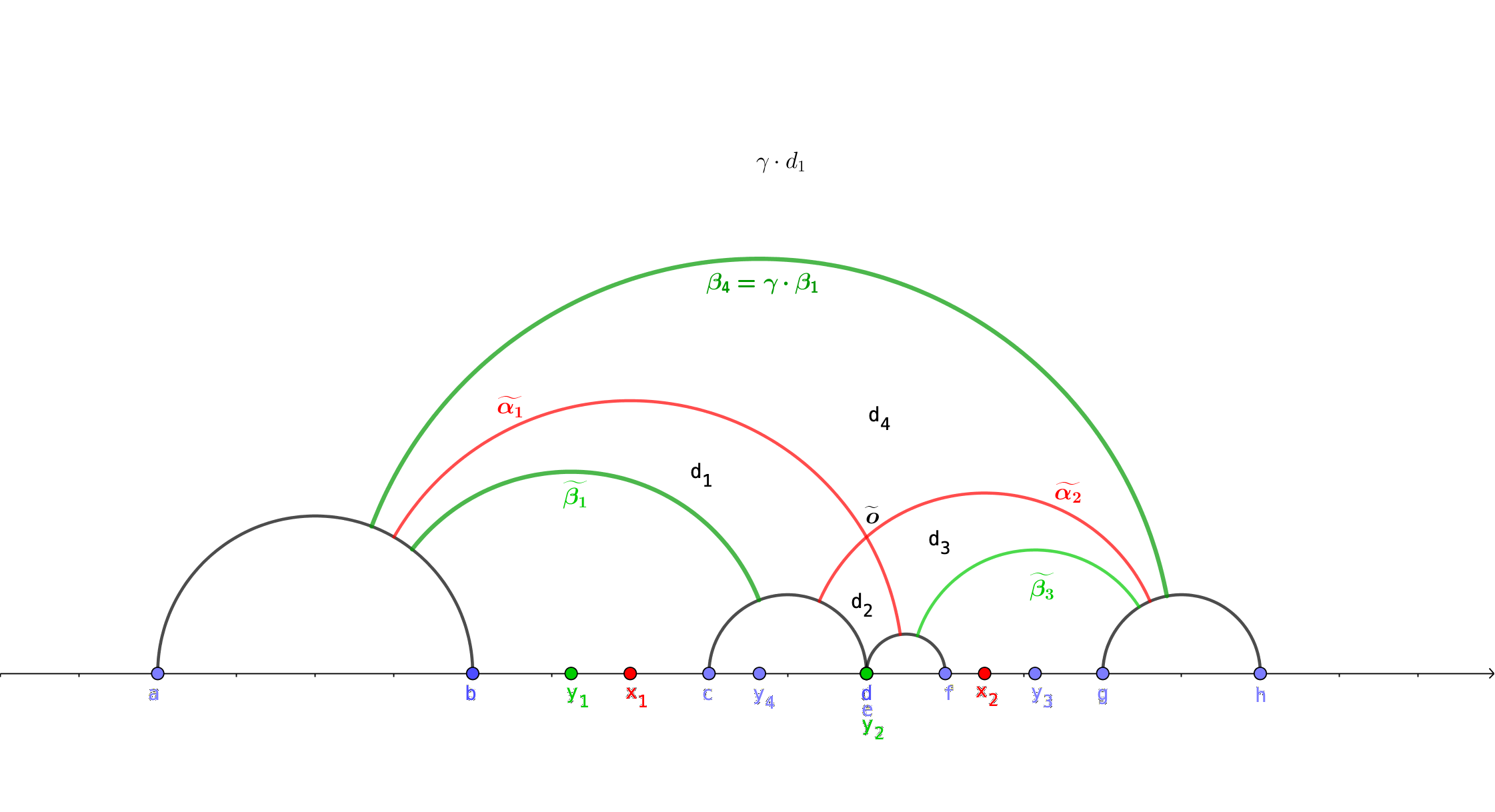}}
		\caption{$d=e=y_2$ and $j=1, j'=4$}
		\label{cod1sp}
	\end{figure}
	Finally, we suppose that $d_{j'}$ lies below $\ga\cdot d_j$. So $j'=4$ and $j\in \{1,2,3\}$. See Figure \ref{cod1sp}.  Again from the definition of $\phi_0$ we get that $\np {d_4}$ is a hyperbolic Killing field with axis perpendicular to $\be_4$ and attracting fixed point at $\infty$. So it is directed upwards. In this case, the tile $\ga\cdot d_j$ lies above $\be_4$, the Killing field $\np{\ga\cdot d_j}$ is directed downwards. Hence the difference vector field 
	$\np{d_{j'}}-\np{\ga\cdot d_j}$ is hyperbolic, directed upwards and towards $\ga\cdot d_j$, as required.  This finishes the proof of the theorem in this case.
\item Suppose that one of the two arcs, say $\al_1$, is finite, while the other one, $\al_2$, is an infinite vertex-to-edge arc. 
\begin{figure}
	\centering
	\frame{	\includegraphics[width=12cm]{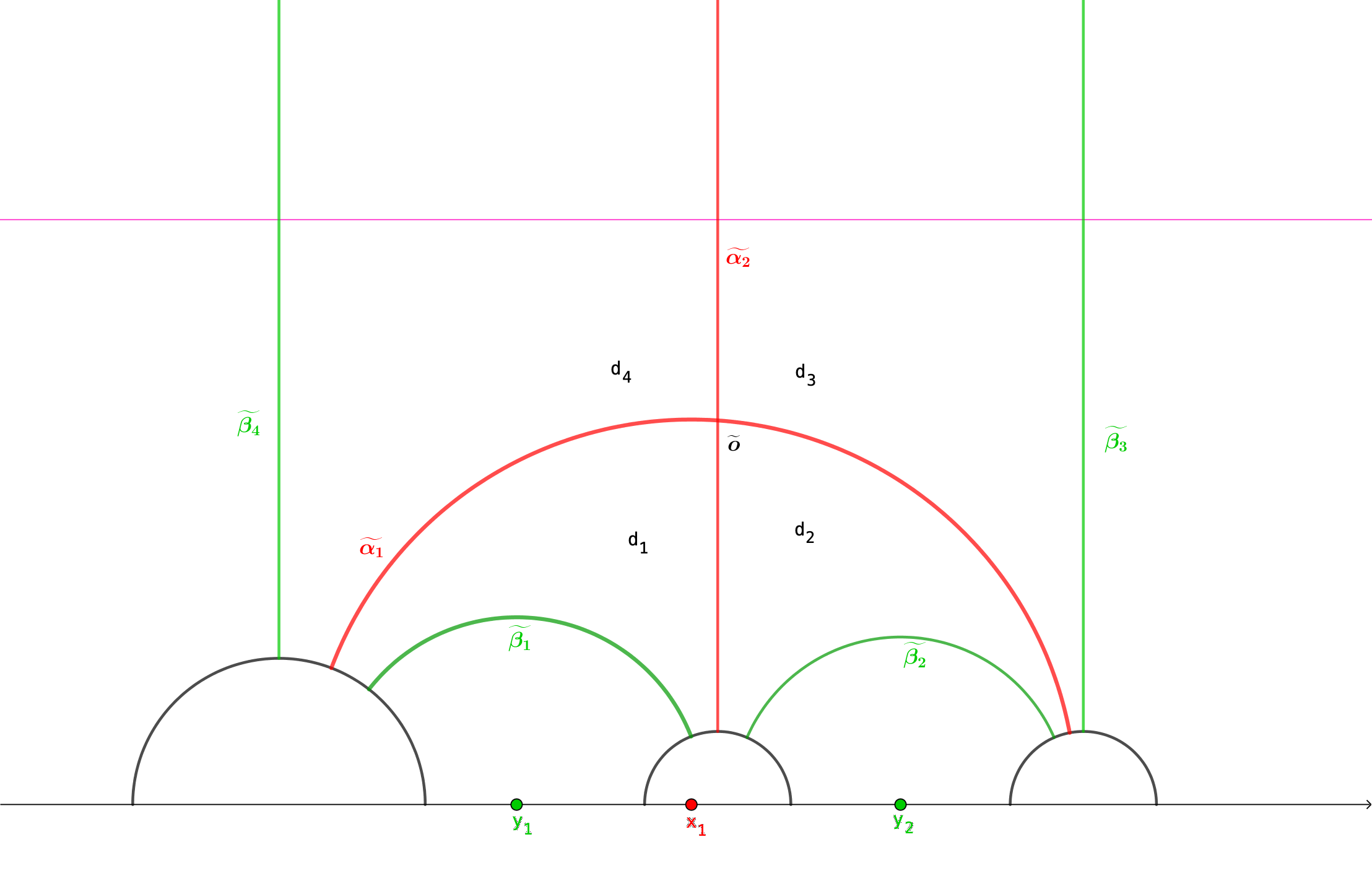}}
	\caption{}
	\label{cod1deco1}
\end{figure}
	We choose an embedding of the universal cover of the polygon in the upper half plane such that one of the lifts of the decorated vertex to which $\al_2$ converges, is at $\infty$. Let $\alt$ be the lift of $\al_2$ that passes through $\infty$ and let $\alo$ be the lift of $\al_1$ that intersects $\alt$ at a point $\wt o$, which is a lift of $o$ in $\HP$. See Figures \ref{cod1deco1}, \ref{cod1deco2}. The four tiles formed around $\wt o$ are labeled $d_1,\ldots,d_4$ in the trigonometric sense such that the tiles $d_3,d_4$ lie above $\alo$, while $d_1,d_2$ lie below $\alo$. The horizontal line is the horoball decoration at $\infty$. For $j\in \set{1,2}$, the tile $d_{j}$ is a pentagon with exactly three arc edges carried by $\alo,\alt$ and a third arc $\wt{\be_j}$, which is a lift of an arc $\be_j\in \edo\cap\edt$. For $j\in \set{3,4}$, the tile $d_{j}$ is either 
	\begin{itemize}
		\item a quadrilateral with a decorated ideal vertex and exactly three internal edges carried by $\alo,\alt$ and a third arc $\wt{\be_j}$, which is a lift of an arc $\be_j\in \edo\cap\edt$ (Fig.\ \ref{cod1deco1}), 
		\item or a triangle with exactly two arc edges carried by $\alo,\alt$ (Fig.\ \ref{cod1deco2}).
	\end{itemize}
	For $j=1,2$, let $y_j$ be the centre of the semi-circle carrying $\wt{\be_j}$. Let $x_0$ be the centre of the semi-circle carrying $\alo$.
	We define the tile map
	\begin{eqnarray*} 
		\phi_0&:\lrtile\longrightarrow &\R_2[z]\\
		&d \mapsto &\left\{ 
		\begin{array}{lll}
			P_j:=(z\mapsto a_i(z-y_i)), & \text{if } d=d _j, &i=1,2,\\
			P_3:=(z\mapsto a_3), & \text{if } d=d_3\\
			P_4:=(z\mapsto a_4), & \text{if } d=d_4\\
			(\ga(z)\mapsto 	\frac{\mathrm{d}\ga(z)}{\mathrm{d}z}P_j(z))), & \text{if } d=\ga\cdot d_j, \ga\in \Gamma\smallsetminus\{Id\}\\
			0, & \text{otherwise},
		\end{array}
		\right.
	\end{eqnarray*}
	where $a_1=a_2=-1, \quad a_3=y_2-x_0>0,\quad a_4=y_1-x_0<0.$
	\begin{figure}
		\centering
		\frame{\includegraphics[height=10cm]{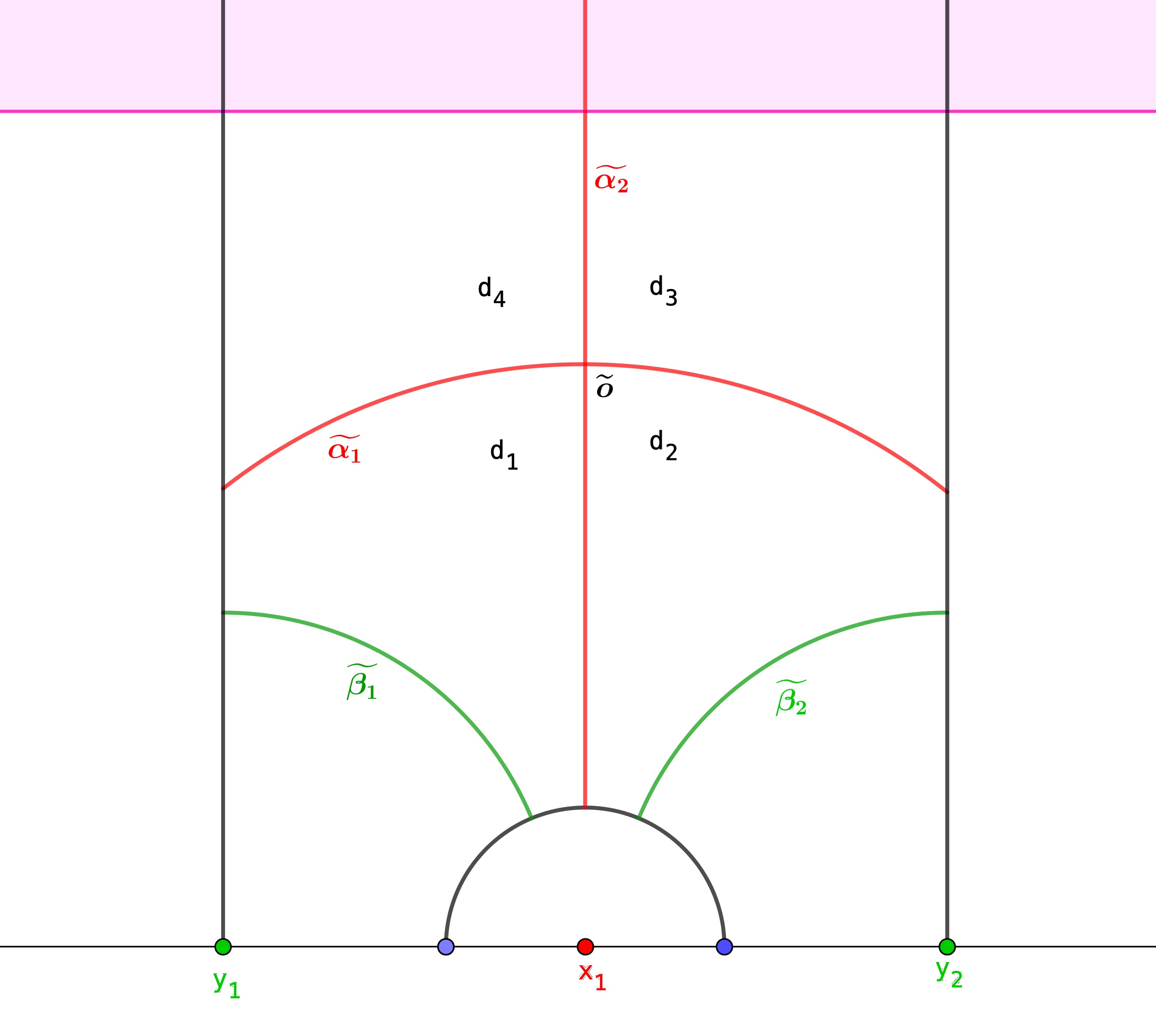}}
		\caption{${\wt\be_2}=\ga\cdot\wt{\be_1}$}
		\label{cod1deco2}
	\end{figure}
	This is a $\Gamma$-equivariant map by construction and satisfies the properties \ref{propdeco}\eqref{ideconeut}-\eqref{idecoalhori}, as defined in proof in Subsection \ref{decodim1}. 
	We shall now verify the sub-cases of \eqref{idecobehori} and \eqref{idecobevert} when there exist $j,j'\in \mathcal J$ and $\ga\in \Gamma\smallsetminus\{Id\}$ such that $\wt{\be_j'}=\ga\cdot\wt{\be_j}$.
	The only way to identify $\wt{\be_3}$ and $\wt{\be_4}$ is via a parabolic element in $\psl$ that has $\infty$ as fixed point. But such elements are not present in $\Gamma$. Hence, the verification of Property \ref{propdeco}\eqref{idecobevert}, as done in the proof of Theorem \ref{decodim1}, suffices for our general case as well. Thus we are left with the sub-case $\ga\cdot\be_1=\be_2$. Note that for this to happen, the endpoints of $\wt{\be_1}$ must lie on two geodesics that do not intersect in $\cHP$. Fig.\ \ref{cod1deco3} shows one such example.
	\begin{figure}
		\centering
		\frame{	\includegraphics[height=8cm]{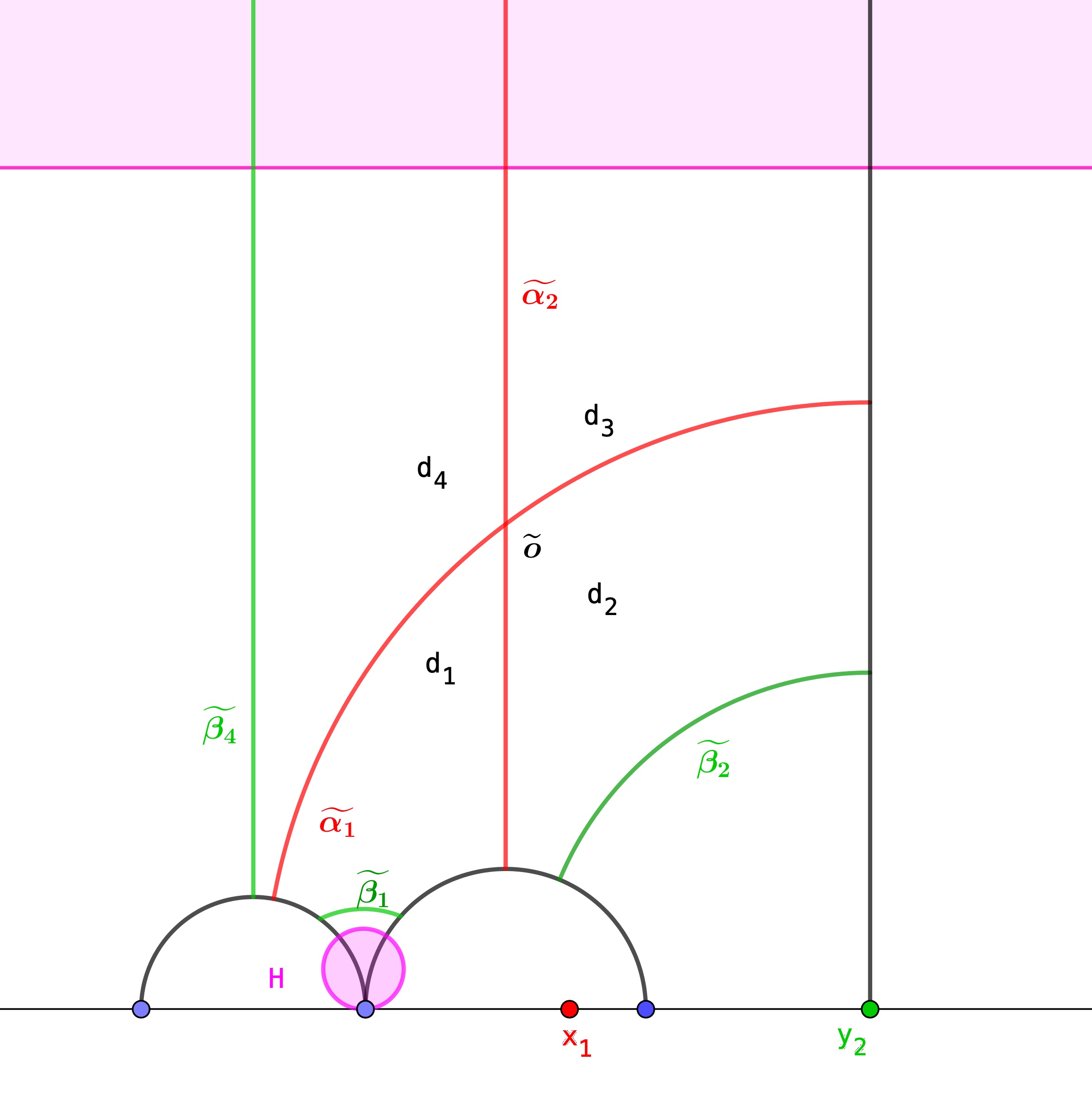}}
		\caption{}
		\label{cod1deco3}
	\end{figure}
	From the definition of $\phi_0$, we get that $\np{d_1},\np{d_{2}}$ are hyperbolic Killing fields whose axes are perpendicular respectively to $ \wt{\be_1},\wt{\be_2}$, with attracting fixed points at $y_1,y_2$. Since both $d_1$ and $d_2$ lie above $\be_1$ and $\be_2$, this means that the two Killing fields are directed towards these edges. Now $\ga$ maps $\be_1$ to $\be_2$ and $\ga\cdot d_1$ lies below $\be_2$. Since $\phi_0$ is $\Gamma$-equivariant, $\np{\ga\cdot d_1}$ is a hyperbolic Killing field directed towards $\ga\cdot\be_{1}$ and hence it points upwards. So the difference vector $\np{d_{2}}-\np{\ga\cdot d_1}$ is hyperbolic, directed downwards and towards $\ga\cdot d_1$, as required. This finishes the proof in this case.
\item The proof in the case of two intersecting maximal arcs is identical to that for undecorated punctured polygons. The tile containing the puncture in the complement of $\sigma_1\cap\sigma_2$ is a once-punctured quadrilateral bounded by a pair of opposite boundary edges and a pair of opposite edge-to-edge arcs. 
\end{enumerate}
	
\subsection{Local homeomorphsim: Codimension $\geq$ 2}
Let $p\in\ac \Pi$ such that $\codim{\sigma_p} \geq 2$.  In the cases of ideal polygons, punctured polygons, the arc complex is a sphere, we have that $$\Link {\sac {\Pi}}{\sigma_p} \simeq \mathbb{S}^{\codim{\sigma_p} -1.}$$ This is also true for decorated polygons because we have shown that their pruned arc complexes are open manifolds. In order to prove that $\mathbb{P}f$ is a local homeomorphism, it suffices to show that its restriction to the link of $\sigma_p$ is a homeomorphism. 
\begin{thm}\label{codimtwo}
	Let $\Pi$ be a ideal (posiibly decorated and once-punctured) polygon. Let $p\in\sac \Pi$ such that $\codim{\sigma_p}= 2$. Then, $\mathbb{P}f|_{\Link{\sac {\Pi}}{\sigma_p}}$ is a homeomorphism. 
\end{thm} 
\begin{proof}
	We shall prove the theorem separately for the different types polygons.
	\begin{itemize}
		\item Ideal $n$-gons $\Pi=\ip n$, for $n\geq 6$: The complex $\Link{\sac {\ip n}}{\sigma_p}$ is either a quadrilateral or a pentagon. So it is enough to show that the continuous map $\mathbb{P}f|_{\Link {\sac {\ip n}}{\sigma_p}}:\s 1 \longrightarrow \s 1$ has degree one.
		
		Suppose that the link is a pentagon. Let $\{\al_i\}_{i=1}^5$ be the five vertices of $\ac {\ip 5}$. Let $\theta_i$, $i\in \Z_5$ be the angle formed at the origin by the vectors $f(\al_{i+1})$ and $f(\al_i)$, in $\tei{\ip 5}$. Then, for every $i=1,\ldots,5$, we have $\theta_i\in (0,\pi)$.
		By theorem \ref{codim1}, we know that there is a choice of strip template such that $\theta_1+\theta_{2}<\pi$. Also, the sum $\sum_{i=3}^{5} \theta_i <3\pi$. Since, $f$ is a continuous map from the circle to itself, the quantity $\sum_{i=1}^{5} \theta_i $ is always a multiple of $2\pi$. Hence, we have $\sum_{i=1}^{5} \theta_i =2\pi$, which implies that the degree of $f$ is 1. Hence, $f$ is a homeomorphism for this choice of strip template. Since, the space of all strip templates is connected and since there is no continuous way of changing the sum of angles from $2\pi$ to $4\pi$, we have that $f$ is a homeomorphism for every choice of strip template. This also proves the homeomorphism of the projectivised strip map in the case of the ideal pentagon $\ip 5$.
		The proof works similarly when the link is a quadrilateral.
		\item Punctured $n$-gons $\Pi=\punc n$, for $n\geq 5$ or one-holed $n$-gons $\holed n$, for $n\geq3$ : Suppose that the complement $\Pi\backslash  \bigcup\limits_{\al\in\sigma^{(0)}} \al$ has one non-triangulated region. If this region contains the puncture, then it can be triangulated in six different ways using two disjoint arcs, exactly one of which is always a maximal arc. So we have that $\Link{\ac\Pi}{\sigma}$ is a hexagon. Like in the case of ideal polygons, for $i=1,\ldots,6$, let $\theta_i$ be the angle subtended at the origin by the vectors $f(\al_{i+1})$ and $f(\al_i)$, in $\tei{\Pi}$. Let $\al_1, \al_3,\al_5$ be the vertices corresponding to the maximal arcs.  By Theorem \ref{codim1}, there exists a strip template such that 
		\[ \theta_i+\theta_{i+1} < \pi, i=1,\ldots,5.\] So the degree of the map is 1.
		Since the arc complex of a punctured triangle $\punc 3$ are also $PL$-homeomorphic to a hexagon, the  homeomorphism of $\mathbb{P}f$ in these cases, is a consequence of the above proof.
		The two cases — exactly one non-triangulated region containing no puncture or hole, and two non-triangulated regions — are treated identically as in the case of ideal polygons.
		\item The proofs in the case of decorated polygons follows from the two above cases. 
	\end{itemize}
\end{proof}

\paragraph{Ideal Square, Punctured bigon, one-holed monogon:} When $\Pi=\ip 4$ or $\punc 2$, the arc complex $\ac {\Pi}$ is a sphere of dimension 0. Let $[\al]$ and $[\be]$ be the two isotopy classes of arcs. If we parametrise the deformation space using the length of $\al$, we see that $f(\al)$ corresponds to the origin where as $f(\be)$ increases its length. So, we have $\mathbb{P}f(\al)\neq \mathbb{P}f(\be)$, which proves the homeomorphism.

\begin{thm}
	Let $p\in\sac {\Pi}$ such that $\codim{\sigma_p}\geq 2$. Then, $\mathbb{P}f|_{\Link{\ac \Pi} {\sigma_p}}$ is a homeomorphism. 
\end{thm}
\begin{proof}
	The statement is verified for $\codim{\sigma_p}=2$. Suppose that the statement holds for $2,\ldots,d-1$. We need to show that $$\mathbb{P}f|_{\Link{\ac \Pi} {\sigma_p}}:\s{d-1}\longrightarrow \s{d-1}$$ is a local homeomorphism. Let $x\in \Link{\ac \Pi} {\sigma_p}$. Then $x$ is contained in the interior of a simplex $\sigma_x$ whose codimension in the link is $d-1-\dim \sigma_x$, which is less than $d$. So by the induction hypothesis, the map  $\mathbb{P}f|_{\Link{\ac \Pi} {\sigma_p}}$ restricted to $\Link{\mathbb{P}f|_{\Link{\ac \Pi} {\sigma_p}}}{\sigma_x}$ is a homeomorphism. This proves that $\mathbb{P}f|_{\Link{\ac \Pi} {\sigma_p}}$ is a local homeomorphism. Since $\s{d-1}$ is compact and simply-connected for $d\geq 3$, it follows that $\mathbb{P}f|_{\Link{\ac \Pi} {\sigma_p}}$ is a homeomorphism. 
\end{proof}

Thus, the map $\mathbb{P}f: \sac \Pi\longrightarrow \ptan \Pi$ is a local homeomorphism for the surfaces $\Pi=\ip n\, (n\geq 4),\, \punc n\, (n\geq 2),\, \dep n\, (n\geq 3)$ and $ \depu n\, (n\geq 3)$. By compactness  and the simply-connectedness of the sphere $\s{m}$ $m\geq 2$, we get that the map  $\mathbb{P}f$is a homeomorphism in the case of the undecorated polygons. 	
\subsection{Properness of the strip map}
In this final section we prove that the projectivised strip map $\mathbb{P}f$ is proper in the case of decorated polygons, which will conclude the proof of the homeomorphism o the same.
\begin{thm}\label{properdeco}
	Let $\dep n$ be a decorated ideal polygon with a metric $m\in \tei {\dep n}$. Then
	the projectivised strip map $\mathbb{P}f:\sac {\dep n} \longrightarrow \mathbb{P}^+(\adm m)$ is proper.
\end{thm}

\begin{proof}
Let $(x_n)_{n\in \N}$ be a sequence in the pruned arc complex $\sac {\dep n}$ such that $x_n\rightarrow \infty$: for every compact $K$ in $\sac {\dep n}$, there exists an integer $n_0\in \N$ such that for all $n\geq n_0$, $x_n\notin K$.  We want to show that $\mathbb{P}f(x_n) \to \infty$ in the projectivised admissible cone $\mathbb{P}^+\adm m$. Recall that the admissible cone $\adm m$ is an open convex subset of $\tang {\dep n}$. Its boundary $\partial \adm m$ contains $\vec{0}\in \tang S$ and is supported by hyperplanes (and their limits) given by the kernels of linear functionals $\mathrm{d}l_{\be}:\tang S \longrightarrow \R$, where $\be$ is an edge or a diagonal (a horoball connection) of the polygon. It suffices to show that $f(x_n)$ tends to infinity (in the sense of leaving every compact subset) inside $\adm m$ but stays bounded away from $\vec{0}$ so that $\mathbb{P}f(x_n)$ tends to infinity in $\mathbb{P}^+\adm m$. \\
Since the arc complex $\ac{\dep n}$ is a finite simplicial complex, there exists a subsequence $(y_n)_{n\in\N}$ that converges to a point $y\in \ac {\dep n}\smallsetminus\sac{\dep n}$. So $y_n$ is of the form: \[ y_n=\sum_{i=1}^{N}t_i(n)[\al_{i}], \text {with } t_i(n)\in (0,1] \text{ and }\sum_{i=1}^{N}t_i(n)=1, \] 
and the limit point $y$ is then given by: $y=\sum_{i=1}^N t_{i}^\infty[\al_{i}],$  where there exists $\mathcal{I}\subsetneq \{1,\ldots,N\}$ such that \[\text{ for } i\in \mathcal I,\,t_{i}(n)\mapsto t_{i}^\infty\in (0,1],\text{ and }\sum_{i\in \mathcal{I}}t_{i}^\infty=1,\] 
\[\text{ for } i\in \{1,\ldots,N\}\smallsetminus \mathcal{I},\, t_{i}(n)\to t_i^\infty=0.\] 
Since $y\in \ac {\dep n}\smallsetminus\sac{\dep n}$, in the complement of $\supp y=\bigcup_{i\in\mathcal{I}}\al_i$, there is a horoball connection, denoted by $\be$. By construction, $\be$ intersects only the arcs $\{\al_i\}_{i \notin \mathcal{I}}$. By continuity of the infinitesimal strip map $f$ on $\sigma$, the sequence $(f(y_n))_n$ converges to $f(y)\in \partial\adm m$ and $$\mathrm{d}l_{\be}(f(y))=\sum_{i \notin \mathcal{I}} t_i^\infty \mathrm{d}l_{\be}(f_{\al_i}(m))=0.$$ Hence $f(y)$ fails to lengthen $\be$. Also 
 $f(y)\neq 0$. This is becaue we have
 \begin{equation}
 	\mathrm{d}l_{\ga}(f(y))=\sum_{i\in \mathcal{I}} t_i^\infty \mathrm{d}l_{\be}(f_{\al_i}(m))>0,
 \end{equation}  
whenever $\ga$ is any horoball connection intersecting the arcs inside $\supp y$ (for e.g the edge containing an endpoint of an arc in the support). 
This concludes the proof.
\end{proof}
A very identical argument proves that 
\begin{thm}\label{properdecop}
	Let $\depu n$ be a decorated ideal polygon with a metric $m\in \tei {\depu n}$. Then
	the projectivised strip map $\mathbb{P}f:\sac {\depu n} \longrightarrow \mathbb{P}^+(\adm m)$ is proper.
\end{thm}
\bibliography{decoratedpolygons.bib}
\bibliographystyle{plain}
\end{document}